\documentclass[10pt]{amsart}
\usepackage[margin=1in]{geometry}
\usepackage{amsmath,amssymb,mathtools,mathrsfs}
\usepackage{xcolor}
\usepackage{tikz}
\usetikzlibrary{arrows.meta,backgrounds,fit,positioning}
\usepackage{placeins}
\usepackage[
  backend=biber,
  style=alphabetic,
  sorting=nyt,
  giveninits=true,
  maxbibnames=99
]{biblatex}
\usepackage{hyperref}

\theoremstyle{remark}
\newtheorem{remark}{Remark}
\theoremstyle{definition}
\newtheorem{definition}{Definition}[section]
\newtheorem{assumption}{Assumption}[section]
\theoremstyle{plain}
\newtheorem{theorem}[definition]{Theorem}
\newtheorem{proposition}[definition]{Proposition}
\newtheorem{lemma}[definition]{Lemma}

\allowdisplaybreaks
\numberwithin{equation}{section}
\renewcommand{\d}{\mathrm d}
\newcommand{\R}{\mathbb R}
\newcommand{\T}{\mathbb T}
\newcommand{\E}{\mathbb E}

\begin{document}

\title[Fluctuating Kinetic Theory]{\Large A Program for Fluctuating Kinetic Theory}
\author[Zhengyan Wu]{\Large Zhengyan Wu}
\address{Department of Mathematics, Technische Universit\"at M\"unchen,
Boltzmannstr. 3, 85748 Garching, Germany}
\email{wuzh@cit.tum.de}

\begin{abstract}
A program for fluctuating kinetic theory is proposed in this paper to study nonlinear L\'evy-type fluctuation
corrections to kinetic PDEs in a unified way.  The
Vlasov--Fokker--Planck (VFP), Landau, and Boltzmann equations can each be
decomposed into a kinetic transport term and a generalized gradient
structure generated by a nonlinear dual pair of dissipation potentials.
Going beyond the linear Onsager structure and the classical
fluctuation--dissipation relation, we propose a new nonlinear L\'evy-type
fluctuation--dissipation relation that identifies the thermal fluctuations
associated with this generalized gradient structure.  On a regular tilted
class, the canonical action of the generalized gradient structure agrees
with the Freidlin--Wentzell-type action predicted by the resulting SPDE.
This consistency gives thermodynamic support for the new nonlinear fluctuation--dissipation relation and leads
to a master fluctuating kinetic equation.  We further establish a sufficient condition
on the nonlinear dual dissipation potential under which its centered
cumulant admits a L\'evy--Khintchine representation, and show that the
resulting L\'evy current can be realized through Gaussian white noise and a
compensated Poisson random measure. As an application of this framework, an \(\alpha\)-stable L\'evy-type Dean-Kawasaki equation is derived. Furthermore, by rescaling the
fluctuating Boltzmann equation under a distinguished relation between the
noise intensity and the hydrodynamic scaling parameter, we prove that its
Hamiltonian converges to the Hamiltonian of the
Landau--Lifshitz--Navier--Stokes equations.  This reveals a scaling relation between kinetic fluctuations and
fluctuating hydrodynamics.

\end{abstract}

\keywords{fluctuating kinetic theory, Vlasov--Fokker--Planck equation,
Landau equation, Boltzmann equation, L\'evy current, dissipation potential}

\maketitle

\setcounter{tocdepth}{1}
\tableofcontents

\section{Introduction}

Kinetic theory connects microscopic particle dynamics with macroscopic
continuum equations.  A central example is the Boltzmann equation.  It is
obtained from the BBGKY hierarchy as an approximation for the one-particle
marginal by closing the hierarchy through propagation of chaos.  For a dilute
gas of hard spheres, it can also be derived directly from the microscopic
dynamics in the Boltzmann--Grad limit
\cite{LanfordClassicalSystems,GallagherSaintRaymondTexier}.  The resulting
mesoscopic description retains the positions and velocities of the particles,
while their individual collision histories are replaced by a statistical
collision operator.  We refer to
\cite{CercignaniBoltzmann,DiPernaLionsBoltzmann,VillaniCollisionalKinetic}
for its physical foundation, range of validity, and mathematical theory.

Another mesoscopic description is given by fluctuating hydrodynamics.  In the
theory of Landau and Lifshitz, thermal noises are added to the constitutive
fluxes of mass, momentum, and energy, and their strengths are determined by
the fluctuation--dissipation relation
\cite{LandauLifshitzFluidMechanics}.  These noises retain particle
fluctuations which are lost in deterministic continuum equations.  The
resulting conservative SPDEs present analytical difficulties, and their
large deviations are related to the fluctuations of the underlying
interacting particle systems \cite{FehrmanGessNonEquilibrium}.

To the best of the author's knowledge, there is no detailed theory relating
the scaling regimes of these two mesoscopic descriptions.  Classical
hydrodynamic-limit results show that suitably rescaled Boltzmann equations
converge to the incompressible Navier--Stokes equations or the
Navier--Stokes--Fourier equations
\cite{GolseSaintRaymondNSLimit,LevermoreMasmoudiNSF}.  A deterministic kinetic
limit removes the fluctuations which give rise to a stochastic hydrodynamic
model.  Motivated from this observation, we consider a finer scale between
the microscopic particle dynamics and the usual kinetic scale, at which the
kinetic description still retains random fluctuations.  We expect that, under
a joint hydrodynamic and fluctuation scaling, the resulting stochastic kinetic
equations converge to models such as the Landau--Lifshitz--Navier--Stokes
equations or the fluctuating Navier--Stokes--Fourier system
\cite{LandauLifshitzFluidMechanics,GessHeydeckerWuLLNS,GessSauerbreyWuNSF}.
Following fluctuating hydrodynamics, the noise in a kinetic equation should
be identified from the fluctuation--dissipation relation and from the
structure of the underlying particle system.  We call this finer mesoscopic
description \emph{fluctuating kinetic theory} (FKT).  The scales and the
relations considered in this program are shown in
Figure~\ref{fig:FKT-program}.

\begin{figure}[htbp]
\centering
\begin{tikzpicture}[
  x=1cm,y=1cm,
  >=Latex,
  model/.style={draw=black!65,rounded corners=2pt,fill=white,
    align=center,text width=3.35cm,minimum height=9mm,
    inner sep=3pt,font=\small},
  particle/.style={model,fill=black!4,text width=3.6cm},
  kinetic/.style={model,fill=black!2},
  flucthydro/.style={model,fill=green!7!white},
  fkt/.style={model,draw=blue!70!black,very thick,
    fill=blue!12!white},
  macro/.style={model,fill=black!5,text width=3.8cm},
  classical/.style={-{Latex[length=2mm]},semithick,draw=black!62},
  proposed/.style={-{Latex[length=2.2mm]},very thick,
    draw=blue!70!black},
  unknown/.style={<->,densely dashed,thick,draw=black!48}
]
  \node[particle] (particles) at (0,0)
    {Interacting particle systems};
  \node[fkt] (fkt) at (-3.35,1.45)
    {Fluctuating kinetic models};
  \node[kinetic] (kinetic) at (-3.35,3.15)
    {Kinetic models\\[-1mm]
     \scriptsize Boltzmann, Landau, VFP};
  \node[flucthydro] (fh) at (3.35,3.15)
    {Fluctuating hydrodynamics};
  \node[macro] (hydro) at (0,6.40)
    {Hydrodynamic models\\[-1mm]
     \scriptsize Navier--Stokes};

  \draw[classical,rounded corners=8pt]
    (particles.west) -- (-5.45,0) -- (-5.45,3.15) -- (kinetic.west);
  \draw[classical,rounded corners=8pt]
    (particles.east) -- (5.45,0) -- (5.45,3.15) -- (fh.east);
  \draw[proposed] (particles.north west) -- (fkt.south east);
  \draw[proposed] (fkt.north) --
    node[right,font=\scriptsize] {$\varepsilon\downarrow0$}
    (kinetic.south);
  \draw[proposed] (fkt.east) --
    node[pos=.58,below=11pt,inner sep=0pt,
      text width=1.9cm,font=\tiny,align=center]
      {stochastic\\hydrodynamic limit}
    (fh.south west);
  \draw[unknown] (kinetic.east) --
    node[above=5pt,fill=white,inner sep=1pt,align=center,
      font=\bfseries\scriptsize,text=black!80]
      {scaling relation:\\[.4mm]?}
    (fh.west);
  \draw[classical] (kinetic.north east) --
    node[pos=.61,above,sloped,font=\scriptsize] {hydrodynamic limit}
    (hydro.south west);
  \draw[classical] (fh.north west) --
    node[pos=.61,above,sloped,font=\scriptsize] {macroscopic limit}
    (hydro.south east);

  % The red frame follows only the blue FKT node and the three proposed
  % scale passages; the endpoint model boxes are deliberately excluded.
  \coordinate (blue-fkt-left) at ([xshift=-2mm]fkt.west);
  \coordinate (blue-particle-end) at
    ([xshift=-.7mm,yshift=.7mm]particles.north west);
  \coordinate (blue-kinetic-end) at
    ([yshift=-.7mm]kinetic.south);
  \coordinate (blue-fh-end) at
    ([xshift=-.7mm,yshift=-.7mm]fh.south west);
  \draw[-{Latex[length=2mm]},semithick,draw=black!45]
    (-6.25,-.2) -- (-6.25,6.70);
  \node[anchor=east,font=\bfseries\scriptsize,text=black!80]
    at (-6.42,0) {microscopic};
  \node[anchor=east,font=\bfseries\scriptsize,text=black!80]
    at (-6.42,2.75) {mesoscopic};
  \node[anchor=east,font=\bfseries\scriptsize,text=black!80]
    at (-6.42,6.40) {macroscopic};

  % Draw the scope frame last, so that its boundary remains in the
  % foreground.  Its label is placed outside and connected by a callout.
  \node[draw=red!72!black,very thick,rounded corners=5pt,
    fit=(fkt)(blue-fkt-left)(blue-particle-end)
        (blue-kinetic-end)(blue-fh-end),
    inner sep=0pt]
    (program) {};
  \node[anchor=west,fill=white,inner sep=1.3pt,
    font=\bfseries\scriptsize,text=black]
    (program-label) at (2.55,1.58)
    {Fluctuating kinetic theory};
  \draw[-{Latex[length=1.8mm]},semithick,draw=red!72!black]
    (program.east) -- (program-label.west);
\end{tikzpicture}
\caption{The scale gap considered in fluctuating kinetic theory.  The red
frame marks the scope of FKT: fluctuating kinetic models, their fluctuations
and large deviations, and their stochastic hydrodynamic limits.}
\label{fig:FKT-program}
\end{figure}
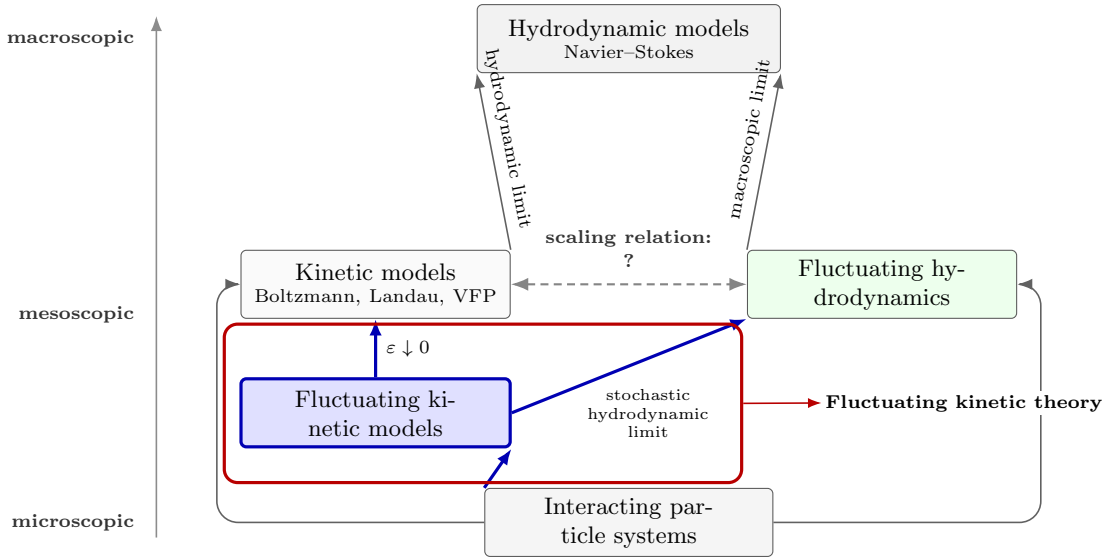
\FloatBarrier

One of the questions represented in Figure~\ref{fig:FKT-program} is how to couple
the kinetic fluctuation scale with the hydrodynamic rescaling such that a
nontrivial fluctuating hydrodynamic limit survives.  In the
Boltzmann-to-LLNS setting studied below, the kinetic
fluctuation parameter \(\varepsilon\), the hydrodynamic rescaling parameter
\(\delta\), and the limiting LLNS noise amplitude \(\mathfrak a\) are related
by
\[
 \frac{\varepsilon}{\delta^2}\longrightarrow\mathfrak a^2.
\]
Determining such scaling relations between stochastic kinetic equations and
fluctuating hydrodynamics is one motivation for the FKT program.

Fluctuation corrections to the Vlasov--Fokker--Planck (VFP), Landau, and
Boltzmann equations have recently been studied through approximation
schemes, solution theories, and fluctuation asymptotics
\cite{MullerVonRenesseZimmerVFP,HaoWuZimmerVFPDK,DuongHeWuLandau,
WuPoissonBoltzmann}.  Regarding to the fluctuations of these models, their
noises were identified by comparison with the corresponding particle
systems.  The fluctuating VFP and Landau equations describe mean-field
particle fluctuations, while the fluctuating Boltzmann equation describes
the randomness of hard-sphere collisions.  We ask how the dissipation of a
kinetic equation determines this noise.  The dissipative parts of the VFP
and Landau equations admit
gradient-flow formulations of linear Onsager type
\cite{DuongPeletierZimmerVFP,CarrilloDelgadinoDesvillettesWuLandau}.  Their
mobilities act linearly on the entropy force, so the classical
fluctuation--dissipation relation applies: the square root of the mobility
is the coefficient of a Gaussian noise, leading to a Dean--Kawasaki-type
noise.  The Boltzmann collision operator is different.  Its dissipative part
has a nonquadratic generalized gradient structure generated by a dual pair
of dissipation potentials \((\Psi_f,\Psi_f^*)\), where \(\Psi_f^*\) is of
cosh type \cite[Example~D, Eq.~(1.45)]{PeletierSchlichtingCosh}.  This structure is recalled in
Section~\ref{sec:Psi-structure-examples}.
This cosh potential is closely related to the dynamical large-deviation
functional for colliding particle systems
\cite{BodineauGallagherSaintRaymondSimonella}.  In
\cite{WuPoissonBoltzmann}, a Poisson collision noise was identified as the
fluctuation correction to the Boltzmann equation.  On the relevant regular
paths, its Freidlin--Wentzell action agrees with the hard-sphere
large-deviation functional.  To recover this Poisson noise directly from the
nonquadratic generalized gradient structure requires to use a nonlinear
extension of the fluctuation--dissipation relation.  We start from this
problem.

We propose the following program of \emph{fluctuating kinetic theory} for
thermodynamic fluctuations associated with generalized gradient structures
generated by dual dissipation potentials \((\Psi_f,\Psi_f^*)\).  The program
is based on a nonlinear fluctuation--dissipation relation and consists of four
steps.

\begin{enumerate}
\item[\emph{Step 1.}]
Formulate a master fluctuating kinetic equation driven by L\'evy noise.  The
noise is characterized by its L\'evy--Khintchine representation, whose
characteristics are determined by the underlying dual dissipation potentials
through the new nonlinear fluctuation--dissipation relation.  It includes
both the Gaussian current of Dean--Kawasaki type and Poisson collision noise.
The master equation therefore contains the fluctuating VFP, Landau, and
Boltzmann equations discussed above.

\item[\emph{Step 2.}]
Construct regularized versions of the models contained in the master
equation and develop their solution theories.

\item[\emph{Step 3.}]
Study the fluctuations and large deviations of these models systematically. 

\item[\emph{Step 4.}]
Investigate stochastic hydrodynamic limits.  A principal example is the
convergence of a rescaled fluctuating Boltzmann equation to the
Landau--Lifshitz--Navier--Stokes equations or to a fluctuating
Navier--Stokes--Fourier system.
\end{enumerate}

This program concerns kinetic theory, interacting particle systems,
stochastic analysis, and nonequilibrium thermodynamics.  In the present
paper, we carry out the first step and give the mathematical setting of the
master equation.  We also describe the remaining steps, state the questions
involved, and discuss their analytical difficulties.

We first introduce the master fluctuating kinetic equation.  Fix
\(d,m\in\mathbb N\), let
\(\mathsf Z=\R_x^d\times\R_v^d\) be the one-particle phase space, and let
\((\mathscr C,\lambda)\) be the measure space of elementary dissipative
channels.  Here \(L^0(\mathscr C,\lambda;\R^m)\) denotes the vector space
of measurable channel fields \(\xi:\mathscr C\to\R^m\), identified up to
\(\lambda\)-almost-everywhere equality.  For a kinetic state \(f\) and a
state test function \(p\), we write
\(\langle f,p\rangle=\int_{\mathsf Z}f(z)p(z)\,\d z\) whenever the integral
is finite.  A channel field
belongs to a test space
\(X\subset L^0(\mathscr C,\lambda;\R^m)\), with pairing
\[
 \langle j,\xi\rangle_{\mathscr C}
 :=\int_{\mathscr C}j(c)\cdot\xi(c)\,\lambda(\d c).
\]
For a state test function \(p\), \(\mathscr Dp(c)\) is the change in \(p\)
produced by the channel \(c\).  Its weak adjoint is defined by
\(\langle\mathscr D^*j,p\rangle
=\langle j,\mathscr Dp\rangle_{\mathscr C}\).  Let \(\mathcal E\) be the
entropy, with first variation \(\delta\mathcal E/\delta f\), and let
\(\Psi_f^*:X\to[0,\infty)\) be the even convex dual dissipation potential
at state \(f\), normalized by \(\Psi_f^*(0)=0\).  Finally, let
\(\mathcal T(f)\) denote the nondissipative kinetic transport.  The derivative
\(D\Psi_f^*(\eta)\) is the dissipative channel current generated by the
force \(\eta\).  With \(\varepsilon>0\) denoting the mass of one effective
particle, or equivalently the size of one jump, the proposed equation is
\begin{align}\label{intro-master-fluctuating-kinetic-equation}
 \d F_t^\varepsilon+\mathcal T(F_t^\varepsilon)\,\d t
 ={}&\mathscr D^*D\Psi_{F_t^\varepsilon}^*
 \left(-\frac12\mathscr D
 \frac{\delta\mathcal E}{\delta f}(F_t^\varepsilon)\right)\d t
 +\mathscr D^*\,\d\mathbb L_t^{
 \varepsilon,F_{t-}^\varepsilon}.
\end{align}
Here \(F_{t-}^\varepsilon\) is the left limit of the state, and
\(\mathbb L^{\varepsilon,f}\) is the centered fluctuation current with the
state frozen at \(f\).  This current is taken to be a centered cylindrical
L\'evy process whose law is characterized by the cumulant functional
\(\mathcal C_f\):
\begin{align}\label{intro-current-cumulant-law}
 \varepsilon\log\E\exp\left\{
 \varepsilon^{-1}\left\langle
 \mathbb L_{t+h}^{\varepsilon,f}-\mathbb L_t^{\varepsilon,f},\xi
 \right\rangle_{\mathscr C}\right\}
 =h\mathcal C_f(\xi),
 \qquad t,h\geq0.
\end{align}
In the centered, fully compensated convention used here, a local
L\'evy cumulant has the L\'evy--Khintchine representation
\begin{align}\label{intro-Levy-Khintchine}
 \mathcal C_f(\xi)
 ={}&\frac12\int_{\mathscr C}
 \xi(c)\cdot A_f(c)\xi(c)\,\lambda(\d c)+\int_{\mathscr C\times(\R^m\setminus\{0\})}
 \left(e^{z\cdot\xi(c)}-1-z\cdot\xi(c)\right)
 \nu_f^{\rm tilt}(\d c,\d z).
\end{align}
Here \(A_f(c)\) is a measurable symmetric nonnegative covariance density,
and \(\nu_f^{\rm tilt}\) is a L\'evy jump measure with the required
integrability and exponential moments.  Thus \(A_f\) describes the Gaussian
part of the current, while \(\nu_f^{\rm tilt}\) describes its jumps.

We determine the L\'evy characteristics from the dissipative structure.  More
precisely, we propose that the dual dissipation potential \(\Psi_f^*\)
determines the full cumulant through
\begin{align}\label{intro-centered-cumulant}
 \mathcal C_f(\xi)
 :={}&\Psi_f^*\left(-\frac12\mathscr D
 \frac{\delta\mathcal E}{\delta f}(f)+\xi\right)
 -\Psi_f^*\left(-\frac12\mathscr D
 \frac{\delta\mathcal E}{\delta f}(f)\right)
 -D\Psi_f^*\left(-\frac12\mathscr D
 \frac{\delta\mathcal E}{\delta f}(f)\right)[\xi].
\end{align}
We call \eqref{intro-centered-cumulant} the nonlinear L\'evy-type
fluctuation--dissipation relation.

For a quadratic \(\Psi_f^*\), relation
\eqref{intro-centered-cumulant} produces a Gaussian current, and for the
Boltzmann cosh potential, it produces a Poisson jump current.  Thus
\eqref{intro-master-fluctuating-kinetic-equation} contains the fluctuating
VFP, Landau, and Boltzmann equations.  Section~2 gives the precise
calculations and also treats a kinetic \(\alpha\)-stable Dean--Kawasaki
equation as a linear approximation of the Boltzmann collision mechanism.
This leads to the following questions.
\begin{enumerate}
\item Which dissipation potentials produce a compatible L\'evy exponent,
and are its Gaussian and jump characteristics uniquely determined?
\item Can the resulting current be represented as a stochastic integral against
Gaussian white noise and a compensated Poisson random measure?
\item Does \eqref{intro-centered-cumulant} express a
physical principle beyond the examples considered here, and how can its consistency
with the underlying thermodynamic dynamics be tested?
\end{enumerate}
Our first three results address these questions.
\begin{itemize}
\item Theorem~\ref{thm:intro-Levy-realizability} gives a sufficient
condition on \(\Psi_f^*\) under which the remainder in
\eqref{intro-centered-cumulant} is a L\'evy--Khintchine exponent and hence
determines admissible Gaussian and jump characteristics.
Proposition~\ref{prop:local-Levy-uniqueness} shows that these
characteristics are unique under the fixed compensation convention used
throughout the paper.
\item Theorem~\ref{thm:intro-stochastic-integral-realization} realizes the
resulting L\'evy current as a stochastic integral driven by Gaussian white
noise and a compensated Poisson random measure.
\item Theorem~\ref{thm:intro-rate-consistency} tests
\eqref{intro-centered-cumulant} through the dynamical action.  The generalized
gradient structure determines a canonical action whose local contribution is
the Legendre--Fenchel cost associated with the entropy force and the dual
dissipation potential.  Minimizing over currents satisfying the kinetic
continuity equation gives the cost of a prescribed density path.  On the
regular tilted class, this action agrees with the Freidlin--Wentzell-type
action obtained from the Gaussian--Poisson stochastic integral in the master
SPDE.  Hence the noise in \eqref{intro-centered-cumulant} has the fluctuation
cost encoded by the generalized gradient structure.  This gives the physical
support for regarding \eqref{intro-centered-cumulant} as the nonlinear
fluctuation--dissipation relation beyond the three kinetic examples above.
\end{itemize}

For the second part of the paper, we devote to deriving the stochastic
hydrodynamic limit in Figure~\ref{fig:FKT-program} for the fluctuating
Boltzmann equation.  We work on
\(\mathsf Z_{\rm B}=\T_x^3\times\R_v^3\), and let
\[
 \mathscr C_{\rm B}
 :=\T_x^3\times\R_v^3\times\R_{v_*}^3\times\mathbb S^2,
 \qquad \d c:=\d x\,\d v\,\d v_*\,\d\sigma.
\]
For \(c=(x,v,v_*,\sigma)\), set
\[
 v'=\frac{v+v_*}{2}+\frac{|v-v_*|}{2}\sigma,
 \qquad
 v_*'=\frac{v+v_*}{2}-\frac{|v-v_*|}{2}\sigma,
\]
and define
\[
 \mathscr D_{\rm B}q(c)
 :=q(x,v')+q(x,v_*')-q(x,v)-q(x,v_*),
 \qquad
 \mathcal J_c:=\delta_{(x,v')}+\delta_{(x,v_*')}
 -\delta_{(x,v)}-\delta_{(x,v_*)}.
\]
Let \(B(c)=B(v-v_*,\sigma)\) be a microreversible collision kernel, with the precise conditions stated in \eqref{collision-measure-microreversibility} and Assumption~\(\ref{ass:Boltzmann-kernel-conditions}\).  Let \(\delta\in(0,1)\) be the hydrodynamic rescaling parameter.  The incompressibly rescaled fluctuating Boltzmann equation is formally
\begin{align}\label{intro-rescaled-fluctuating-Boltzmann-SPDE}
 \d F_t^{\delta,\varepsilon}
 +\frac1\delta v\cdot\nabla_xF_t^{\delta,\varepsilon}\,\d t
 ={}&\frac1{\delta^2}Q_{\rm B}
 (F_t^{\delta,\varepsilon},F_t^{\delta,\varepsilon})\,\d t\notag\\
 &+\varepsilon\int_{\mathscr C_{\rm B}\times[0,\infty)}
 \mathcal J_c
 \boldsymbol1_{\{u\leq(2\varepsilon\delta^2)^{-1}
 BF_{t-}^{\delta,\varepsilon}F_{t-,*}^{\delta,\varepsilon}\}}
 \widetilde{\mathcal N}(\d t,\d c,\d u),
\end{align}
where \(\widetilde{\mathcal N}\) is the compensated Poisson random measure
associated with intensity \(\d t\,\d c\,\d u\), and
\(F_{t-,*}^{\delta,\varepsilon}(x,v,v_*)
:=F_{t-}^{\delta,\varepsilon}(x,v_*)\).  Under the diffusive scaling,
transport and collisions occur at rates \(\delta^{-1}\) and
\(\delta^{-2}\), respectively, while \(\varepsilon\) is the normalized mass
transferred by one effective hard-sphere collision.

The target equations are the incompressible Landau--Lifshitz--Navier--Stokes
equations
\begin{align}\label{intro-LLNS-SPDE}
 \d u+\Pi\nabla_x\cdot(u\otimes u)\,\d t
 =\mu\Delta u\,\d t
 -\mathfrak a\sqrt{2\mu}\,\Pi\nabla_x\cdot\d W^{\rm F},
 \qquad \nabla_x\cdot u=0.
\end{align}
Here \(\Pi\) is the Helmholtz projection, \(\mu>0\) is the viscosity, and
\(W^{\rm F}\) is symmetric trace-free space--time white noise with
\[
 \E[\dot W^{\rm F}_{ij}(t,x)\dot W^{\rm F}_{kl}(s,y)]
 =\mathsf C_{ij,kl}\delta(t-s)\delta(x-y),
 \qquad
 \mathsf C_{ij,kl}:=\delta_{ik}\delta_{jl}+\delta_{il}\delta_{jk}
 -\frac23\delta_{ij}\delta_{kl}.
\]
A process-level convergence theorem for these continuum SPDEs is currently
out of reach.  The obstructions are different on the kinetic and fluid sides,
but they already becomes apparent at the level of the continuum equations.
On the kinetic side, the Poisson collision noise in
\eqref{intro-rescaled-fluctuating-Boltzmann-SPDE} may destroy the
nonnegativity of the kinetic density, which is a new difficulty identified
in \cite{WuPoissonBoltzmann}.  That work therefore constructs and analyzes a
finite-dimensional coarse-grained jump process.  On the fluid side, the
random stress in three-dimensional LLNS is the divergence of space--time
white noise, making the equations supercritical under their natural
parabolic scaling.  Consequently, regularity structures and paracontrolled
distributions are not applicable in this setting
\cite{HairerRegularityStructures,GubinelliImkellerPerkowski}.  We therefore
study the stochastic scaling limit by comparing the exponential-generator
Hamiltonians of the two models.

For a Markov process \((X_t)_{t\geq0}\) with generator \(\mathcal L\),
started from a state \(x\), and for a suitable test observable \(\Phi\),
its exponential-generator Hamiltonian is
\begin{align}\label{intro-exponential-generator-Hamiltonian}
 \mathcal H(x,\Phi)
 :=e^{-\Phi(x)}\mathcal L e^\Phi(x)
 =\lim_{h\downarrow0}\frac1h
 \log\E_x\exp\{\Phi(X_h)-\Phi(x)\},
\end{align}
whenever the expressions are defined.  It is the infinitesimal logarithmic
moment-generating function of the tested increment and is used in the
Hamilton--Jacobi approach to large deviations
\cite[Chapters~1 and~5]{FengKurtzLargeDeviations}.  Its use in hydrodynamic
fluctuation theory is developed, for example, in
\cite{FengMikamiZimmer2021}.  

We next specify the two Hamiltonians.  For
\(\varrho>0\), \(U\in\R^3\), and \(\tau>0\), let
\[
 \mathcal M_{\varrho,U,\tau}(v)
 :=\varrho(2\pi\tau)^{-3/2}
 \exp\left(-\frac{|v-U|^2}{2\tau}\right),
\]
Let \(u\in C^\infty(\T^3;\R^3)\) be a divergence-free macroscopic velocity
field, and let \(\theta\in C^\infty(\T^3;\R)\) be a scalar thermodynamic
fluctuation.  For \(\delta\in(0,1)\) sufficiently small, set
\begin{align}\label{intro-local-Maxwellian}
 F_\delta[u,\theta](x,v)
 :=\mathcal M_{1-\delta\theta(x),\,\delta u(x),\,
                    1+\delta\theta(x)}(v).
\end{align}
We abbreviate \(F_\delta:=F_\delta[u,\theta]\) and
\(F_{\delta,*}(x,v,v_*):=F_\delta(x,v_*)\).  The pairing
\(\langle\cdot,\cdot\rangle\) below denotes integration over
\(\mathsf Z_{\rm B}\).
Write \(M:=\mathcal M_{1,0,1}\), let \(L\) be the nonnegative linearized
Boltzmann operator in \(L^2(M\,\d v)\), and define
\[
 \mathsf A_{ij}(v):=v_iv_j-\frac{|v|^2}{3}\delta_{ij},
 \qquad
 L\widehat{\mathsf A}_{ij}=\mathsf A_{ij},
 \qquad \widehat{\mathsf A}_{ij}\perp\ker L.
\]
Recall the trace-free isotropy tensor
\[
 \mathsf C_{ij,kl}
 :=\delta_{ik}\delta_{jl}+\delta_{il}\delta_{jk}
   -\frac23\delta_{ij}\delta_{kl},  \qquad 1\leq i,j,k,l\leq3. 
\]
Rotational invariance determines a unique coefficient \(\mu>0\) through
\[
 \int_{\R^3}\widehat{\mathsf A}_{ij}(v)\mathsf A_{kl}(v)M(v)\,\d v
 =\mu\mathsf C_{ij,kl},
 \qquad 1\leq i,j,k,l\leq3.
\]
Contracting the first
and third indices and the second and fourth indices gives
\[
 \sum_{i,j=1}^3\mathsf C_{ij,ij}
 =\sum_{i,j=1}^3\left(
   \delta_{ii}\delta_{jj}+\delta_{ij}\delta_{ji}
   -\frac23\delta_{ij}\delta_{ij}\right)
 =3\cdot3+3-\frac23\cdot3=10.
\]
Consequently,
\[
 \mu=\frac1{10}\int_{\R^3}
 \sum_{i,j=1}^3
 \widehat{\mathsf A}_{ij}(v)\mathsf A_{ij}(v)M(v)\,\d v, 
\]
and this coefficient is the shear viscosity in \eqref{intro-LLNS-SPDE}.
Given a smooth divergence-free momentum test field
\(\varphi\in C^\infty(\T^3;\R^3)\), put
\begin{align}\label{intro-perturbed-momentum-test}
 \nabla_{\rm sym}w
 &:=\frac12\bigl(\nabla w+(\nabla w)^{\mathsf T}\bigr),\notag\\
 q_{\delta,\varphi}(x,v)
 &:=\delta^{-1}v\cdot\varphi(x)
 +\sum_{i,j=1}^3
 (\nabla_{\rm sym}\varphi)_{ij}(x)\widehat{\mathsf A}_{ij}(v).
\end{align}
The corrected fluctuating Boltzmann Hamiltonian is
\begin{align}\label{intro-corrected-Boltzmann-Hamiltonian}
 \mathbb H_{\delta,\varepsilon}^{\rm B}(u,\theta;\varphi)
 :=-\frac1\delta
 \left\langle v\cdot\nabla_xF_\delta[u,\theta],
 q_{\delta,\varphi}\right\rangle+\frac1{2\varepsilon\delta^2}
 \int_{\mathscr C_{\rm B}}B(c)F_\delta F_{\delta,*}
 \left(e^{\varepsilon\mathscr D_{\rm B}q_{\delta,\varphi}(c)}-1\right)
 \,\d c,
\end{align}
while the LLNS Hamiltonian is
\begin{align}\label{intro-LLNS-Hamiltonian}
 \mathbb H_{\mathfrak a}^{\rm LLNS}(u;\varphi)
 :={}&\int_{\T^3}\left[
 \sum_{i,j=1}^3u_i u_j\partial_j\varphi_i
 -2\mu\sum_{i,j=1}^3(\nabla_{\rm sym}u)_{ij}
 (\nabla_{\rm sym}\varphi)_{ij}\right]\d x\notag\\
 &+\mathfrak a^2\mu\int_{\T^3}
 \sum_{i,j,k,l=1}^3
 (\nabla_{\rm sym}\varphi)_{ij}\mathsf C_{ij,kl}
 (\nabla_{\rm sym}\varphi)_{kl}\,\d x.
\end{align}
We will prove that, under the scaling relation
\begin{align}\label{intro-distinguished-stochastic-hydrodynamic-scaling}
 \delta\downarrow0,\qquad \varepsilon\downarrow0,
 \qquad \frac{\varepsilon}{\delta^2}\longrightarrow\mathfrak a^2,
\end{align}
the Hamiltonian of the rescaled fluctuating Boltzmann equation in
\eqref{intro-corrected-Boltzmann-Hamiltonian} converges to the LLNS
Hamiltonian in \eqref{intro-LLNS-Hamiltonian}.  Hence the kinetic fluctuation
scale and the hydrodynamic scale are related by
\eqref{intro-distinguished-stochastic-hydrodynamic-scaling} at the
Hamiltonian level.

\subsection{Main results}

We state concise versions of the four main results.  Their complete
hypotheses are given in the corresponding sections.  The first result gives
a sufficient condition under which the dissipation potential makes
\eqref{intro-centered-cumulant} a L\'evy cumulant.

\begin{theorem}
\label{thm:intro-Levy-realizability}
Fix \(f\in L^1(\mathsf Z)\) with \(f>0\) almost everywhere and satisfying
the regularity requirements of
Assumption~\ref{ass:Hessian-Levy-sufficient-condition}.  Suppose that
\(X\) satisfies the localization conditions
\eqref{localized-channel-test-assumption}--
\eqref{localized-channel-fields-in-X}.  Suppose also that
\(\Psi_f^*\) has the following form, 
\[
 \Psi_f^*(\xi)=\int_{\mathscr C}
 \psi_f^*(c,\xi(c))\,\lambda(\d c),
\]
and that its channelwise Hessian
\(H_f(c,r):=D_r^2\psi_f^*(c,r)\) is matrix-exponentially convex: for
\(\lambda\)-almost every \(c\),
\[
 \sum_{i,j=1}^N u_i\cdot H_f(c,r_i+r_j)u_j\geq0
\]
for every \(N\in\mathbb N\) and every
\(r_i,u_i\in\R^m\).  Assume also the Hessian finiteness and continuity
condition
\[
 D^2\Psi_f^*(\rho)[\xi,\zeta]
 =\int_{\mathscr C}\xi(c)\cdot
 H_f(c,\rho(c))\zeta(c)\,\lambda(\d c),
 \qquad \rho,\xi,\zeta\in X,
\]
where the right-hand side is a finite continuous bilinear form.  Then there
exist a measurable matrix field \(A_f\), unique \(\lambda\)-almost
everywhere and satisfying
\(A_f(c)^{\mathsf T}=A_f(c)\) and
\(u\cdot A_f(c)u\geq0\) for every \(c\in\mathscr C\) and
\(u\in\R^m\), and a unique L\'evy
measure \(\nu_f^{\rm tilt}\) such that \(\mathcal C_f\) has
the L\'evy--Khintchine representation
\eqref{intro-Levy-Khintchine}, namely,
\begin{align*}
 &\Psi_f^*\left(-\frac12\mathscr D
 \frac{\delta\mathcal E}{\delta f}(f)+\xi\right)
 -\Psi_f^*\left(-\frac12\mathscr D
 \frac{\delta\mathcal E}{\delta f}(f)\right)
 -D\Psi_f^*\left(-\frac12\mathscr D
 \frac{\delta\mathcal E}{\delta f}(f)\right)[\xi]\\
 &={}
 \frac12\int_{\mathscr C}
 \xi(c)\cdot A_f(c)\xi(c)\,\lambda(\d c)
 +\int_{\mathscr C\times(\R^m\setminus\{0\})}
 \left(e^{z\cdot\xi(c)}-1-z\cdot\xi(c)\right)
 \nu_f^{\rm tilt}(\d c,\d z),
 \qquad \xi\in X.
\end{align*}
\end{theorem}

The second result realizes a current with a given L\'evy--Khintchine
representation as a Gaussian--Poisson stochastic integral.

\begin{theorem}
\label{thm:intro-stochastic-integral-realization}
Fix \(f\in L^1(\mathsf Z)\) with \(f>0\) almost everywhere, and suppose
that \(\mathcal C_f\) admits the representation
\eqref{intro-Levy-Khintchine} with admissible characteristics
\(A_f\) and \(\nu_f^{\rm tilt}\).  Let
\(\sigma_f=A_f^{1/2}\).  If \(W\) is Gaussian white noise with control
measure \(\d s\,\lambda(\d c)\), and if
\(\widetilde{\mathcal N}_f^\varepsilon\) is an independent compensated
Poisson random measure with compensator
\(\varepsilon^{-1}\d s\,\nu_f^{\rm tilt}(\d c,\d z)\), then
\begin{align*}
 \left\langle\mathbb L_t^{\varepsilon,f},\xi\right\rangle_{\mathscr C}
 :=\sqrt\varepsilon\int_0^t\int_{\mathscr C}
 \sigma_f(c)^{\mathsf T}\xi(c)\cdot W(\d s,\d c)+\varepsilon\int_0^t
 \int_{\mathscr C\times(\R^m\setminus\{0\})}
 z\cdot\xi(c)\,
 \widetilde{\mathcal N}_f^\varepsilon(\d s,\d c,\d z)
\end{align*}
defines a centered cylindrical L\'evy current satisfying
\eqref{intro-current-cumulant-law} for every \(t,h\geq0\) and every
\(\xi\in X\).  Hence it realizes the current prescribed by
\eqref{intro-centered-cumulant}, uniquely in cylindrical law.
\end{theorem}

This gives the following formal Gaussian--Poisson stochastic-integral
representation of the master fluctuating kinetic equation:
\begin{align}\label{intro-explicit-master-SPDE}
 \d F_t^\varepsilon+\mathcal T(F_t^\varepsilon)\,\d t
 ={}&\mathscr D^*D\Psi_{F_t^\varepsilon}^*
 \left(-\frac12\mathscr D
 \frac{\delta\mathcal E}{\delta f}(F_t^\varepsilon)\right)\d t
 \notag\\
 &+\sqrt\varepsilon\,\mathscr D^*\left(
 \sigma_{F_{t-}^\varepsilon}(c)W(\d t,\d c)\right)+\varepsilon\mathscr D^*
 \int_{\R^m\setminus\{0\}}z\,
 \widetilde{\mathcal N}^\varepsilon(\d t,\d c,\d z),
\end{align}
where the predictable compensator of
\(\mathcal N^\varepsilon\), conditional on the left-limit state, is
\(\varepsilon^{-1}\d t\,
\nu_{F_{t-}^\varepsilon}^{\rm tilt}(\d c,\d z)\), and
\begin{align*}
 \widetilde{\mathcal N}^\varepsilon(\d t,\d c,\d z)
 :=\mathcal N^\varepsilon(\d t,\d c,\d z)
 -\varepsilon^{-1}\d t\,
 \nu_{F_{t-}^\varepsilon}^{\rm tilt}(\d c,\d z).
\end{align*}
Theorem~\ref{thm:intro-stochastic-integral-realization} proves the
representation at every frozen state.  Evaluating the same characteristics
along the evolving state gives the formal L\'evy-type SPDE
\eqref{intro-explicit-master-SPDE}.  Its solution theory depends on the
resulting state-dependent coefficients and belongs to Step~2 of the program.

The third result compares two path-space costs.  Fix \(T>0\) and a regular
path
\[
 f\in C^1([0,T];L^1(\mathsf Z)),
 \qquad \mathcal T(f)\in C([0,T];L^1(\mathsf Z)),
 \qquad f_0:=f(0),
\]
for which the expressions below are finite.  A Gaussian--Poisson control is
a pair \((h,q)\), where
\begin{align*}
 h&\in L^2([0,T]\times\mathscr C,
            \d t\,\lambda;\R^m),\\
 q&:[0,T]\times\mathscr C\times(\R^m\setminus\{0\})
     \longrightarrow[0,\infty)
\end{align*}
is measurable.  Its instantaneous cost is
\begin{align}\label{intro-control-cost}
 \mathfrak K_f(h,q)
 :={}&\frac12\int_{\mathscr C}|h(c)|^2\,\lambda(\d c)\notag\\
 &+\int_{\mathscr C\times(\R^m\setminus\{0\})}
 \ell(q(c,z))\,\nu_f^{\rm tilt}(\d c,\d z),
 \qquad \ell(r):=r\log r-r+1,\qquad 0\log0:=0.
\end{align}
The controlled current acts on \(\xi\in X\) by
\begin{align}\label{intro-controlled-current}
 J_f^{h,q}[\xi]
 :={}&D\Psi_f^*\left(-\frac12\mathscr D
 \frac{\delta\mathcal E}{\delta f}\right)[\xi]
 +\int_{\mathscr C}\xi(c)\cdot\sigma_f(c)h(c)\,
 \lambda(\d c)\notag\\
 &+\int_{\mathscr C\times(\R^m\setminus\{0\})}
 z\cdot\xi(c)\bigl(q(c,z)-1\bigr)
 \nu_f^{\rm tilt}(\d c,\d z).
\end{align}
We call \((h,q)\) admissible for \(f\) if its time-integrated cost is
finite and the controlled skeleton equation
\begin{align}\label{intro-controlled-skeleton}
 \left\langle\partial_t f_t+\mathcal T(f_t),p\right\rangle
 =J_{f_t}^{h_t,q_t}[\mathscr Dp]
\end{align}
holds for almost every \(t\in[0,T]\) and every
\(p\in C_c^\infty(\mathsf Z)\).  We denote the set of these controls by
\(\mathcal U(f)\).  The weak-convergence representation of the stochastic
integral then gives the controlled action
\begin{align}\label{intro-SPDE-action}
 I_{[0,T]}^{\rm SPDE}(f)
 :=\inf_{(h,q)\in\mathcal U(f)}
 \int_0^T\mathfrak K_{f_t}(h_t,q_t)\,\d t,
\end{align}
with value \(+\infty\) if \(\mathcal U(f)=\varnothing\).  Assuming that the
transport preserves \(\mathcal E\) and that the chain rule holds, the
canonical action reads 
\begin{align}\label{intro-canonical-density-action}
 I_{[0,T]}^{\rm can}(f)
 :={}&\frac12\bigl[\mathcal E(f_T)-\mathcal E(f_0)\bigr]+\int_0^T
 \inf_{\substack{J\in X^*:\;\mathscr D^*J
 =\partial_t f_t+\mathcal T(f_t)}}
 \Psi_{f_t}(J)\,\d t+\int_0^T\Psi_{f_t}^*\left(-\frac12\mathscr D
 \frac{\delta\mathcal E}{\delta f}(f_t)\right)\,\d t.
\end{align}
Here \(\Psi_f\) is the Legendre dual of \(\Psi_f^*\) defined in
\eqref{primal-dissipation-potential}, and \eqref{intro-canonical-density-action} is the explicit density-level
\(\Psi+\Psi^*\)-type energy--dissipation action.

\begin{theorem}
\label{thm:intro-rate-consistency}
The classes \(\mathcal A_{\rm reg}\) and \(\mathcal A_{\rm tilt}\) are
introduced in Definition~\ref{def:regular-canonical-path}.  For every path
\(f\in\mathcal A_{\rm reg}\), we have 
\begin{align*}
 I_{[0,T]}^{\rm can}(f)\leq I_{[0,T]}^{\rm SPDE}(f).
\end{align*}
If \(f\in\mathcal A_{\rm tilt}\) and \(p\) is a corresponding tilt field
from the same definition, define
\begin{align*}
 h_t^p(c)&:=\sigma_{f_t}(c)^{\mathsf T}\mathscr Dp_t(c),\\
 q_t^p(c,z)&:=\exp\{z\cdot\mathscr Dp_t(c)\}.
\end{align*}
Then \((h^p,q^p)\in\mathcal U(f)\), and
\begin{align}\label{intro-rate-consistency}
 I_{[0,T]}^{\rm can}(f)
 =I_{[0,T]}^{\rm SPDE}(f)
 =\int_0^T\mathfrak K_{f_t}(h_t^p,q_t^p)\,\d t.
\end{align}
Moreover, \((h^p,q^p)\) minimizes the SPDE control cost.
\end{theorem}

On the regular tilted class, the rate function obtained from the
Gaussian--Poisson stochastic integral therefore agrees with the canonical
rate function.  The fourth result compares the two Hamiltonians introduced
above and gives the Hamiltonian formulation of the stochastic hydrodynamic
limit.

\begin{theorem}
\label{thm:intro-Boltzmann-LLNS-Hamiltonian}
Fix
\begin{align*}
 u,\varphi\in C^\infty(\T^3;\R^3),
 \qquad \theta\in C^\infty(\T^3;\R),
 \qquad \nabla\cdot u=\nabla\cdot\varphi=0.
\end{align*}
Here \(u\) is the macroscopic velocity, \(\theta\) is the scalar
thermodynamic fluctuation that produces the density and temperature
parameters \(1-\delta\theta\) and \(1+\delta\theta\) in the local Maxwellian
\eqref{intro-local-Maxwellian}, and \(\varphi\) is the momentum test field.
For these fixed fields, define \(F_\delta[u,\theta]\),
\(q_{\delta,\varphi}\), \(\mathbb H_{\delta,\varepsilon}^{\rm B}\), and
\(\mathbb H_{\mathfrak a}^{\rm LLNS}\) by
\eqref{intro-local-Maxwellian}, \eqref{intro-perturbed-momentum-test},
\eqref{intro-corrected-Boltzmann-Hamiltonian}, and
\eqref{intro-LLNS-Hamiltonian}, respectively.
Assume that the collision kernel satisfies
Assumption~\ref{ass:Boltzmann-kernel-conditions}.  If
\begin{align*}
 \delta\downarrow0,\qquad \varepsilon\downarrow0,
 \qquad \frac{\varepsilon}{\delta^2}\longrightarrow\mathfrak a^2
 \quad\text{for some }\mathfrak a\geq0,
\end{align*}
then
\begin{align}\label{intro-Boltzmann-LLNS-Hamiltonian-limit}
 \mathbb H_{\delta,\varepsilon}^{\rm B}(u,\theta;\varphi)
 \longrightarrow
 \mathbb H_{\mathfrak a}^{\rm LLNS}(u;\varphi).
\end{align}
\end{theorem}

\subsection{Background and related literature}

\ 

\textbf{Fluctuation--dissipation relations and generalized gradient
structures.}
We start from Onsager's theory of linear irreversible processes.  Near
equilibrium, a symmetric positive mobility gives the linear dependence of
thermodynamic currents on their conjugate forces, and microscopic
reversibility gives the reciprocal relations for the kinetic coefficients
\cite{OnsagerI,OnsagerII}.  The probabilistic statement is that the same
mobility determines the covariance of the equilibrium fluctuations.  This
fluctuation--dissipation relation was developed by Callen and Welton and by
the Green--Kubo theory of linear response
\cite{CallenWelton,GreenIrreversibleFluids,KuboLinearResponse}.  In a
Markovian Langevin description, it gives the familiar square-root-mobility
Gaussian noise, up to the normalization by the temperature and Boltzmann
constant.  Onsager and Machlup expressed the probability of an entire
Gaussian fluctuation path through the associated quadratic dissipation
function \cite{OnsagerMachlup}, while Landau and Lifshitz used the same
principle to prescribe fluctuating stresses and heat fluxes in hydrodynamics
\cite{LandauLifshitzFluidMechanics}.

Generalized gradient systems replace the quadratic Onsager dissipation by a
convex dual pair \((\Psi_f,\Psi_f^*)\).  The derivative of \(\Psi_f^*\)
then gives a nonlinear force--current law.  Such structures arise from
dynamical large deviations
\cite{MielkePeletierRenger,LieroMielkePeletierRenger}.  Diffusive
microscopic dynamics lead to quadratic potentials and Gaussian
fluctuations, whereas jump dynamics lead to nonquadratic potentials, often
of cosh type \cite{PeletierSchlichtingCosh}.  We ask for the corresponding fluctuation--dissipation
relation at the level of the full noise law.  In
\eqref{intro-centered-cumulant}, the finite-force Bregman remainder of
\(\Psi_f^*\) is identified with a L\'evy--Khintchine cumulant.  For a
quadratic potential this gives the classical Gaussian relation, while for a
nonlinear jump dissipation it also gives Poissonian fluctuations.

\textbf{GENERIC structure.}
The General Equation for Non-Equilibrium Reversible--Irreversible Coupling
(GENERIC) describes reversible and irreversible dynamics in the same
thermodynamic system.  It was developed by Grmela and
\"Ottinger \cite{GrmelaOttingerGENERICI,OttingerGrmelaGENERICII}; see also
the monograph \cite{OttingerBeyondEquilibrium}.  The reversible part is
generated by the energy through an antisymmetric Poisson operator, while
the irreversible part is generated by the entropy through a symmetric
positive dissipative operator.  Nonlinear dissipation potentials,
variational formulations, and their relations to large deviations are
studied in
\cite{MielkeGENERIC,DuongPeletierZimmerVFP}.

Pre-GENERIC combines a vector field tangent to the entropy level sets with a
generalized gradient flow, without requiring a Poisson operator or the full
GENERIC degeneracy conditions \cite{KraaijLazarescuMaesPeletier}.  When
\(\mathcal T\) preserves \(\mathcal E\), the deterministic part of the
master equation is of this form: \(\mathcal T(f)\) is the non-dissipative
kinetic transport and \((\Psi_f,\Psi_f^*)\) determines the dissipative
current.  We add a stochastic part by using the centered Bregman remainder
of the same potential \(\Psi_f^*\) as the L\'evy fluctuation cumulant.  Thus,
when the transport contributes no noise, the master equation is a
fluctuating extension of generalized pre-GENERIC.

A full GENERIC formulation also requires the Hamiltonian structure of the
transport term and may require a larger state space.  For example, the
Vlasov--Fokker--Planck density describes an open subsystem that exchanges
energy with a heat bath.  Duong, Peletier, and Zimmer obtain a closed GENERIC
system by adjoining a scalar variable that records this energy transfer
\cite[Section~3.1]{DuongPeletierZimmerVFP}.  Since we study only the fluctuations generated by the dissipative structure and want our framework to cover the single VFP equation, we do not adopt the GENERIC structure. Instead, we keep \(f\) as the state variable and write \(\mathcal T\) as the kinetic transport operator.

\textbf{Macroscopic fluctuation theory.}
Macroscopic fluctuation theory (MFT) is based on dynamical large deviations
of diffusive interacting particle systems
\cite{BertiniMFT2002,KipnisLandim,BertiniMFT}.  Its basic variables are a density and
a current linked by the continuity equation.  The constitutive current and
the mobility determine a local current cost.  Contracting over currents
compatible with a prescribed density path gives the density-path action.
The same density--current formulation gives the quasipotential and the
Hamilton--Jacobi equation for nonequilibrium steady states.  Force--current
decompositions with nonquadratic dissipation are studied in
\cite{PattersonRengerSharma,RengerMFTGENERIC}.  The canonical action used in
Section~4 is the kinetic counterpart of this construction.

\textbf{Kinetic theory and underlying particle systems.}
The kinetic equations considered here arise from different interaction
regimes: VFP and Landau equation from mean-field interacting particle systems,
and the Boltzmann equation from binary collisions.  Their physical
derivations, scalings, and mathematical structures are discussed in
\cite{CercignaniIllnerPulvirenti,SpohnLargeScale,
VillaniCollisionalKinetic}.

For the VFP equation, the
deterministic theory includes hypoelliptic regularization, hypocoercivity,
confinement, and convergence to equilibrium
\cite{Bouchut2002,VillaniHypocoercivity,DuongPeletierZimmerVFP}.
For the Landau equation, it includes homogeneous, inhomogeneous dynamics, and kinetic regularity
\cite{VillaniLandau1998,GuoLandau2002,GolseImbertMouhotVasseur}.  The corresponding
Boltzmann theory includes global solutions near vacuum, renormalized
solutions for large data, perturbative solutions near equilibrium, and the
non-cutoff theory
\cite{DiPernaLionsBoltzmann,GuoBoltzmann2004,GressmanStrain2011}.

At the microscopic level, nonlinear kinetic equations are obtained from
many-particle dynamics through propagation of chaos
\cite{Kac1956,Sznitman1991,MischlerMouhot2013,JabinWang2018}.  Deng, Hani,
and Ma recently derived the Boltzmann equation from hard-sphere dynamics on
arbitrary finite time intervals over which the corresponding Boltzmann
solution remains regular \cite{DengHaniMaLongTime}.  In a companion work,
they combined this long-time kinetic limit with hydrodynamic limits of the
Boltzmann equation to derive the compressible Euler and incompressible
Navier--Stokes--Fourier equations from hard-sphere dynamics
\cite{DengHaniMaHilbertSixth}.  Central-limit
fluctuations and pathwise large deviations retain the next-order information
which is absent from the deterministic kinetic equation
\cite{Meleard1996,DawsonGartner1987,
BodineauGallagherSaintRaymondSimonella,Daniel23}.  This is the information
which the fluctuating kinetic models are intended to describe.

\textbf{From the Boltzmann equation to the Navier--Stokes equations.}
The hydrodynamic limit of the Boltzmann equation gives a rigorous connection
between kinetic theory and incompressible fluid dynamics.  Under the
incompressible scaling, solutions close to a global Maxwellian converge, in
suitable regimes, to the incompressible Navier--Stokes--Fourier system
\cite{GolseSaintRaymondNSLimit,LevermoreMasmoudiNSF}.  These results provide the
deterministic background for the fluctuating hydrodynamic limit studied in
Section~5.

\textbf{Fluctuating hydrodynamics and conservative SPDEs.}
In fluctuating hydrodynamics, random stresses and fluxes are added to
deterministic conservation laws, with covariance fixed by thermodynamic
dissipation.  The physical derivations include the Landau--Lifshitz
formulation, the theory of hydrodynamic fluctuations, and the
Boltzmann--Langevin equation
\cite{LandauLifshitzFluidMechanics,BixonZwanzig,Espanol1998}.  Modern formulations and numerical studies emphasize the
same conservative fluctuation--dissipation balance
\cite{BellGarciaWilliams2007,DonevVandenEijndenGarciaBell2010}.

The Dean--Kawasaki equation is the diffusive example of a conservative
SPDE.  Its exact singular form, its rigidity, and its regularized or
correlated-noise variants have been studied in
\cite{Kawasaki1994,Dean1996,KonarovskyiLehmannVonRenesse2019,CornalbaFischer2023,CornalbaFischerIngmannsRaithel,
DjurdjevacKrempPerkowski,FehrmanGess2024,
DirrFehrmanGess2024,WangWuZhangSingularDK,
WuFluctuatingIsingKacKawasaki,MartiniMayorcasAdditiveNoise}.  For fluctuating fluids, the principal models are the
Landau--Lifshitz--Navier--Stokes equations
\cite{LandauLifshitzFluidMechanics,DonevVandenEijndenGarciaBell2010,
GessHeydeckerWuLLNS} and the fluctuating Navier--Stokes--Fourier system
\cite{LandauLifshitzFluidMechanics,Espanol1998,
GessSauerbreyWuNSF}.  Their
space--time-white conservative fluxes make the continuum equations highly
singular, so several works interpret them through spatial regularization,
discretization, large deviations, or effective low-wave-number dynamics.

Thermodynamically structured conservative noises also occur in stochastic
thin-film equations
\cite{FischerGrun2018,DareiotisGessGnannGrun2021,Sauerbrey2024} and in geometric SPDEs such as stochastic
curve-shortening and mean-curvature flows
\cite{DirrLuckhausNovaga2001,HofmanovaRogerVonRenesse2021}.  As in the
kinetic models considered here, their analysis involves degenerate
mobilities, conservative rough noise, positivity or geometric constraints,
and weak solution concepts.

\textbf{Large deviations for SPDEs.}
The weak-convergence approach used below is based on variational
representations for exponential functionals of the driving noise.  The
Bou\'e--Dupuis formula treats finite-dimensional Brownian motion
\cite{BoueDupuis1998}.  Its infinite-dimensional extension gives a criterion
for stochastic systems driven by Brownian noise
\cite{BudhirajaDupuis2000,BudhirajaDupuisMaroulas2008}.  In its usual SPDE form, the
criterion reduces the proof of the Laplace principle to two properties of
the controlled solution maps. This
method has been applied to stochastic fluid equations, reaction--diffusion
and porous-medium equations, conservation laws, and singular scaling limits
\cite{SritharanSundar2006,FehrmanGessNonEquilibrium,
BudhirajaDupuisMaroulas2008,ZhaiZhangSecondGrade}.

For a Poisson random measure, changing the compensator has an entropy cost
in place of the quadratic Brownian control cost.  Variational representations
for Brownian--Poisson functionals and the corresponding weak-convergence
criterion were developed in
\cite{BudhirajaDupuisMaroulas,BudhirajaChenDupuis}.  Applications include
stochastic hydrodynamical models, shell models and porous-medium equations
with jump noise
\cite{MannaSritharanSundar2009,BudhirajaChenDupuis,
YangZhaiZhangJump,ZhaiZhangLevyNS}.  We use this Gaussian--Poisson control structure in
Section~4 to derive the formal Freidlin--Wentzell action of the master
equation and compare it with the canonical action determined by the dual
dissipation potentials.

\subsection{Structure of the paper}

The structure of this paper is organized as follows.  In Section~2, we
recall the master fluctuating kinetic equation, fix its formal functional
setting, and derive the generalized gradient structures of the VFP, Landau,
and Boltzmann equations.  The kinetic \(\alpha\)-stable Dean--Kawasaki
equation is also introduced as a linear approximation of the Boltzmann
collision mechanism.  In Section~3, we prove a Hessian criterion for L\'evy
realizability and the uniqueness of the resulting local L\'evy
characteristics, and then give a Gaussian--Poisson stochastic-integral
realization of the cylindrical current.  In Section~4, we introduce the
canonical action and its thermodynamic meaning, derive the candidate SPDE
action by the weak-convergence method, and compare the two actions using
Gaussian--Poisson convex duality.  Equality is proved on a regular tilted
class, including the Boltzmann paths considered in
\cite[Appendix~B]{WuPoissonBoltzmann}, and is used to justify the nonlinear
fluctuation--dissipation relation.  In Section~5, we introduce the rescaled
fluctuating Boltzmann and LLNS models and prove their Hamiltonian consistency
under the scaling relation \(\varepsilon/\delta^2\to\mathfrak a^2\).  In
Section~6, we summarize the available results and discuss the open problems
and the scope of the program.

\section{Generalized gradient structures of kinetic equations}
\label{sec:Psi-structure-examples}

In this section, we present the formal structure of the master fluctuating
kinetic equation and show how it contains the fluctuating
Vlasov--Fokker--Planck, Landau, and Boltzmann equations.  After deriving
these three examples, we present a kinetic \(\alpha\)-stable
Dean--Kawasaki equation as a linear approximation of the Boltzmann
collision mechanism.  The discussion is structural.  Once
\(\mathcal E\), \(\mathscr D\), and \(\Psi_f^*\) are specified, each model
requires its own state space, regularization, and solution theory.

Throughout the paper, all stochastic objects are defined on a complete
filtered probability space
\begin{align*}
 (\Omega,\mathcal F,(\mathcal F_t)_{t\geq0},\mathbb P)
\end{align*}
satisfying the usual conditions.  The space is assumed to be rich enough to
support the Gaussian white noises and Poisson random measures used below.
Processes are adapted, and stochastic integrands and state-dependent
intensities are predictable whenever required.

Let \(\mathsf Z=\R_x^d\times\R_v^d\).  We write
\begin{align*}
 L^0(\mathsf Z)
 :=\{p:\mathsf Z\to\R:p\text{ is Lebesgue measurable}\}/\!\sim,
\end{align*}
where \(p\sim q\) means that \(p=q\) almost everywhere, and set
\begin{align*}
 L_+^1(\mathsf Z)
 :=\{f\in L^1(\mathsf Z):f\geq0\ \text{a.e.}\}.
\end{align*}

The entropy or free energy is an extended real-valued functional
\begin{align}\label{kinetic-state-functional-setting}
 \mathcal E:\operatorname{Dom}(\mathcal E)
 \subset L_+^1(\mathsf Z)\longrightarrow\R\cup\{+\infty\}.
\end{align}
At the states where it exists, its first variation is the map
\begin{align}\label{entropy-first-variation-domain}
 \frac{\delta\mathcal E}{\delta f}:
 \operatorname{Dom}\left(\frac{\delta\mathcal E}{\delta f}\right)
 \subset\operatorname{Dom}(\mathcal E)\longrightarrow L^0(\mathsf Z)
\end{align}
characterized formally by
\begin{align}\label{entropy-first-variation-definition}
 \left.\frac{\d}{\d s}\right|_{s=0}\mathcal E(f+sh)
 =\int_{\mathsf Z}\frac{\delta\mathcal E}{\delta f}(f)(z)h(z)\,\d z
\end{align}
for admissible perturbations \(h\).  The master equation is considered at
states for which \(\frac{\delta\mathcal E}{\delta f}(f)\) belongs to the
domain of the channel derivative introduced next.

Let \((\mathscr C,\mathcal B,\lambda)\) be a standard Borel space with a
\(\sigma\)-finite measure \(\lambda\).  Its points label the elementary
channels through which the kinetic state dissipates.  A channel records the
variables needed to specify one local dissipative or collision event, and
its form depends on the model.  For VFP,
\(\mathscr C=\R_x^d\times\R_v^d\) and \(c=(x,v)\).  For Landau,
\(\mathscr C=\R_x^d\times\R_v^d\times\R_{v_*}^d\) and
\(c=(x,v,v_*)\).  For Boltzmann,
\(\mathscr C=\R_x^d\times\R_v^d\times\R_{v_*}^d\times\mathbb S^{d-1}\)
and \(c=(x,v,v_*,\sigma)\).  For the kinetic \(\alpha\)-stable model,
\(\mathscr C=\R_x^d\times\R_v^d\times\R_*^d\) and \(c=(x,v,z)\).
The channel current is vector-valued in the VFP and Landau examples and
scalar-valued in the Boltzmann and \(\alpha\)-stable examples.

A channel field is a measurable map
\(\xi:\mathscr C\to\R^m\), where \(m=d\) for VFP and Landau and \(m=1\)
for Boltzmann and the \(\alpha\)-stable model.  It tests, or biases, the
current carried by every elementary channel.  As usual,
\(L^0(\mathscr C,\lambda;\R^m)\) denotes these fields modulo
\(\lambda\)-almost-everywhere equality.  We use a real Banach space
\begin{align}\label{Banach-channel-space}
 X\subset L^0(\mathscr C,\lambda;\R^m)
\end{align}
and denote its continuous dual by \(X^*\).  Suppose that a channel current
\(J\in X^*\) is induced by an \(\R^m\)-valued measure that is absolutely
continuous with respect to \(\lambda\).  Writing its Radon--Nikodym density
as \(j:\mathscr C\to\R^m\), we have
\(J(\d c)=j(c)\lambda(\d c)\), and its action on \(\xi\in X\) is
\begin{align}\label{channel-pairing}
 J[\xi]=\langle j,\xi\rangle_{\mathscr C}
 :=\int_{\mathscr C}j(c)\cdot\xi(c)\,\lambda(\d c).
\end{align}

The channel derivative is a linear map
\begin{align}\label{channel-derivative-domain}
 \mathscr D:\operatorname{Dom}(\mathscr D)
 \subset L^0(\mathsf Z)\longrightarrow X.
\end{align}
We regard \(\operatorname{Dom}(\mathscr D)\) as a topological vector space
on which \(\mathscr D\) is continuous, and denote its continuous dual by
\(\bigl(\operatorname{Dom}(\mathscr D)\bigr)^*\).  The relevant kinetic
states are assumed to act continuously on this test space through the
density--test pairing, and are therefore also regarded as elements of
\(\bigl(\operatorname{Dom}(\mathscr D)\bigr)^*\).  

For every state \(f\), the dual dissipation potential is a proper, lower
semicontinuous, even convex functional
\begin{align}\label{dual-dissipation-potential-domain}
 \Psi_f^*:X\longrightarrow[0,+\infty],
 \qquad \Psi_f^*(0)=0, 
\end{align}
with domain
\(\operatorname{Dom}(\Psi_f^*):=\{\xi\in X:\Psi_f^*(\xi)<\infty\}\).  Its Legendre
dual is
\begin{align}\label{primal-dissipation-potential}
 \Psi_f:X^*\longrightarrow[0,+\infty],\qquad \Psi_f(J):=\sup_{\xi\in X}\{J[\xi]-\Psi_f^*(\xi)\}.
\end{align}
At its differentiability points, the G\^ateaux derivative is the map
\begin{align}\label{dual-dissipation-derivative-domain}
 D\Psi_f^*:\operatorname{Dom}(D\Psi_f^*)
 \subset\operatorname{Dom}(\Psi_f^*)\longrightarrow X^*,
\end{align}
defined by
\begin{align*}
 D\Psi_f^*(\eta)[\xi]
 :=\left.\frac{\d}{\d s}\right|_{s=0}\Psi_f^*(\eta+s\xi), \quad \text{for any }\xi\in X. 
\end{align*}

The weak adjoint of the channel derivative 
\begin{align}\label{channel-adjoint-domain}
 \mathscr D^*:X^*\longrightarrow
 \bigl(\operatorname{Dom}(\mathscr D)\bigr)^*
\end{align}
is defined by
\begin{align}\label{channel-adjoint-pairing}
 \langle\mathscr D^*J,p\rangle:=J[\mathscr Dp],
 \qquad J\in X^*,\ p\in\operatorname{Dom}(\mathscr D).
\end{align}
The nondissipative kinetic transport has the form
\begin{align}\label{kinetic-transport-domain}
 \mathcal T:\operatorname{Dom}(\mathcal T)
 \subset L_+^1(\mathsf Z)\longrightarrow
 \bigl(\operatorname{Dom}(\mathscr D)\bigr)^*.
\end{align}
All domains above are chosen for the particular kinetic model under
consideration.

The entropy force and the corresponding deterministic dissipative current
are
\begin{align}\label{entropy-force-convention}
 \eta_f:=-\frac12\mathscr D
 \frac{\delta\mathcal E}{\delta f}(f)
 \in\operatorname{Dom}(D\Psi_f^*),
 \qquad D\Psi_f^*(\eta_f)\in X^*.
\end{align}
The nonlinear fluctuation--dissipation relation assigns to the same
potential the centered functional
\begin{align}\label{centered-Psi-cumulant}
 \mathcal C_f:X&\longrightarrow[0,+\infty],
 \notag\\
 \xi&\longmapsto\Psi_f^*(\eta_f+\xi)-\Psi_f^*(\eta_f)
 -D\Psi_f^*(\eta_f)[\xi].
\end{align}
Here \(\mathcal C_f(\xi)=+\infty\) is allowed when
\(\eta_f+\xi\notin\operatorname{Dom}(\Psi_f^*)\).  Substitution of
\(\xi=0\) gives \(\mathcal C_f(0)=0\).  Moreover, for every \(h\in X\),
\begin{align*}
 D\mathcal C_f(0)[h]
 &=D\Psi_f^*(\eta_f)[h]-D\Psi_f^*(\eta_f)[h]=0.
\end{align*}
Finally, convexity of \(\Psi_f^*\) gives the following inequality
\begin{align*}
 \Psi_f^*(\eta_f+\xi)
 \geq\Psi_f^*(\eta_f)+D\Psi_f^*(\eta_f)[\xi],
\end{align*}
and hence \(\mathcal C_f(\xi)\geq0\). This implies that \(\mathcal C_f\) is centered at the
origin and has no linear part.  Under the precise integrability conditions stated in Assumption~\ref{ass:Hessian-Levy-sufficient-condition}, there are a measurable symmetric
nonnegative matrix field \(A_f(c)\) and a L\'evy measure
\(\nu_f^{\rm tilt}\) such that
\begin{align}\label{abstract-current-Levy-Khintchine}
 \mathcal C_f(\xi)
 ={}\frac12\int_{\mathscr C}
 \xi(c)\cdot A_f(c)\xi(c)\,\lambda(\d c)+\int_{\mathscr C\times(\R^m\setminus\{0\})}
 \bigl(e^{z\cdot\xi(c)}-1-z\cdot\xi(c)\bigr)
 \nu_f^{\rm tilt}(\d c,\d z).
\end{align}

When this representation holds on \(X\), the centered current at a frozen
state \(f\), denoted by \(\mathbb L^{\varepsilon,f}\), is a cylindrical
L\'evy process over \(X\): for every \(\xi_1,\ldots,\xi_n\in X\), the process
\[
 \bigl(\langle\mathbb L_t^{\varepsilon,f},\xi_1\rangle_{\mathscr C},
 \ldots,
 \langle\mathbb L_t^{\varepsilon,f},\xi_n\rangle_{\mathscr C}\bigr)
\]
is an \(\R^n\)-valued L\'evy process.  Thus
\(\mathbb L^{\varepsilon,f}\) is formally an \(X^*\)-valued current,
although a cylindrical process need not have an actual \(X^*\)-valued
realization.  Its one-dimensional cumulants satisfy
\begin{align}\label{conditional-Levy-current-definition}
 \varepsilon\log\E\exp\left\{
 \varepsilon^{-1}\left\langle
 \mathbb L_{t+h}^{\varepsilon,f}-\mathbb L_t^{\varepsilon,f},\xi
 \right\rangle_{\mathscr C}\right\}
 =h\mathcal C_f(\xi),
 \qquad t,h\geq0.
\end{align}
The image under \(\mathscr D^*\) is the cylindrical
\(\bigl(\operatorname{Dom}(\mathscr D)\bigr)^*\)-valued process defined by
\begin{align}\label{cylindrical-adjoint-current}
 \left\langle\mathscr D^*\mathbb L_t^{\varepsilon,f},p\right\rangle
 :=\left\langle\mathbb L_t^{\varepsilon,f},\mathscr Dp
 \right\rangle_{\mathscr C},
 \qquad p\in\operatorname{Dom}(\mathscr D).
\end{align}

With these conventions, the master fluctuating kinetic equation is the
formal identity in \(\bigl(\operatorname{Dom}(\mathscr D)\bigr)^*\).
Formally, \(F_t^\varepsilon\in L_+^1(\mathsf Z)\), and we assume that
\(\mathcal E(F_t^\varepsilon)\),
\(\frac{\delta\mathcal E}{\delta f}(F_t^\varepsilon)\), and
\(\mathcal T(F_t^\varepsilon)\) are well defined.  We then write
\begin{align}\label{unified-fluctuating-kinetic-equation}
 \d F_t^\varepsilon+\mathcal T(F_t^\varepsilon)\,\d t
 ={}&\mathscr D^*D\Psi_{F_t^\varepsilon}^*(\eta_{F_t^\varepsilon})\,\d t
 +\mathscr D^*\,\d\mathbb L_t^{\varepsilon,F_{t-}^\varepsilon}.
\end{align}
For a density \(f\) and a test function \(p\), write
\(\langle f,p\rangle:=\int_{\mathsf Z}f(z)p(z)\,\d z\) whenever this
integral is finite.  The weak form of
\eqref{unified-fluctuating-kinetic-equation} is then, for every
\(p\in\operatorname{Dom}(\mathscr D)\),
\begin{align}\label{unified-fluctuating-kinetic-equation-weak}
 \d\langle F_t^\varepsilon,p\rangle
 +\langle\mathcal T(F_t^\varepsilon),p\rangle\,\d t
 ={}&D\Psi_{F_t^\varepsilon}^*(\eta_{F_t^\varepsilon})
 [\mathscr Dp]\,\d t
 +\d\left\langle
 \mathbb L_t^{\varepsilon,F_{t-}^\varepsilon},\mathscr Dp
 \right\rangle_{\mathscr C}.
\end{align}
Throughout this paper, the equation serves as a structural identity, and the
solution theories of concrete models belong to Step~2 of
the program.  The space \(X\) and the domains introduced above provide a
functional setting in which the channel derivative, the dissipation
potential, and the current can be paired.  They do not constitute a
solution theory for \eqref{unified-fluctuating-kinetic-equation}, and no
well-posedness result for its model-dependent realizations is proved here.
Their later roles are more specific.  In Section~\ref{sec:Levy-realizability}
we freeze the state \(f\) and ask whether the Bregman remainder of
\(\Psi_f^*\) uniquely determines local L\'evy characteristics and a
stochastic-integral realization.  The uniqueness of the local
characteristics requires \(X\) to contain sufficiently many localized
channel fields.  In Section~\ref{sec:rate-identification}, the same space
\(X\) is fixed along an entire path, and the regular path classes are
defined by requiring the entropy forces, tilted forces, and channel
currents to belong to the corresponding domains.  Thus these path classes
depend on the chosen functional setting, even though their solution theory
is not developed in this paper.

For each example below, we specify
\(\mathcal E,\mathcal T,\mathscr D\), and \(\Psi_f^*\).  The calculations
are formal and are understood for smooth strictly positive densities with
enough decay to justify the integrations by parts.

\medskip
\noindent\textbf{The Vlasov--Fokker--Planck equation.}
Let
\begin{align*}
 \mathsf Z=\R_x^d\times\R_v^d,\qquad
 \mathscr C=\R_x^d\times\R_v^d,\qquad
 \lambda(\d c)=\d x\,\d v,
\end{align*}
and fix the channel space
\begin{align}\label{VFP-channel-space}
 X_{\mathrm{VFP}}
 :=L^\infty(\mathscr C;\R^d).
\end{align}
We work at states for which
\(\nabla_v\log f+v\in X_{\mathrm{VFP}}\).  Since \(f\in L^1(\mathsf Z)\),
the quadratic potential below and all of its derivatives are finite and
continuous on this fixed Banach space.  We take
\begin{align}\label{VFP-channel-data}
 \mathcal T_{\mathrm{VFP}}(f)
 &:=v\cdot\nabla_xf,&
 \mathscr D_{\mathrm{VFP}}p&:=\nabla_vp,\\
 \Psi_{f,\mathrm{VFP}}^*(\xi)
 &:=\int_{\mathsf Z}f(x,v)|\xi(x,v)|^2\,\d x\,\d v.
\end{align}

The force comes from the relative free energy
\begin{align}\label{VFP-free-energy}
 \mathcal E_{\mathrm{VFP}}(f)
 :=\int_{\mathsf Z}f\log\frac{f}{M}\,\d x\,\d v,
 \qquad
 M(v):=(2\pi)^{-d/2}e^{-|v|^2/2}.
\end{align}
Up to an irrelevant additive constant,
\[
 \frac{\delta\mathcal E_{\mathrm{VFP}}}{\delta f}
 =\log f+1+\frac{|v|^2}{2}.
\]
Since \(\mathscr D_{\mathrm{VFP}}=\nabla_v\),
\begin{align*}
 -\frac12\mathscr D_{\mathrm{VFP}}
 \frac{\delta\mathcal E_{\mathrm{VFP}}}{\delta f}
 =-\frac12(\nabla_v\log f+v),
\end{align*}
which explains the choice in \eqref{VFP-channel-data}.

For arbitrary channel fields \(\eta,\xi\), direct differentiation gives
\begin{align}\label{VFP-Psi-derivative}
 D\Psi_{f,\mathrm{VFP}}^*(\eta)[\xi]
 =\left.\frac{\d}{\d s}\right|_{s=0}
 \int_{\mathsf Z}f|\eta+s\xi|^2\,\d x\,\d v=2\int_{\mathsf Z}f\,\eta\cdot\xi\,\d x\,\d v.
\end{align}
Thus the derivative is represented by the channel-current density
\[
 D\Psi_{f,\mathrm{VFP}}^*(\eta)=2f\eta.
\]
The Legendre transform can also be computed pointwise, here we omit the calculation.  For a channel-current
density \(j:\mathsf Z\to\R^d\),
\begin{align}\label{VFP-primal-Psi}
 \Psi_{f,\mathrm{VFP}}(j)
 &:=\sup_{\eta\in X_{\mathrm{VFP}}}
 \left\{
 \int_{\mathsf Z}j\cdot\eta\,\d x\,\d v
 -\Psi_{f,\mathrm{VFP}}^*(\eta)
 \right\}=\frac14\int_{\mathsf Z}\frac{|j(x,v)|^2}{f(x,v)}
 \,\d x\,\d v.
\end{align}
The pointwise maximizer is \(\eta=j/(2f)\), equivalently
\(j=D\Psi_{f,\mathrm{VFP}}^*(\eta)\).  Thus the primal potential measures
the quadratic cost of a velocity current relative to the mobility \(2f\).
Evaluating it at the physical force,
\begin{align}\label{VFP-current-density}
 D\Psi_{f,\mathrm{VFP}}^*
 \left(-\frac12(\nabla_v\log f+v)\right)
 =-\nabla_vf-vf.
\end{align}
Because
\(\mathscr D_{\mathrm{VFP}}^*=-\operatorname{div}_v\), for every test
function \(p\),
 \begin{align*}
 \left\langle\mathscr D_{\mathrm{VFP}}^*
 D\Psi_{f,\mathrm{VFP}}^*
 \left(-\frac12(\nabla_v\log f+v)\right),p\right\rangle
 &=\int_{\mathsf Z}(-\nabla_vf-vf)\cdot\nabla_vp\,\d x\,\d v\\
 &=\left\langle\Delta_vf+\nabla_v\cdot(vf),p\right\rangle.
\end{align*}
This recovers the Ornstein--Uhlenbeck part of VFP.  

Since the potential is quadratic, its Bregman remainder is exact:
\begin{align}\label{VFP-centered-cumulant}
 \mathcal C_{f,\mathrm{VFP}}(\xi)
 &=
 \int_{\mathsf Z}f\bigl(
 |-\tfrac12(\nabla_v\log f+v)+\xi|^2
 -|-\tfrac12(\nabla_v\log f+v)|^2
 +(\nabla_v\log f+v)\cdot\xi\bigr)\,\d x\,\d v\notag\\
 &=\int_{\mathsf Z}f|\xi|^2\,\d x\,\d v.
\end{align}
Comparison with \eqref{abstract-current-Levy-Khintchine} uniquely gives
\begin{align*}
 A_f(x,v)=2f(x,v)I_d,
 \qquad
 \nu_f^{\rm tilt}=0.
\end{align*}
Thus the current has no jumps and is Gaussian.  If \(W\) is an
\(\R^d\)-valued \(L^2(\mathcal C)\)-cylindrical Wiener process, set
\begin{align*}
 \d\mathbb L_t^{\varepsilon,f}
 =\sqrt{2\varepsilon f}\,\d W_t.
\end{align*}
Conditionally on the frozen state \(f\), its scalar projection has quadratic
variation
\begin{align*}
 \d\left\langle
 \langle\mathbb L^{\varepsilon,f},\xi\rangle_{\mathscr C}
 \right\rangle_t
 =2\varepsilon\int_{\mathsf Z}f|\xi|^2\,\d x\,\d v\,\d t,
\end{align*}
so the Gaussian exponential formula yields
\eqref{conditional-Levy-current-definition}.  Applying
\(\mathscr D_{\mathrm{VFP}}^*=-\operatorname{div}_v\) produces the
conservative square-root noise.  Therefore the master equation becomes
\begin{align}\label{unified-gives-VFP}
 \d F_t^\varepsilon
 +v\cdot\nabla_xF_t^\varepsilon\,\d t
 =\left(\Delta_vF_t^\varepsilon
       +\nabla_v\cdot(vF_t^\varepsilon)\right)\d t
 +\sqrt{2\varepsilon}\,
   \operatorname{div}_v\left(\sqrt{F_t^\varepsilon}\,\d W_t\right),
\end{align}
up to the irrelevant replacement \(W\mapsto-W\).  This is the fluctuating
VFP--Dean--Kawasaki equation studied in \cite{HaoWuZimmerVFPDK}.

\medskip
\noindent\textbf{The Landau equation.}
Let
\(\mathsf Z=\R_x^d\times\R_v^d\).  The channels are particle pairs at
the same spatial point:
\begin{align*}
 \mathscr C
 :=\R_x^d\times\R_v^d\times\R_{v_*}^d,\qquad
 \lambda(\d c):=\d x\,\d v\,\d v_*.
\end{align*}
For this example, fix the ambient channel space
\begin{align}\label{Landau-channel-space}
 X_{\mathrm L}:=L^\infty(\mathscr C;\R^d).
\end{align}
Let
\[
 \Pi_{(v-v_*)^\perp}
 :=I-\frac{(v-v_*)\otimes(v-v_*)}{|v-v_*|^2},
 \qquad v\neq v_*,
\]
with an arbitrary value when \(v=v_*\), and let \(A(v-v_*)\geq0\) be the
scalar Landau collision weight.  The Landau channel derivative is the
projected difference gradient
\begin{align}\label{Landau-channel-derivative}
 \mathscr D_{\mathrm L}p(x,v,v_*)
 :=\Pi_{(v-v_*)^\perp}
 \bigl(\nabla_vp(x,v)-\nabla_{v_*}p(x,v_*)\bigr).
\end{align}
It is antisymmetric under \(v\leftrightarrow v_*\), and its adjoint is
defined by
\begin{align*}
 \langle\mathscr D_{\mathrm L}^*j,p\rangle
 =\int_{\mathscr C}j(x,v,v_*)\cdot
 \mathscr D_{\mathrm L}p(x,v,v_*)\,\d x\,\d v\,\d v_*.
\end{align*}
The range of \(\mathscr D_{\mathrm L}\) lies in the closed subspace of
\(X_{\mathrm L}\) formed by fields that are antisymmetric under
\(v\leftrightarrow v_*\) and tangent to \((v-v_*)^\perp\).  Components
outside this effective subspace are invisible to
\(\mathscr D_{\mathrm L}^*\).  We use the full ambient space in
\eqref{Landau-channel-space} so that localized channel fields remain
available, while all forces and currents appearing below lie in the
effective subspace.  The states are restricted so that the entropy force
in \eqref{Landau-force} is bounded and
\(A(v-v_*)f(x,v)f(x,v_*)\) is integrable on \(\mathscr C\).

The dissipated functional is the Boltzmann entropy
\begin{align}\label{Landau-entropy}
 \mathcal E_{\mathrm L}(f)
 :=\int_{\mathsf Z}f\log f\,\d x\,\d v,
 \qquad
 \frac{\delta\mathcal E_{\mathrm L}}{\delta f}
 =\log f+1.
\end{align}
Because the difference gradient annihilates the constant \(1\),
\eqref{entropy-force-convention} becomes
\begin{align}\label{Landau-force}
 -\frac12\mathscr D_{\mathrm L}\log f
 =-\frac12\Pi_{(v-v_*)^\perp}
 \bigl(\nabla_v\log f-\nabla_{v_*}\log f_*\bigr),
\end{align}
where \(f=f(x,v)\) and \(f_*=f(x,v_*)\).  Set
\begin{align}\label{Landau-Psi-data}
 \mathcal T_{\mathrm L}(f):=v\cdot\nabla_xf,\qquad \Psi_{f,\mathrm L}^*(\xi)
 :=\frac12\int_{\mathscr C}
 A(v-v_*)ff_*|\xi(x,v,v_*)|^2
 \,\d x\,\d v\,\d v_*.
\end{align}

For arbitrary effective pair fields \(\eta,\xi\),
\begin{align}\label{Landau-Psi-derivative}
 D\Psi_{f,\mathrm L}^*(\eta)[\xi]
 =\left.\frac{\d}{\d s}\right|_{s=0}
 \frac12\int_{\mathscr C}Aff_*|\eta+s\xi|^2\,\lambda(\d c)=\int_{\mathscr C}Aff_*\,\eta\cdot\xi\,\lambda(\d c).
\end{align}
This implies that 
\[
 D\Psi_{f,\mathrm L}^*(\eta)=Aff_*\eta.
\]
This pair-current formulation agrees with the gradient-flow geometry of the
Landau equation developed in
\cite{CarrilloDelgadinoDesvillettesWuLandau}.  The quadratic Legendre
transform of \(\Psi_{f,\mathrm L}^*\) on effective pair-current densities is
\begin{align}\label{Landau-primal-Psi}
 \Psi_{f,\mathrm L}(j)
 &=\frac12\int_{\mathscr C}
 \frac{|j(x,v,v_*)|^2}
 {A(v-v_*)f(x,v)f(x,v_*)}\,\d x\,\d v\,\d v_*.
\end{align}
The right-hand side is \(+\infty\) unless \(j\) is antisymmetric, tangent to
\((v-v_*)^\perp\), and vanishes wherever
\(A(v-v_*)f(x,v)f(x,v_*)=0\).  On the finite domain, the maximizing force
field is
\[
 \eta(x,v,v_*)
 =\frac{j(x,v,v_*)}{A(v-v_*)f(x,v)f(x,v_*)}.
\]
Consequently, at the physical force,
\[
 D\Psi_{f,\mathrm L}^*
 \left(-\frac12\mathscr D_{\mathrm L}\log f\right)
 =-\frac12Aff_*\mathscr D_{\mathrm L}\log f.
\]
For every test function \(p\),
\begin{align}\label{Landau-drift-identity}
 \left\langle
 \mathscr D_{\mathrm L}^*
 D\Psi_{f,\mathrm L}^*
 \left(-\frac12\mathscr D_{\mathrm L}\log f\right),p\right\rangle
 =-\frac12\int_{\mathscr C}
 Aff_*\mathscr D_{\mathrm L}\log f\cdot
 \mathscr D_{\mathrm L}p\,\lambda(\d c)
 =\langle Q_{\mathrm L}(f,f),p\rangle.
\end{align}
This is the weak Onsager form of the Landau collision operator and explains
both the force and the factor \(1/2\).

The quadratic Bregman remainder reads 
\begin{align}\label{Landau-centered-cumulant}
 \mathcal C_{f,\mathrm L}(\xi)
 =\frac12\int_{\mathscr C}Aff_*|\xi|^2\,\lambda(\d c).
\end{align}
Hence the L\'evy characteristics are
\begin{align*}
 A_f(x,v,v_*)=A(v-v_*)f(x,v)f(x,v_*)I_d,
 \qquad
 \nu_f^{\rm tilt}=0.
\end{align*}
The covariance therefore suggests a Gaussian pair current.  Let
\(W^{\mathrm a}\) be \(L^2(\mathcal C)\)-cylindrical Wiener process,
normalized as the identity on its antisymmetric tangential fields; it may
be obtained by antisymmetrizing a standard pair white noise and projecting
onto \((v-v_*)^\perp\).
Then
\begin{align*}
 \d\mathbb L_t^{\varepsilon,f}
 =\sqrt{\varepsilon A(v-v_*)f(x,v)f(x,v_*)}\,
   \d W_t^{\mathrm a}(x,v,v_*).
\end{align*}
For every antisymmetric test field \(\xi\), the quadratic variation of its
projection is
\begin{align*}
 \d\left\langle
 \langle\mathbb L^{\varepsilon,f},\xi\rangle_{\mathscr C}
 \right\rangle_t
 =\varepsilon\int_{\mathscr C}Aff_*|\xi|^2\,\lambda(\d c)\,\d t,
\end{align*}
which is twice \(\varepsilon\mathcal C_{f,\mathrm L}(\xi)\), as required by
the Gaussian exponential formula.  Applying
\(\mathscr D_{\mathrm L}^*\) gives the fluctuating Landau equation
\begin{align}\label{unified-gives-Landau}
 \d F_t^\varepsilon+v\cdot\nabla_xF_t^\varepsilon\,\d t
 =Q_{\mathrm L}(F_t^\varepsilon,F_t^\varepsilon)\,\d t
 +\mathscr D_{\mathrm L}^*\left(
 \sqrt{\varepsilon A(v-v_*)
 F_t^\varepsilon(x,v)F_t^\varepsilon(x,v_*)}\,
 \d W_t^{\mathrm a}\right).
\end{align}
The homogeneous equation is obtained by deleting \(x\) and
the transport term, in that case, the resulting equation agrees with the
fluctuating Landau equation studied in
\cite{DuongHeWuLandau}. 

\medskip
\noindent\textbf{The Boltzmann equation.}
Let
\(\mathsf Z=\R_x^d\times\R_v^d\) and
\begin{align*}
 \mathscr C
 :=\R_x^d\times\R_v^d\times\R_{v_*}^d\times\mathbb S^{d-1},
 \qquad
 \lambda(\d c):=\d x\,\d v\,\d v_*\,\d\sigma.
\end{align*}
Fix the scalar channel space
\begin{align}\label{Boltzmann-channel-space}
 X_{\mathrm B}:=L^\infty(\mathscr C).
\end{align}
A channel \(c=(x,v,v_*,\sigma)\) records variables included in the elastic binary collision
\((v,v_*)\mapsto(v',v_*')\), where
\begin{align}\label{Boltzmann-post-collisional-velocities}
 v'&:=\frac{v+v_*}{2}+\frac{|v-v_*|}{2}\,\sigma,
 &
 v_*'&:=\frac{v+v_*}{2}-\frac{|v-v_*|}{2}\,\sigma.
\end{align}
Consequently
\[
 v'+v_*'=v+v_*,
 \qquad
 |v'|^2+|v_*'|^2=|v|^2+|v_*|^2.
\]
Define the collision increment
\begin{align}\label{Boltzmann-channel-derivative}
 \mathscr D_{\mathrm B}p(c)
 :=p(x,v')+p(x,v_*')-p(x,v)-p(x,v_*),
\end{align}
and the corresponding signed state increment
\begin{align*}
 \mathcal J_c
 :=\delta_{(x,v')}+\delta_{(x,v_*')}
 -\delta_{(x,v)}-\delta_{(x,v_*)}.
\end{align*}
Thus
\(\mathscr D_{\mathrm B}p(c)=\langle\mathcal J_c,p\rangle\) and
\[
 \mathscr D_{\mathrm B}^*j
 =\int_{\mathscr C}j(c)\mathcal J_c\,\lambda(\d c)
\]
in the distributional sense.

The dissipated functional is again the Boltzmann entropy
\begin{align}\label{Boltzmann-entropy}
 \mathcal E_{\mathrm B}(f)
 :=\int_{\mathsf Z}f\log f\,\d x\,\d v,
 \qquad
 \frac{\delta\mathcal E_{\mathrm B}}{\delta f}=\log f+1.
\end{align}
For brevity, we define  
\begin{align}\label{Boltzmann-a-b}
 a(c):=f(x,v)f(x,v_*),
 \qquad
 b(c):=f(x,v')f(x,v_*').
\end{align}
Since \(\mathscr D_{\mathrm B}1=0\),
\begin{align}\label{Boltzmann-force-from-entropy}
 -\frac12\mathscr D_{\mathrm B}
 \frac{\delta\mathcal E_{\mathrm B}}{\delta f}
 =-\frac12\log\frac{b}{a}
 =\frac12\log\frac{a}{b}.
\end{align}
Hence the logarithmic ratio of the pre- and post-collisional densities is
exactly the entropy force along a collision channel.
We work at states for which this force belongs to \(X_{\mathrm B}\) and
the collision activity
\begin{align*}
 \int_{\mathscr C}B(v-v_*,\sigma)\sqrt{a(c)b(c)}\,\lambda(\d c)
\end{align*}
is finite.  Under these conditions, the cosh potential below is twice
continuously G\^ateaux differentiable on the fixed space
\(X_{\mathrm B}\).

Let \(B(v-v_*,\sigma)\) be a microreversible collision kernel, abbreviate
\(B(c):=B(v-v_*,\sigma)\), and set
\begin{align}\label{Boltzmann-Psi-data}
 \mathcal T_{\mathrm B}(f):=v\cdot\nabla_xf,\qquad \Psi_{f,\mathrm B}^*(\xi)
 :=\frac12\int_{\mathscr C}B(v-v_*,\sigma)
 \sqrt{a(c)b(c)}\bigl(\cosh\xi(c)-1\bigr)\lambda(\d c).
\end{align}
For arbitrary scalar channel fields \(\eta,\xi\),
\begin{align}\label{Boltzmann-Psi-derivative}
 D\Psi_{f,\mathrm B}^*(\eta)[\xi]
 =\left.\frac{\d}{\d s}\right|_{s=0}
 \frac12\int_{\mathscr C}B\sqrt{ab}
 \bigl(\cosh(\eta+s\xi)-1\bigr)\,\lambda(\d c)=\frac12\int_{\mathscr C}B\sqrt{ab}\,
 \sinh(\eta)\xi\,\lambda(\d c).
\end{align}
The scalar cosh--arsinh Legendre pair is given explicitly in
\cite[Eqs.~(1.2)--(1.3)]{PeletierSchlichtingCosh}, while the corresponding
dual dissipation potential for the Boltzmann collision operator appears in
\cite[Example~D, Eq.~(1.45)]{PeletierSchlichtingCosh}.  In our
normalization, the pointwise maximizer in the Legendre transform satisfies
\begin{align*}
 j=\frac12B\sqrt{ab}\,\sinh\xi,
 \qquad
 \xi=\operatorname{arsinh}\left(\frac{2j}{B\sqrt{ab}}\right).
\end{align*}
Substituting this value of \(\xi\), we obtain
\begin{align}\label{Boltzmann-primal-Psi}
 \Psi_{f,\mathrm B}(j)
 =\int_{\mathscr C}\Bigg[
 &j\,\operatorname{arsinh}\left(\frac{2j}{B\sqrt{ab}}\right)
 -\sqrt{j^2+\frac14B^2ab}
 +\frac12B\sqrt{ab}
 \Bigg]\,\lambda(\d c).
\end{align}

At points where \(B\sqrt{ab}=0\), the integrand is understood to be \(0\)
when \(j=0\) and \(+\infty\) otherwise.  This \(\operatorname{arsinh}\)-type
functional is the nonlinear cost of prescribing a net collision current.
At the physical force,
\begin{align*}
 e^{\frac12\log(a/b)}=\sqrt{\frac{a}{b}},
 \qquad
 e^{-\frac12\log(a/b)}=\sqrt{\frac{b}{a}},
 \qquad
 \sinh\left(\frac12\log\frac{a}{b}\right)
 =\frac{a-b}{2\sqrt{ab}}.
\end{align*}
Therefore
\(D\Psi_{f,\mathrm B}^*\left(\frac12\log\frac{a}{b}\right)\)
is represented by
the net collision-current density
\begin{align}\label{Boltzmann-current-density}
 j_f(c)=\frac14B(v-v_*,\sigma)\bigl(a(c)-b(c)\bigr).
\end{align}
For every test function \(p\),
\begin{align}\label{Boltzmann-drift-identity}
 \left\langle\mathscr D_{\mathrm B}^*
 D\Psi_{f,\mathrm B}^*
 \left(\frac12\log\frac{a}{b}\right),p\right\rangle
 =\frac14\int_{\mathscr C}B(a-b)\mathscr D_{\mathrm B}p\,\lambda(\d c)=\frac12\int_{\mathscr C}Ba\,\mathscr D_{\mathrm B}p\,\lambda(\d c)
 =\langle Q_{\mathrm B}(f,f),p\rangle.
\end{align}
The second equality uses the pre--post-collisional change of variables,
which interchanges \(a\) and \(b\) and changes
\(\mathscr D_{\mathrm B}p\) into
\(-\mathscr D_{\mathrm B}p\).  This calculation explains both the cosh
potential and its normalization.

We next compute the centered fluctuation functional.  Pointwise in \(c\),
\begin{align*}
 &\cosh\left(\frac12\log\frac{a}{b}+\xi\right)
 -\cosh\left(\frac12\log\frac{a}{b}\right)
 -\sinh\left(\frac12\log\frac{a}{b}\right)\xi\\
 &\quad=\frac12e^{\frac12\log(a/b)}(e^\xi-1-\xi)
 +\frac12e^{-\frac12\log(a/b)}(e^{-\xi}-1+\xi).
\end{align*}
Using
\(\sqrt{ab}\,e^{\frac12\log(a/b)}=a\) and
\(\sqrt{ab}\,e^{-\frac12\log(a/b)}=b\), we obtain
\begin{align}\label{Boltzmann-centered-cumulant}
 \mathcal C_{f,\mathrm B}(\xi)
 =\frac14\int_{\mathscr C}B
 \Bigl[
 a\bigl(e^\xi-1-\xi\bigr)
 +b\bigl(e^{-\xi}-1+\xi\bigr)
 \Bigr]\,\lambda(\d c).
\end{align}
In contrast with the L\'evy--Khintchine formula, we notice that the  Gaussian part \(A_f=0\), and the two jump directions are uniquely identified by  
\begin{align}\label{Boltzmann-Levy-measure}
 \nu_{f,\mathrm B}^{\rm tilt}(\d c,\d z)
 :=\frac14B(v-v_*,\sigma)
 \left[
 a(c)\delta_{+1}(\d z)
 +b(c)\delta_{-1}(\d z)
 \right]\lambda(\d c).
\end{align}
Then we can see that \(a/4\) is the forward collision activity and \(b/4\) is the backward
activity, both weighted by \(B\).  

To obtain an explicit realization, let
\(\widetilde{\mathcal N}_+\) and
\(\widetilde{\mathcal N}_-\) be independent compensated Poisson random
measures on
\([0,\infty)\times\mathscr C\times[0,\infty)\), each with intensity
\(\d t\,\lambda(\d c)\,\d u\).  At a frozen state \(f\), define 
\begin{align}\label{two-sided-Boltzmann-current}
 \langle\d\mathbb L_t^{\varepsilon,f},\xi\rangle_{\mathscr C}
 :={}&\varepsilon\int_{\mathscr C\times[0,\infty)}
 \xi(c)\boldsymbol1_{\{u\leq(4\varepsilon)^{-1}B(c)a(c)\}}
 \widetilde{\mathcal N}_+(\d t,\d c,\d u)\notag\\
 &-\varepsilon\int_{\mathscr C\times[0,\infty)}
 \xi(c)\boldsymbol1_{\{u\leq(4\varepsilon)^{-1}B(c)b(c)\}}
 \widetilde{\mathcal N}_-(\d t,\d c,\d u).
\end{align}
In the following, we will show that this is a suitable candidate of the noise by verifying its L\'evy cumulant.  If \(\widetilde{\mathcal N}\) is a compensated Poisson random
measure on
\([0,\infty)\times\mathscr C\times[0,\infty)\) with intensity
\(\d s\,\lambda(\d c)\,\d u\), then the Poisson exponential formula and its
compensated form \cite[Theorem~2.3.8(1) and Eq.~(2.9)]{ApplebaumLevy},
applied first on finite-intensity sets and then by approximation, yield,
whenever the right-hand side is finite,
\begin{align*}
 &\log\E\exp\left\{
 \int_t^{t+h}\int_{\mathscr C}\int_0^\infty
 g(s,c,u)\,\widetilde{\mathcal N}(\d s,\d c,\d u)
 \right\}\\
 &\qquad={}
 \int_t^{t+h}\int_{\mathscr C}\int_0^\infty
 \left(e^{g(s,c,u)}-1-g(s,c,u)\right)
 \,\d u\,\lambda(\d c)\,\d s.
\end{align*}
We write
\(
 \Delta_{t,h}\mathbb L^{\varepsilon,f}
 :=\mathbb L_{t+h}^{\varepsilon,f}-\mathbb L_t^{\varepsilon,f}
\).
In the exponential
\(\varepsilon^{-1}\langle
\Delta_{t,h}\mathbb L^{\varepsilon,f},\xi\rangle_{\mathscr C}\),
a forward jump contributes \(\xi(c)\), whereas a backward jump contributes
\(-\xi(c)\).  Independence of
\(\widetilde{\mathcal N}_+\) and
\(\widetilde{\mathcal N}_-\), followed by integration in the thinning
variable \(u\), therefore gives
\begin{align*}
 \log\E\exp\left\{
 \varepsilon^{-1}\left\langle
 \Delta_{t,h}\mathbb L^{\varepsilon,f},\xi
 \right\rangle_{\mathscr C}\right\} &={}
 \frac{h}{4\varepsilon}\int_{\mathscr C}B(c)a(c)
 \bigl(e^{\xi(c)}-1-\xi(c)\bigr)\,\lambda(\d c)\\
 &+
 \frac{h}{4\varepsilon}\int_{\mathscr C}B(c)b(c)
 \bigl(e^{-\xi(c)}-1+\xi(c)\bigr)\,\lambda(\d c).
\end{align*}
\begin{samepage}
Consequently,
\begin{align*}
 &\varepsilon\log\E\exp\left\{
 \varepsilon^{-1}\left\langle
 \Delta_{t,h}\mathbb L^{\varepsilon,f},\xi
 \right\rangle_{\mathscr C}\right\}\\
 &={}
 \frac h4\int_{\mathscr C}B(c)
 \left[a(c)\bigl(e^{\xi(c)}-1-\xi(c)\bigr)
 +b(c)\bigl(e^{-\xi(c)}-1+\xi(c)\bigr)\right]\lambda(\d c)\\
 &={}
 h\int_{\mathscr C\times\{-1,+1\}}
 \bigl(e^{z\xi(c)}-1-z\xi(c)\bigr)
 \nu_{f,\mathrm B}^{\rm tilt}(\d c,\d z)
 =h\mathcal C_{f,\mathrm B}(\xi).
\end{align*}
\end{samepage}
This implies exactly \eqref{conditional-Levy-current-definition}, with the
L\'evy measure \eqref{Boltzmann-Levy-measure} and cumulant
\eqref{Boltzmann-centered-cumulant}. After applying \(\mathscr D_{\mathrm B}^*\), a positive jump changes the
state by \(\varepsilon\mathcal J_c\), while a negative jump changes it by
\(-\varepsilon\mathcal J_c\), we next combine the two clocks in
\eqref{two-sided-Boltzmann-current} into one forward collision clock.  

Let \(\mathfrak r:\mathscr C\to\mathscr C\) be the pre--post-collisional
involution.  Away from the null set \(v=v_*\), it can be written as
\begin{align*}
 \mathfrak r(x,v,v_*,\sigma)
 :=\left(x,v',v_*',\frac{v-v_*}{|v-v_*|}\right).
\end{align*}
This reveals that the incoming velocities of the reversed channel are \((v',v_*')\),
and its outgoing velocities are \((v,v_*)\), and the value of \(\mathfrak r\) on \(v=v_*\) may be chosen arbitrarily.
 In the following, we will verify explicitly that the
geometric collision transformation preserves \(\lambda\).  Introduce the
center-of-mass and relative-velocity variables
\begin{align*}
 V:=\frac{v+v_*}{2},
 \qquad
 g:=v-v_*=\rho\vartheta,
 \qquad
 \rho:=|g|,
 \qquad
 \vartheta:=\frac{g}{|g|}\in\mathbb S^{d-1}.
\end{align*}
The absolute Jacobian of the linear change \((v,v_*)\leftrightarrow(V,g)\) is one,
and polar coordinates give
\begin{align}\label{collision-channel-measure-polar}
 \d v\,\d v_*
 =\d V\,\d g
 =\d V\,\rho^{d-1}\d\rho\,\d\vartheta.
\end{align}
In these coordinates,
\begin{align*}
 c
 =\left(x,V+\frac{\rho}{2}\vartheta,
             V-\frac{\rho}{2}\vartheta,\sigma\right),\qquad \mathfrak r c
 =\left(x,V+\frac{\rho}{2}\sigma,
             V-\frac{\rho}{2}\sigma,\vartheta\right).
\end{align*}
Thus \(\mathfrak r\) leaves \((x,V,\rho)\) unchanged and merely exchanges
the two angular variables \(\vartheta\) and \(\sigma\).  For every nonnegative
measurable \(\Phi:\mathscr C\to[0,\infty]\),
\begin{align}\label{collision-channel-measure-invariance-calculation}
\int_{\mathscr C}\Phi(\mathfrak r c)\,\lambda(\d c)&={}
 \int_{\R_x^d}\int_{\R_V^d}\int_0^\infty
 \int_{\mathbb S^{d-1}}\int_{\mathbb S^{d-1}}
 \Phi\left(x,V+\frac{\rho}{2}\sigma,
              V-\frac{\rho}{2}\sigma,\vartheta\right)
 \d\vartheta\,\d\sigma\,\rho^{d-1}\d\rho\,\d V\,\d x\notag\\
 &={}
 \int_{\R_x^d}\int_{\R_V^d}\int_0^\infty
 \int_{\mathbb S^{d-1}}\int_{\mathbb S^{d-1}}
 \Phi\left(x,V+\frac{\rho}{2}\vartheta,
              V-\frac{\rho}{2}\vartheta,\sigma\right)
 \d\vartheta\,\d\sigma\,\rho^{d-1}\d\rho\,\d V\,\d x=\int_{\mathscr C}\Phi(c)\,\lambda(\d c). 
\end{align}
This shows that 
\begin{align}\label{collision-channel-measure-invariance}
 \lambda\circ\mathfrak r^{-1}=\lambda.
\end{align}
The set \(g=0\), on which \(\vartheta\) is undefined, has
\(\d v\,\d v_*\)-measure zero and does not affect the calculation.
Microreversibility additionally gives \(B(\mathfrak r c)=B(c)\), Moreover, we have the following properties 
\begin{align*}
 \mathfrak r^2c&=c,&
 \lambda\circ\mathfrak r^{-1}&=\lambda,&
 B(\mathfrak r c)&=B(c),&
 a(\mathfrak r c)&=b(c),&
 b(\mathfrak r c)&=a(c),\notag\\
 \mathcal J_{\mathfrak r c}&=-\mathcal J_c,&
 \mathscr D_{\mathrm B}p(\mathfrak r c)&=-\mathscr D_{\mathrm B}p(c).
\end{align*}

\begin{samepage}
Apply \(\mathscr D_{\mathrm B}^*\) to
\eqref{two-sided-Boltzmann-current} and test the resulting state-space noise
against \(p\).  Its increment is
\begin{align*}
 \d M_t^{\varepsilon,p}
 ={}&\varepsilon\int_{\mathscr C\times[0,\infty)}
 \mathscr D_{\mathrm B}p(c)
 \boldsymbol1_{\{u\leq(4\varepsilon)^{-1}B(c)a(c)\}}
 \widetilde{\mathcal N}_+(\d t,\d c,\d u)\\
 &-\varepsilon\int_{\mathscr C\times[0,\infty)}
 \mathscr D_{\mathrm B}p(c)
 \boldsymbol1_{\{u\leq(4\varepsilon)^{-1}B(c)b(c)\}}
 \widetilde{\mathcal N}_-(\d t,\d c,\d u).
\end{align*}
\end{samepage}
In the following, we focus on the second integral.  Since the incoming
velocities of \(\mathfrak r c\) are \((v',v_*')\) and its outgoing velocities
are \((v,v_*)\), the definition \eqref{Boltzmann-channel-derivative} gives
the explicit calculation
\begin{align}\label{reversed-channel-derivative-calculation}
 \mathscr D_{\mathrm B}p(\mathfrak r c)
 &=p(x,v)+p(x,v_*)-p(x,v')-p(x,v_*')\notag\\
 &=-\bigl[p(x,v')+p(x,v_*')-p(x,v)-p(x,v_*)\bigr]=-\mathscr D_{\mathrm B}p(c).
\end{align}
Similarly, evaluating the backward activity at the reversed channel yields
\begin{align}\label{reversed-channel-activity-calculation}
 b(\mathfrak r c)
 &=f(x,v)f(x,v_*)=a(c),
 &B(\mathfrak r c)&=B(c).
\end{align}

To define the transformed Poisson measure precisely, set
\begin{align*}
 E:=[0,\infty)\times\mathscr C\times[0,\infty),
 \qquad
 \mu(\d s,\d c,\d u):=\d s\,\lambda(\d c)\,\d u,
\end{align*}
and define the measurable involution
\(\mathfrak R:E\to E\) by
\(\mathfrak R(s,c,u)=(s,\mathfrak r c,u)\).  For every measurable
\(A\subset E\) with \(\mu(A)<\infty\), define the image random measure and
its compensation by
\begin{align}\label{reversed-Poisson-measure-definition}
 \mathcal N_-^{\mathfrak r}(A)
 :=\mathcal N_-(\mathfrak R^{-1}A),
 \qquad
 \widetilde{\mathcal N}_-^{\mathfrak r}(A)
 :=\mathcal N_-^{\mathfrak r}(A)-\mu(A).
\end{align}
Because \(\mathfrak R\) changes only the channel variable, the preceding
invariance of \(\lambda\) implies, for every nonnegative measurable
\(H:E\to[0,\infty]\),
\begin{align}\label{Poisson-intensity-invariance-calculation}
 \int_EH\,\d(\mu\circ\mathfrak R^{-1})
 &=\int_0^\infty\int_{\mathscr C}\int_0^\infty
 H(s,\mathfrak r c,u)\,\d u\,\lambda(\d c)\,\d s\notag\\
 &=\int_0^\infty\int_{\mathscr C}\int_0^\infty
 H(s,c,u)\,\d u\,\lambda(\d c)\,\d s
 =\int_EH\,\d\mu.
\end{align}
Consequently,
\begin{align}\label{Poisson-intensity-invariance}
 \mu\circ\mathfrak R^{-1}=\mu.
\end{align}
This shows that
\(\mathcal N_-^{\mathfrak r}\) is a Poisson random measure with intensity
\(\mu\), and hence
\(\widetilde{\mathcal N}_-^{\mathfrak r}\) is a compensated Poisson random
measure with the same intensity as
\(\widetilde{\mathcal N}_-\).  Moreover, it remains independent of
\(\widetilde{\mathcal N}_+\), since it is a measurable function of
\(\mathcal N_-\).  From \eqref{reversed-Poisson-measure-definition},
for every admissible integrand \(G\),
\begin{align}\label{reversed-Poisson-change-of-variables}
 \int_EG(s,c,u)\,\widetilde{\mathcal N}_-(\d s,\d c,\d u)
 =\int_EG(s,\mathfrak r c,u)\,
 \widetilde{\mathcal N}_-^{\mathfrak r}(\d s,\d c,\d u).
\end{align}

Denote the second term in \(\d M_t^{\varepsilon,p}\) by
\begin{align*}
 I_-(\d t):={}&-\varepsilon
 \int_{\mathscr C\times[0,\infty)}
 \mathscr D_{\mathrm B}p(c)
 \boldsymbol1_{\{u\leq(4\varepsilon)^{-1}B(c)b(c)\}}
 \widetilde{\mathcal N}_-(\d t,\d c,\d u).
\end{align*}
Applying \eqref{reversed-Poisson-change-of-variables} with
\(c\mapsto\mathfrak r c\), and then using
\eqref{reversed-channel-derivative-calculation}--
\eqref{reversed-channel-activity-calculation}, gives, on the same probability
space,
\begin{align*}
 I_-(\d t)
 ={}&-\varepsilon
 \int_{\mathscr C\times[0,\infty)}
 \mathscr D_{\mathrm B}p(\mathfrak r c)
 \boldsymbol1_{\{u\leq(4\varepsilon)^{-1}
 B(\mathfrak r c)b(\mathfrak r c)\}}
 \widetilde{\mathcal N}_-^{\mathfrak r}(\d t,\d c,\d u)\\
 ={}&\varepsilon
 \int_{\mathscr C\times[0,\infty)}
 \mathscr D_{\mathrm B}p(c)
 \boldsymbol1_{\{u\leq(4\varepsilon)^{-1}B(c)a(c)\}}
 \widetilde{\mathcal N}_-^{\mathfrak r}(\d t,\d c,\d u).
\end{align*}
Thus the original minus sign is cancelled exactly by the antisymmetry of
the collision increment under pre--post reversal.  Adding the unchanged
first integral now yields the pathwise identity
\begin{align*}
 \d M_t^{\varepsilon,p}
 ={}&
 \varepsilon\int_{\mathscr C\times[0,\infty)}
 \mathscr D_{\mathrm B}p(c)
 \boldsymbol1_{\{u\leq(4\varepsilon)^{-1}B(c)a(c)\}}
 \bigl(\widetilde{\mathcal N}_+
       +\widetilde{\mathcal N}_-^{\mathfrak r}\bigr)
 (\d t,\d c,\d u)\\
 \stackrel{\mathrm{law}}={}&
 \varepsilon\int_{\mathscr C\times[0,\infty)}
 \mathscr D_{\mathrm B}p(c)
 \boldsymbol1_{\{u\leq(2\varepsilon)^{-1}B(c)a(c)\}}
 \widetilde{\mathcal N}(\d t,\d c,\d u).
\end{align*}
The final step is an equality in law.  Here
\(\mathcal N_++\mathcal N_-^{\mathfrak r}\) is a Poisson random measure with
intensity \(2\mu\), and
\begin{align*}
 \widetilde{\mathcal N}_+
 +\widetilde{\mathcal N}_-^{\mathfrak r}
 =\mathcal N_++\mathcal N_-^{\mathfrak r}-2\mu.
\end{align*}
Consequently, thinning this superposed measure at rate
\((4\varepsilon)^{-1}B(c)a(c)\) has the same law as thinning a Poisson random
measure of intensity \(\mu\) at the doubled rate
\((2\varepsilon)^{-1}B(c)a(c)\).  Since the equality holds jointly for every
finite collection of state-space test functions \(p\), it identifies the
corresponding distribution-valued noises, then the master equation takes only the forward form
\begin{align}\label{unified-gives-Boltzmann}
 \d F_t^\varepsilon+v\cdot\nabla_xF_t^\varepsilon\,\d t
 ={}&Q_{\mathrm B}(F_t^\varepsilon,F_t^\varepsilon)\,\d t\notag\\
 &+\varepsilon\int_{\mathscr C\times[0,\infty)}
 \mathcal J_c
 \boldsymbol1_{\{u\leq(2\varepsilon)^{-1}
 B(v-v_*,\sigma)F_{t-}^\varepsilon(x,v)
 F_{t-}^\varepsilon(x,v_*)\}}
 \widetilde{\mathcal N}(\d t,\d c,\d u).
\end{align}
This is the
forward-collision representation of the Poissonian stochastic Boltzmann
equation introduced in \cite{WuPoissonBoltzmann}.

\medskip
\noindent\textbf{The kinetic \(\alpha\)-stable Dean--Kawasaki equation.}

A fractional velocity operator gives a linear approximation of the
Boltzmann collision mechanism.  Such operators arise, for example, as
anomalous diffusion limits of linear Boltzmann equations with heavy-tailed
equilibria \cite{MelletMischlerMouhot2011}.  This leads to a kinetic model
in which particles move freely in position and their velocities follow
independent isotropic symmetric \(\alpha\)-stable processes, with
\(\alpha\in(0,2)\).

Set \(\R_*^d:=\R^d\setminus\{0\}\) and
\begin{align*}
 \nu_\alpha(\d z):=c_{d,\alpha}|z|^{-d-\alpha}\,\d z.
\end{align*}
The normalization constant \(c_{d,\alpha}>0\) is chosen so that the
velocity generator of a symmetric \(\alpha\)-stable process is
\begin{align}\label{stable-generator-scope}
 \mathcal L_\alpha\phi(v)
 :=\int_{\R_*^d}\left[
 \phi(v+z)-\phi(v)
 -\boldsymbol1_{\{|z|\leq1\}}z\cdot\nabla_v\phi(v)
 \right]\nu_\alpha(\d z)=-(-\Delta_v)^{\alpha/2}\phi(v).
\end{align}
Since \(\nu_\alpha\) is symmetric,
\(\mathcal L_\alpha^*=\mathcal L_\alpha\).  Let
\(\mathfrak N(\d t,\d x,\d v,\d z,\d u)\) be a Poisson random measure with
intensity
\(\d t\,\d x\,\d v\,\nu_\alpha(\d z)\,\d u\), and let
\(\widetilde{\mathfrak N}\) denote its compensation.  The formal kinetic
\(\alpha\)-stable Dean--Kawasaki equation is
\begin{align}\label{alpha-stable-Dean-Kawasaki-scope}
 \d F_t^\varepsilon
 +v\cdot\nabla_xF_t^\varepsilon\,\d t
 ={}&\mathcal L_\alpha^*F_t^\varepsilon\,\d t\notag\\
 &+\varepsilon
 \int_{\R_x^d\times\R_v^d\times\R_*^d\times(0,\infty)}
 \bigl(\delta_{(x,v+z)}-\delta_{(x,v)}\bigr)
 \boldsymbol1_{\{u\leq F_{t-}^\varepsilon(x,v)/\varepsilon\}}
 \widetilde{\mathfrak N}(\d t,\d x,\d v,\d z,\d u).
\end{align}
One jump transfers mass \(\varepsilon\) from \((x,v)\) to
\((x,v+z)\), at rate
\(\varepsilon^{-1}F_{t-}^\varepsilon(x,v)\nu_\alpha(\d z)\).

This model is also covered by the structure of the master equation.  Its
channel space, reference measure, and channel derivative can be defined by 
\begin{align}\label{stable-channel-structure}
 \mathscr C_\alpha&:=\R_x^d\times\R_v^d\times\R_*^d,\qquad
 \lambda_\alpha(\d x,\d v,\d z)
 :=\d x\,\d v\,\nu_\alpha(\d z),\notag\\
 \mathscr D_\alpha p(x,v,z)&:=p(x,v+z)-p(x,v).
\end{align}
Since the stable L\'evy measure has infinite mass near \(z=0\), we use the
weighted channel space
\begin{align}\label{stable-channel-space}
 X_\alpha
 :=\left\{\xi\in L^0(\mathscr C_\alpha):
 \|\xi\|_{X_\alpha}
 :=\mathop{\operatorname{ess\,sup}}_{(x,v,z)\in\mathscr C_\alpha}
 \frac{|\xi(x,v,z)|}{1\wedge|z|}<\infty\right\}.
\end{align}
This is a Banach space that does not depend on the state.  In particular,
if \(p\) and \(\nabla_vp\) are bounded, then
\begin{align*}
 |\mathscr D_\alpha p(x,v,z)|
 \leq \max\{\|\nabla_vp\|_\infty,2\|p\|_\infty\}
 (1\wedge|z|),
\end{align*}
and hence \(\mathscr D_\alpha p\in X_\alpha\).  The same weight is
compatible with the small-jump integrability
\begin{align*}
 \int_{\R_*^d}(1\wedge|z|^2)\,\nu_\alpha(\d z)<\infty.
\end{align*}
For a strictly positive density \(\rho\), define
\begin{align}\label{stable-entropy-dissipation-pair}
 \mathcal E_\alpha(\rho)
 :={}&\int_{\R_x^d\times\R_v^d}
 \rho(x,v)\log\rho(x,v)\,\d x\,\d v,\notag\\
 \Psi_{\rho,\alpha}^*(\xi)
 :={}&\int_{\mathscr C_\alpha}
 \sqrt{\rho(x,v)\rho(x,v+z)}
 \bigl(\cosh\xi(x,v,z)-1\bigr)\,
 \lambda_\alpha(\d x,\d v,\d z).
\end{align}

For arbitrary scalar channel fields \(\eta\) and \(\xi\), differentiation
under the integral gives
\begin{align}\label{stable-dual-dissipation-derivative}
 D\Psi_{\rho,\alpha}^*(\eta)[\xi]
 :=
 \left.\frac{\d}{\d s}\right|_{s=0}
 \Psi_{\rho,\alpha}^*(\eta+s\xi)=
 \int_{\mathscr C_\alpha}
 \sqrt{\rho(x,v)\rho(x,v+z)}\,
 \sinh\eta(x,v,z)\,\xi(x,v,z)\,
 \lambda_\alpha(\d x,\d v,\d z).
\end{align}
Thus the channel-current density represented by this derivative is
\begin{align*}
 D\Psi_{\rho,\alpha}^*(\eta)(x,v,z)
 =\sqrt{\rho(x,v)\rho(x,v+z)}\,\sinh\eta(x,v,z).
\end{align*}
Let \(\Psi_{\rho,\alpha}\) be the Legendre transform of
\(\Psi_{\rho,\alpha}^*\).  More explicitly, for a current \(j\),
\begin{align*}
 \Psi_{\rho,\alpha}(j)
 =\int_{\mathscr C_\alpha}\Bigg[
 &j(x,v,z)\operatorname{arsinh}
 \left(\frac{j(x,v,z)}{\sqrt{\rho(x,v)\rho(x,v+z)}}\right)\\
 &-\sqrt{j(x,v,z)^2+\rho(x,v)\rho(x,v+z)}
 +\sqrt{\rho(x,v)\rho(x,v+z)}
 \Bigg]\,\lambda_\alpha(\d x,\d v,\d z).
\end{align*}
This follows from the scalar cosh--arsinh Legendre pair in
\cite[Eqs.~(1.2)--(1.3)]{PeletierSchlichtingCosh}, applied pointwise with
activity \(\sqrt{\rho(x,v)\rho(x,v+z)}\).

The entropy force on a channel is
\(\left(-\frac12\mathscr D_\alpha
 \frac{\delta\mathcal E_\alpha}{\delta\rho}\right)(x,v,z)
 =\frac12\log\frac{\rho(x,v)}{\rho(x,v+z)}\). Consequently,
\begin{align*}
 D\Psi_{\rho,\alpha}^*
 \left(-\frac12\mathscr D_\alpha
 \frac{\delta\mathcal E_\alpha}{\delta\rho}\right)(x,v,z)
 =\sqrt{\rho(x,v)\rho(x,v+z)}
 \sinh\left(\frac12\log\frac{\rho(x,v)}{\rho(x,v+z)}\right)=\frac12\bigl(\rho(x,v)-\rho(x,v+z)\bigr).
\end{align*}
We compute its image under \(\mathscr D_\alpha^*\) first with the jump
integral restricted to \(|z|>r\), where \(r\in(0,1)\).  For every test
function \(p\),
\begin{align*}
 &\frac12\int_{\R^d\times\{|z|>r\}}
 \bigl(\rho(x,v)-\rho(x,v+z)\bigr)
 \bigl(p(x,v+z)-p(x,v)\bigr)\,
 \d v\,\nu_\alpha(\d z)\\
 =&
 \frac12\int_{\R^d\times\{|z|>r\}}
 \rho(x,v)\bigl(p(x,v+z)-p(x,v)\bigr)\,
 \d v\,\nu_\alpha(\d z)\\
 &+
 \frac12\int_{\R^d\times\{|z|>r\}}
 \rho(x,\bar v)\bigl(p(x,\bar v+\bar z)-p(x,\bar v)\bigr)\,
 \d\bar v\,\nu_\alpha(\d\bar z)\\
 =&
 \int_{\R^d\times\{|z|>r\}}
 \rho(x,v)\bigl(p(x,v+z)-p(x,v)\bigr)\,
 \d v\,\nu_\alpha(\d z).
\end{align*}
In the second term we used
\(\bar v=v+z\) and \(\bar z=-z\).  Under this change,
\(\rho(x,v+z)=\rho(x,\bar v)\), while
\[
 p(x,v+z)-p(x,v)
 =p(x,\bar v)-p(x,\bar v+\bar z)
 =-\bigl(p(x,\bar v+\bar z)-p(x,\bar v)\bigr).
\]
The minus sign cancels the minus sign in front of the
\(\rho(x,v+z)\)-term, and the measure is unchanged because
\(\nu_\alpha\) is symmetric.

Symmetry cancels the compensation term in
\eqref{stable-generator-scope}, since
\(\int_{\{r<|z|\leq1\}}z\,\nu_\alpha(\d z)=0\).  Hence
\begin{align*}
 &\int_{\R^d\times\{|z|>r\}}
 \rho(x,v)\bigl(p(x,v+z)-p(x,v)\bigr)\,
 \d v\,\nu_\alpha(\d z)\\
 &={}
 \int_{\R^d}\rho(x,v)\int_{\{|z|>r\}}
 \left[
 p(x,v+z)-p(x,v)
 -\boldsymbol1_{\{|z|\leq1\}}z\cdot\nabla_vp(x,v)
 \right]\nu_\alpha(\d z)\,\d v.
\end{align*}
Letting \(r\downarrow0\) and using
\eqref{stable-generator-scope}, we obtain
\begin{align*}
 \frac12\int_{\mathscr C_\alpha}
 \bigl(\rho(x,v)-\rho(x,v+z)\bigr)
 \mathscr D_\alpha p(x,v,z)\,\lambda_\alpha(\d x,\d v,\d z)
 =\int_{\R_x^d\times\R_v^d}
 \rho(x,v)(\mathcal L_\alpha p)(x,v)\,\d x\,\d v.
\end{align*}
This equation identifies the dissipative drift with
\(\mathcal L_\alpha^*\rho\).  Together with the kinetic transport
\(\mathcal T(\rho)=v\cdot\nabla_x\rho\), this gives the deterministic part of
\eqref{alpha-stable-Dean-Kawasaki-scope}.

We next identify the noise.  Set
\(\eta(x,v,z)=\frac12\log(\rho(x,v)/\rho(x,v+z))\).  For an arbitrary scalar
channel field \(\xi\), the pointwise identity
\begin{align*}
 \sqrt{\rho(x,v)\rho(x,v+z)}
 \bigl[\cosh(\eta+\xi)-\cosh\eta-\sinh(\eta)\,\xi\bigr]=
 \frac{\rho(x,v)}2\bigl(e^\xi-1-\xi\bigr)
 +\frac{\rho(x,v+z)}2\bigl(e^{-\xi}-1+\xi\bigr).
\end{align*}
Consequently, the centered Bregman remainder is
\begin{align}\label{stable-centered-cumulant-scope}
 &\Psi_{\rho,\alpha}^*\left(
 -\frac12\mathscr D_\alpha
 \frac{\delta\mathcal E_\alpha}{\delta\rho}+\xi\right)
 -\Psi_{\rho,\alpha}^*\left(
 -\frac12\mathscr D_\alpha
 \frac{\delta\mathcal E_\alpha}{\delta\rho}\right)
 -D\Psi_{\rho,\alpha}^*\left(
 -\frac12\mathscr D_\alpha
 \frac{\delta\mathcal E_\alpha}{\delta\rho}\right)[\xi]\notag\\
 &={}
 \frac12\int_{\mathscr C_\alpha}
 \rho(x,v)\bigl(e^{\xi(x,v,z)}-1-\xi(x,v,z)\bigr)\,
 \lambda_\alpha(\d x,\d v,\d z)\notag\\
 &\quad{}+\frac12\int_{\mathscr C_\alpha}
 \rho(x,v+z)\bigl(e^{-\xi(x,v,z)}-1+\xi(x,v,z)\bigr)\,
 \lambda_\alpha(\d x,\d v,\d z).
\end{align}
This L\'evy--Khintchine exponent can first be realized by two independent
compensated Poisson currents.  Let
\(\widetilde{\mathfrak N}_+\) and
\(\widetilde{\mathfrak N}_-\) have the common reference intensity
\(\d t\,\d x\,\d v\,\nu_\alpha(\d z)\,\d u\).  Conditional on a frozen state \(\rho\), set
\begin{align}\label{stable-two-sided-current-scope}
 \left\langle\d\mathbb L_t^{\varepsilon,\rho},\xi\right\rangle
 :={}&\varepsilon\int_{\mathscr C_\alpha\times(0,\infty)}
 \xi(x,v,z)\boldsymbol1_{\{u\leq\rho(x,v)/(2\varepsilon)\}}
 \widetilde{\mathfrak N}_+(\d t,\d x,\d v,\d z,\d u)\notag\\
 &-\varepsilon\int_{\mathscr C_\alpha\times(0,\infty)}
 \xi(x,v,z)\boldsymbol1_{\{u\leq\rho(x,v+z)/(2\varepsilon)\}}
 \widetilde{\mathfrak N}_-(\d t,\d x,\d v,\d z,\d u).
\end{align}
The two terms have channel jumps \(+1\) and \(-1\), respectively, and
their cumulant is precisely \eqref{stable-centered-cumulant-scope}.

As in the reduction
\eqref{two-sided-Boltzmann-current}--\eqref{unified-gives-Boltzmann}, the
two channel noises combine after applying \(\mathscr D_\alpha^*\).
Testing a test function \(p\) gives
\begin{align*}
 \left\langle\mathscr D_\alpha^*\d\mathbb L_t^{\varepsilon,\rho},p\right\rangle
 ={}&\varepsilon\int_{\mathscr C_\alpha\times(0,\infty)}
 \mathscr D_\alpha p(x,v,z)
 \boldsymbol1_{\{u\leq\rho(x,v)/(2\varepsilon)\}}
 \widetilde{\mathfrak N}_+(\d t,\d x,\d v,\d z,\d u)\\
 &-\varepsilon\int_{\mathscr C_\alpha\times(0,\infty)}
 \mathscr D_\alpha p(x,v,z)
 \boldsymbol1_{\{u\leq\rho(x,v+z)/(2\varepsilon)\}}
 \widetilde{\mathfrak N}_-(\d t,\d x,\d v,\d z,\d u).
\end{align*}
In the last integral, we have used
\(\bar v=v+z\) and \(\bar z=-z\).  Furthermore, we have 
\begin{align*}
 -\mathscr D_\alpha p(x,v,z)
 &=p(x,v)-p(x,v+z)
 =p(x,\bar v+\bar z)-p(x,\bar v)
 =\mathscr D_\alpha p(x,\bar v,\bar z),\\
 \rho(x,v+z)&=\rho(x,\bar v),\qquad
 \d v\,\nu_\alpha(\d z)
 =\d\bar v\,\nu_\alpha(\d\bar z).
\end{align*}
The image of \(\widetilde{\mathfrak N}_-\) under this transformation is
therefore a compensated Poisson random measure with the same reference
intensity as \(\widetilde{\mathfrak N}_+\), and it remains independent of
\(\widetilde{\mathfrak N}_+\).  Consequently, it follows that 
\begin{align*}
 \left\langle\mathscr D_\alpha^*\d\mathbb L_t^{\varepsilon,\rho},p\right\rangle
 \stackrel{\mathrm{law}}{={}}&
 \varepsilon\int_{\mathscr C_\alpha\times(0,\infty)}
 \mathscr D_\alpha p(x,v,z)
 \boldsymbol1_{\{u\leq\rho(x,v)/\varepsilon\}}
 \widetilde{\mathfrak N}(\d t,\d x,\d v,\d z,\d u).
\end{align*}
Thus the two-sided realization \eqref{stable-two-sided-current-scope} is
equal in law to the single noise in
\eqref{alpha-stable-Dean-Kawasaki-scope} on state-space observables, and then the kinetic $\alpha$-stable Dean-Kawasaki equation \eqref{alpha-stable-Dean-Kawasaki-scope} can be derived from the master fluctuating kinetic equation.

\section{L\'evy realization of dissipation potentials}
\label{sec:Levy-realizability}

In this section, we turn from the formal structure of Section~2 to the first rigorous result.
The state \(f\), and hence the channel force
\(-\frac12\mathscr D\frac{\delta\mathcal E}{\delta f}\), is frozen
throughout this section, so the question reduces to a representation problem
for functionals on channel fields.  At a frozen state \(f\), a centered local L\'evy current has two
characteristics: a Gaussian covariance density \(A_f(c)\) and a jump
intensity \(\nu_f(\d c,\d z)\), where \(z\in\R^m\) is the amplitude of one
channel event.  Its log moment-generating functional has the form
\begin{align*}
 \frac12\int_{\mathscr C}\xi\cdot A_f\xi\,\lambda(\d c)
 +\int_{\mathscr C\times(\R^m\setminus\{0\})}
 \bigl(e^{z\cdot\xi}-1-z\cdot\xi\bigr)\,\nu_f(\d c,\d z).
\end{align*}

We now impose the integrability needed for the
cumulant, its force shift, and the derivatives used below. Recall that we use the Banach channel space \(X\) and its dual \(X^*\) from
\eqref{Banach-channel-space}, together with the operators
\(\mathscr D,\mathscr D^*\) defined in
\eqref{channel-derivative-domain} and \eqref{channel-adjoint-domain}, and
related by \eqref{channel-adjoint-pairing}.
The dual dissipation potential and its effective domain are given in
\eqref{dual-dissipation-potential-domain}, and the frozen entropy force
\(\eta_f\in X\) is defined in \eqref{entropy-force-convention}.  

\begin{definition}[Admissible L\'evy channel data]
\label{def:symmetric-Levy-data}
Fix \(f\in L_+^1(\mathsf Z)\) such that the entropy force
\begin{align*}
 \eta_f:=-\frac12\mathscr D
 \frac{\delta\mathcal E}{\delta f}(f)
\end{align*}
is well defined and belongs to \(X\).  Admissible L\'evy channel data at
the frozen state \(f\) consist of
\begin{enumerate}
\item a measurable matrix field
      \(A_f:\mathscr C\to\R^{m\times m}\) such that, for every
      \(c\in\mathscr C\), the matrix \(A_f(c)\) is symmetric and
      nonnegative;
\item a measure \(\nu_f(\d c,\d z)\) on
      \(\mathscr C\times(\R^m\setminus\{0\})\) whose channel marginal
      does not charge \(\lambda\)-null sets, in the sense that
      \(\lambda(K)=0\) implies
      \(\nu_f(K\times(\R^m\setminus\{0\}))=0\).
\end{enumerate}
We assume that \(A_f\) is locally integrable and that, for every measurable
\(K\subset\mathscr C\) with \(\lambda(K)<\infty\),
\begin{align}\label{local-Levy-integrability}
 \int_{K\times(\R^m\setminus\{0\})}(1\wedge|z|^2)
 \,\nu_f(\d c,\d z)<\infty.
\end{align}
In addition, for every \(\xi\in X\), assume
\begin{align}\label{Levy-exponential-integrability}
 &\int_{\mathscr C}\xi(c)\cdot A_f(c)\xi(c)\,\lambda(\d c)<\infty,\notag\\
 &\int_{\mathscr C\times(\R^m\setminus\{0\})}
 |z\cdot\xi(c)|^2
 \exp\left\{|z\cdot\eta_f(c)|+|z\cdot\xi(c)|\right\}
 \,\nu_f(\d c,\d z)<\infty.
\end{align}
The data are called \emph{symmetric} if \(\nu_f\) is invariant under
reflection of the jump variable, more precisely, for every measurable
\(B\subset\mathscr C\times(\R^m\setminus\{0\})\),
\begin{align*}
 \nu_f(B)=\nu_f\bigl(\{(c,-z):(c,z)\in B\}\bigr).
\end{align*}
\end{definition}

The central question of this section is when the centered functional \(\mathcal C_f\) determined by
\(\Psi_f^*\) has a L\'evy--Khintchine representation.  Starting from
symmetric admissible data \((A_f,\nu_f)\), the entropy force produces the
jump intensity in that representation by an exponential tilt of \(\nu_f\),
as defined later in \eqref{Hessian-criterion-tilt}.
Condition~\eqref{Levy-exponential-integrability} guarantees that the
integrals in this tilted representation, together with their derivatives in
every direction \(\xi\in X\), are finite.  This convenient condition allows
us to use the tilt under a single set of moment assumptions later on.

We now give a condition directly on the dissipation potential.  For a
symmetric matrix-valued function
\(H:\R^m\to\R^{m\times m}\), we say that \(H\) is
\emph{matrix-exponentially convex} if
\begin{align}\label{matrix-exponential-convexity}
 \sum_{i,j=1}^N u_i\cdot H(r_i+r_j)u_j\geq0
\end{align}
for every \(N\in\mathbb N\) and every
\(r_1,\ldots,r_N,u_1,\ldots,u_N\in\R^m\).  The purpose of this condition is
to recognize when a Hessian can come from a L\'evy--Khintchine exponent.
The one-dimensional scalar origin of the representation is the classical
Bernstein--Widder theorem; see
\cite[Chapter~VI, Section~21]{WidderLaplace}.  We use its multidimensional
scalar extension due to Shucker
\cite[Definition~4, Theorem~4, and the corollary following it]{ShuckerWidder}.
Applying this result to \(u\cdot H(\cdot)u\) for each \(u\in\R^m\), and then
using polarization, gives the required matrix-valued measure.  The
L\'evy--Khintchine representation can then be constructed from this Laplace
transform.
When \(H\) is a
Hessian, this representation takes the more specific form
\begin{align*}
 H(r)=A+\int_{\R^m\setminus\{0\}}
 (z\otimes z)e^{z\cdot r}\,\nu(\d z),
\end{align*}
where \(A\) is symmetric and nonnegative and \(\nu\) is a positive
measure.  Integrating this identity twice in \(r\) gives a quadratic
Gaussian term and an exponential jump term, and these are exactly the two
parts of a L\'evy--Khintchine representation.  Therefore
\eqref{matrix-exponential-convexity} provides a sufficient condition for
the L\'evy realization constructed below.  

\begin{assumption}
\label{ass:Hessian-Levy-sufficient-condition}
Fix a strictly positive kinetic density
\begin{align}\label{Hessian-criterion-state-class}
 f\in L_+^1(\mathsf Z)
 :=\{g\in L^1(\mathsf Z):g\geq0\ \text{a.e.}\},
 \qquad f>0\quad\text{a.e.},
\end{align}
such that \(\mathcal E(f)\in\R\), the first variation
\(\delta\mathcal E/\delta f(f)\) is represented by a measurable function,
and
\begin{align}\label{Hessian-criterion-force}
 \eta_f:=-\frac12\mathscr D
 \frac{\delta\mathcal E}{\delta f}(f)\in X.
\end{align}
Suppose that the dual dissipation potential has the following form:
\begin{align}\label{local-integral-dissipation-potential}
 \Psi_f^*(\xi)
 =\int_{\mathscr C}\psi_f^*(c,\xi(c))\,\lambda(\d c),
 \qquad \xi\in X,
\end{align}
where \(\psi_f^*(c,r)\) is measurable in \(c\) and twice continuously
differentiable in \(r\in\R^m\).  For \(\lambda\)-almost every \(c\) and
every \(r\in\R^m\), assume
\begin{align*}
 \psi_f^*(c,-r)=\psi_f^*(c,r),
 \qquad
 \psi_f^*(c,0)=0,
 \qquad D_r\psi_f^*(c,0)=0
\end{align*}
and set
\begin{align*}
 H_f(c,r):=D_r^2\psi_f^*(c,r).
\end{align*}
Assume that \(r\mapsto H_f(c,r)\) is matrix-exponentially convex for
\(\lambda\)-almost every \(c\), and that \(\Psi_f^*\) is twice continuously
G\^ateaux differentiable on \(X\), with
\begin{align}\label{integrated-Hessian-assumption}
 D^2\Psi_f^*(\rho)[\xi,\zeta]
 =\int_{\mathscr C}
 \xi(c)\cdot H_f(c,\rho(c))\zeta(c)\,\lambda(\d c)
\end{align}
defining a finite continuous bilinear form for every
\(\rho,\xi,\zeta\in X\).
Finally, assume that the Hessian at the origin is locally integrable:
\begin{align}\label{local-Hessian-trace-integrability}
 \int_K\operatorname{tr}H_f(c,0)\,\lambda(\d c)<\infty
\end{align}
for every measurable \(K\subset\mathscr C\) with \(\lambda(K)<\infty\).
\end{assumption}

\begin{theorem}\label{thm:Hessian-Levy-sufficient-condition}
Suppose that Assumption~\ref{ass:Hessian-Levy-sufficient-condition} holds.
Then there is a measurable matrix field
\(A_f:\mathscr C\to\R^{m\times m}\) that is symmetric and nonnegative-definite: for every
\(c\in\mathscr C\) and \(u\in\R^m\),
\begin{align*}
 A_f(c)^{\mathsf T}=A_f(c),
 \qquad u\cdot A_f(c)u\geq0, 
\end{align*}
and a measure \(\nu_f\) on
\(\mathscr C\times(\R^m\setminus\{0\})\) that is symmetric in its jump
variable: for every measurable
\(B\subset\mathscr C\times(\R^m\setminus\{0\})\),
\begin{align*}
 \nu_f(B)=\nu_f\bigl(\{(c,-z):(c,z)\in B\}\bigr), 
\end{align*}
such that, 
\begin{align}\label{Hessian-Levy-representation-Psi}
 \Psi_f^*(\xi)
 ={}&\frac12\int_{\mathscr C}
 \xi(c)\cdot A_f(c)\xi(c)\,\lambda(\d c)\notag\\
 &+\int_{\mathscr C\times(\R^m\setminus\{0\})}
 \bigl(e^{z\cdot\xi(c)}-1-z\cdot\xi(c)\bigr)
 \,\nu_f(\d c,\d z).
\end{align}
Furthermore, these data can be chosen to satisfy the local L\'evy integrability and exponential-moment
conditions of Definition~\ref{def:symmetric-Levy-data}.  In particular,
the Bregman remainder of \(\Psi_f^*\) at the entropy force
\eqref{Hessian-criterion-force} has the L\'evy--Khintchine representation
\begin{align}\label{Hessian-criterion-centered-LK}
 &\Psi_f^*\left(-\frac12\mathscr D
 \frac{\delta\mathcal E}{\delta f}(f)+\xi\right)
 -\Psi_f^*\left(-\frac12\mathscr D
 \frac{\delta\mathcal E}{\delta f}(f)\right)
 -D\Psi_f^*\left(-\frac12\mathscr D
 \frac{\delta\mathcal E}{\delta f}(f)\right)[\xi]\notag\\
 &\qquad={}
 \frac12\int_{\mathscr C}
 \xi(c)\cdot A_f(c)\xi(c)\,\lambda(\d c)+\int_{\mathscr C\times(\R^m\setminus\{0\})}
 \bigl(e^{z\cdot\xi(c)}-1-z\cdot\xi(c)\bigr)
 \,\nu_f^{\rm tilt}(\d c,\d z),
 \qquad \xi\in X,
\end{align}
where
\begin{align}\label{Hessian-criterion-tilt}
 \nu_f^{\rm tilt}(\d c,\d z)
 :=e^{z\cdot\eta_f(c)}\nu_f(\d c,\d z)
 =\exp\left\{-\frac12z\cdot
 \left(\mathscr D\frac{\delta\mathcal E}{\delta f}(f)\right)(c)
 \right\}\nu_f(\d c,\d z).
\end{align}
\end{theorem}

\begin{proof}
For \(\lambda\)-almost every \(c\in\mathscr C\), the function
\(r\mapsto H_f(c,r)\) is continuous and satisfies
\eqref{matrix-exponential-convexity}.  Fix \(u\in\R^m\) and write
\(h_u(r):=u\cdot H_f(c,r)u\).  For
\(s_1,\ldots,s_N\in\R^m\) and
\(a_1,\ldots,a_N\in\R\), take \(r_i=s_i/2\) and
\(u_i=a_i u\) in \eqref{matrix-exponential-convexity}.  We obtain
\begin{align*}
 \sum_{i,j=1}^N a_i a_j
 h_u\left(\frac{s_i+s_j}{2}\right)\geq0.
\end{align*}
Thus \(h_u\) satisfies the OS-positivity condition in
\cite[Definition~4]{ShuckerWidder}.  It follows from
\cite[Theorem~4 and the corollary following it]{ShuckerWidder} that there is
a unique finite positive measure \(\mu_{f,c}^u\) on \(\R^m\) such that
\begin{align*}
 h_u(r)=\int_{\R^m}e^{z\cdot r}\,\mu_{f,c}^u(\d z).
\end{align*}
If the opposite sign is used in the Laplace-transform convention, we
reflect the measure.

Let \(e_1,\ldots,e_m\) be the standard basis of \(\R^m\), and define
finite signed measures by
\begin{align*}
 (M_{f,c})_{kk}:=\mu_{f,c}^{e_k},\qquad (M_{f,c})_{kl}:=\frac14\left(
 \mu_{f,c}^{e_k+e_l}-\mu_{f,c}^{e_k-e_l}\right),
 \qquad k\neq l.
\end{align*}
The uniqueness of the scalar representing measures and the polarization
identity imply that
\begin{align*}
 \mu_{f,c}^u=\sum_{k,l=1}^m u_k u_l(M_{f,c})_{kl},
 \qquad u\in\R^m.
\end{align*}
Therefore, for every Borel set \(E\subset\R^m\),
\begin{align*}
 u\cdot M_{f,c}(E)u=\mu_{f,c}^u(E)\geq0.
\end{align*}
Hence \(M_{f,c}\) is a finite nonnegative matrix-valued measure.  Applying
the scalar representation to the basis vectors and their pairwise sums and
differences gives
\begin{align}\label{Hessian-Bernstein-Widder-representation}
 H_f(c,r)=\int_{\R^m}e^{z\cdot r}\,M_{f,c}(\d z),
 \qquad r\in\R^m.
\end{align}
It is unique because all its scalarizations are unique, and
\(M_{f,c}(\R^m)=H_f(c,0)\).
We will integrate these pointwise
measures over the channel variable \(c\).  For this purpose,
\(c\mapsto M_{f,c}\) must be a measurable kernel, meaning that
\(c\mapsto M_{f,c;ij}(B)\) is measurable for every Borel set
\(B\subset\R^m\) and every pair of matrix indices \(i,j\).  This follows
from the measurability of
\(c\mapsto H_f(c,r)\) for every \(r\in\R^m\) and the uniqueness of the
representing measure.

Representation
\eqref{Hessian-Bernstein-Widder-representation} makes \(H_f(c,\cdot)\)
real analytic.  Since \(H_f(c,r)=D_r^2\psi_f^*(c,r)\), for every
\(i,j,k\in\{1,\ldots,m\}\), the mixed derivatives satisfy
\begin{align*}
 \partial_{r_k}H_{f;ij}(c,r)
 =\partial_{r_k}\partial_{r_i}\partial_{r_j}\psi_f^*(c,r)=\partial_{r_i}\partial_{r_k}\partial_{r_j}\psi_f^*(c,r)
 =\partial_{r_i}H_{f;kj}(c,r).
\end{align*}
On the other hand, differentiating
\eqref{Hessian-Bernstein-Widder-representation} gives
\begin{align*}
 \partial_{r_k}H_{f;ij}(c,r)
 &=\int_{\R^m}z_k e^{z\cdot r}\,M_{f,c;ij}(\d z),\\
 \partial_{r_i}H_{f;kj}(c,r)
 &=\int_{\R^m}z_i e^{z\cdot r}\,M_{f,c;kj}(\d z).
\end{align*}
Consequently, the
finite signed measure \(z_kM_{f,c;ij}(\d z)-z_iM_{f,c;kj}(\d z)\) 
has vanishing bilateral Laplace transform for every \(r\in\R^m\).
Using the uniqueness of the bilateral Laplace transform to see that 
\begin{align}\label{Hessian-measure-compatibility}
 z_kM_{f,c;ij}(\d z)=z_iM_{f,c;kj}(\d z).
\end{align}
The scalar trace measure of \(M_{f,c}\) is given by
\begin{align*}
 B\longmapsto\operatorname{tr}\bigl(M_{f,c}(B)\bigr),
 \qquad B\subset\R^m\ \text{Borel}.
\end{align*}
Suppose that
\(\operatorname{tr}(M_{f,c}(B))=0\), then the nonnegative matrix
\(M_{f,c}(B)\) has zero trace and hence vanishes. This yields that every component measure \(M_{f,c;ij}\) is absolutely continuous with
respect to this trace measure. The Radon--Nikodym
theorem therefore provides measurable functions \(G_{f,c;ij}\) such that
\begin{align*}
 M_{f,c;ij}(\d z)
 =G_{f,c;ij}(z)\,\operatorname{tr}M_{f,c}(\d z).
\end{align*}
Collecting these functions into a matrix gives
\begin{align*}
 M_{f,c}(\d z)
 =G_{f,c}(z)\,\operatorname{tr}M_{f,c}(\d z).
\end{align*}
Since \(M_{f,c}\) is a symmetric nonnegative matrix-valued measure,
\(G_{f,c}(z)\) is symmetric and nonnegative for
\(\operatorname{tr}M_{f,c}\)-almost every \(z\).  Moreover, taking the trace in the
Radon--Nikodym identity gives
\begin{align*}
 \operatorname{tr}G_{f,c}(z)=1
 \qquad\text{for }\operatorname{tr}M_{f,c}\text{-almost every }z.
\end{align*}
Substituting the Radon--Nikodym representation into
\eqref{Hessian-measure-compatibility} gives, simultaneously for all
\(i,j,k\in\{1,\ldots,m\}\),
\begin{align*}
 z_kG_{f,c;ij}(z)=z_iG_{f,c;kj}(z)
 \qquad\text{for }\operatorname{tr}M_{f,c}\text{-almost every }z.
\end{align*}
Fix such a \(z\neq0\) and choose an index \(\ell\) with
\(z_\ell\neq0\).  Taking \(k=\ell\) gives
\begin{align*}
 G_{f,c;ij}(z)
 =z_i\frac{G_{f,c;\ell j}(z)}{z_\ell}.
\end{align*}
Thus, for every \(j\), the \(j\)-th column of \(G_{f,c}(z)\) is a
multiple of \(z\). Moreover, if \(\alpha_j(z):=\frac{G_{f,c;\ell j}(z)}{z_\ell}\),
then we have that \(G_{f,c}(z)=z\otimes\alpha(z)\). Since \(G_{f,c}(z)\) is symmetric,
\(z_i\alpha_j(z)=z_j\alpha_i(z)\) for all \(i,j\).  Taking
\(i=\ell\) shows that
\begin{align*}
 \alpha(z)=\frac{\alpha_\ell(z)}{z_\ell}z.
\end{align*}
Consequently, \(G_{f,c}\) can be represented by 
\begin{align*}
 G_{f,c}(z)=\frac{\alpha_\ell(z)}{z_\ell}z\otimes z.
\end{align*}
Finally, the normalization \(\operatorname{tr}G_{f,c}(z)=1\) gives \(\frac{\alpha_\ell(z)}{z_\ell}|z|^2=1\). Hence
\begin{align*}
 G_{f,c}(z)=\frac{z\otimes z}{|z|^2}
 \qquad\text{for }\operatorname{tr}M_{f,c}\text{-almost every }z\neq0.
\end{align*}
Define
\begin{align}\label{Hessian-data-construction}
 A_f(c):=M_{f,c}(\{0\}),
 \qquad
 \nu_{f,c}(\d z)
 :=\boldsymbol1_{\{z\neq0\}}|z|^{-2}
 \operatorname{tr}M_{f,c}(\d z),
 \qquad
 \nu_f(\d c,\d z)
 :=\lambda(\d c)\nu_{f,c}(\d z).
\end{align}
It follows that
\begin{align}\label{Hessian-Laplace-Levy-form}
 H_f(c,r)
 =A_f(c)+\int_{\R^m\setminus\{0\}}
 (z\otimes z)e^{z\cdot r}\,\nu_{f,c}(\d z).
\end{align}
Since \(\psi_f^*(c,\cdot)\) is even, its Hessian is even.  Uniqueness in
\eqref{Hessian-Bernstein-Widder-representation} then shows that
\(M_{f,c}(\d z)=M_{f,c}(\d(-z))\), and hence
\(\nu_f\) is symmetric.

In the following, we recover \(\psi_f^*(c,r)\) from its Hessian.  The second-order
Taylor formula with integral remainder along the segment
\(s\mapsto sr\) gives
\begin{align*}
 \psi_f^*(c,r)
 =\psi_f^*(c,0)+D_r\psi_f^*(c,0)\cdot r+\int_0^1(1-s)\,
 r\cdot H_f(c,sr)r\,\d s.
\end{align*}
The first two terms vanish by assumption.  We substitute
\eqref{Hessian-Laplace-Levy-form} into the remaining term to obtain
\begin{align*}
 \psi_f^*(c,r)
 ={}&\int_0^1(1-s)r\cdot A_f(c)r\,\d s\\
 &+\int_0^1\int_{\R^m\setminus\{0\}}
 (1-s)(z\cdot r)^2e^{s z\cdot r}
 \,\nu_{f,c}(\d z)\,\d s.
\end{align*}
The second integrand is nonnegative, and hence Tonelli's theorem allows us
to interchange the \(s\)- and \(z\)-integrals.  Using
\begin{align*}
 \int_0^1(1-s)\,\d s=\frac12,\qquad (z\cdot r)^2\int_0^1(1-s)e^{s z\cdot r}\,\d s
 =e^{z\cdot r}-1-z\cdot r,
\end{align*}
we obtain
\begin{align*}
 \psi_f^*(c,r)
 =\frac12r\cdot A_f(c)r
 +\int_{\R^m\setminus\{0\}}
 \bigl(e^{z\cdot r}-1-z\cdot r\bigr)
 \,\nu_{f,c}(\d z).
\end{align*}
Taking \(r=\xi(c)\), integrating over \(c\in\mathscr C\), and using
\(\nu_f(\d c,\d z)=\lambda(\d c)\nu_{f,c}(\d z)\), we deduce that
\begin{align*}
 \Psi_f^*(\xi)
 =\frac12\int_{\mathscr C}
 \xi(c)\cdot A_f(c)\xi(c)\,\lambda(\d c)+\int_{\mathscr C\times(\R^m\setminus\{0\})}
 \bigl(e^{z\cdot\xi(c)}-1-z\cdot\xi(c)\bigr)
 \,\nu_f(\d c,\d z),
\end{align*}
which is \eqref{Hessian-Levy-representation-Psi}.

We next apply the integral Taylor formula at the shifted point \(\eta_f\):
\begin{align*}
 \mathcal C_f(\xi)
 =\int_0^1(1-s)D^2\Psi_f^*(\eta_f+s\xi)[\xi,\xi]\,\d s.
\end{align*}
By \eqref{integrated-Hessian-assumption} and
\eqref{Hessian-Laplace-Levy-form},
\begin{align*}
 D^2\Psi_f^*(\eta_f+s\xi)[\xi,\xi]&={}
 \int_{\mathscr C}\xi(c)\cdot A_f(c)\xi(c)\,\lambda(\d c)\\
 &\qquad{}+
 \int_{\mathscr C\times(\R^m\setminus\{0\})}
 (z\cdot\xi(c))^2
 e^{z\cdot\eta_f(c)}e^{s z\cdot\xi(c)}
 \,\nu_f(\d c,\d z).
\end{align*}
Substituting this expression into the Taylor formula and applying Tonelli's
theorem, we obtain
\begin{align*}
 \mathcal C_f(\xi)
 =\frac12\int_{\mathscr C}
 \xi(c)\cdot A_f(c)\xi(c)\,\lambda(\d c)+\int_{\mathscr C\times(\R^m\setminus\{0\})}
 e^{z\cdot\eta_f(c)}(z\cdot\xi(c))^2
 \left(\int_0^1(1-s)e^{s z\cdot\xi(c)}\,\d s\right)
 \,\nu_f(\d c,\d z).
\end{align*}
Using
\begin{align*}
 \int_0^1(1-s)a^2e^{sa}\,\d s=e^a-1-a
\end{align*}
with \(a=z\cdot\xi(c)\), we can see that 
\begin{align*}
 \mathcal C_f(\xi)
 ={}&\frac12\int_{\mathscr C}
 \xi(c)\cdot A_f(c)\xi(c)\,\lambda(\d c)+\int_{\mathscr C\times(\R^m\setminus\{0\})}
 \bigl(e^{z\cdot\xi(c)}-1-z\cdot\xi(c)\bigr)
 e^{z\cdot\eta_f(c)}\,\nu_f(\d c,\d z).
\end{align*}
By the definition \eqref{Hessian-criterion-tilt}, the factor
\(e^{z\cdot\eta_f(c)}\) is part of the tilted jump measure.  Therefore we complete the proof of \eqref{Hessian-criterion-centered-LK}.

We finally verify that \((A_f,\nu_f)\) are admissible L\'evy channel data
in the sense of Definition~\ref{def:symmetric-Levy-data}.  For the local
integrability condition \eqref{local-Levy-integrability}, we take the trace of
\eqref{Hessian-Laplace-Levy-form} at \(r=0\), integrating over a measurable
set \(K\subset\mathscr C\) of finite \(\lambda\)-measure, and using
\eqref{local-Hessian-trace-integrability}.  This gives
\begin{align*}
 \int_K\operatorname{tr}A_f(c)\,\lambda(\d c)
 +\int_{K\times(\R^m\setminus\{0\})}|z|^2\,
 \nu_f(\d c,\d z)<\infty.
\end{align*}
Since \(A_f(c)\) is nonnegative and
\(1\wedge|z|^2\leq |z|^2\), all the local integrability requirements in
\eqref{local-Levy-integrability} follow.

We now turn to the two bounds in
\eqref{Levy-exponential-integrability}.  It follows from
\eqref{Hessian-Laplace-Levy-form} at \(r=0\) and
\eqref{integrated-Hessian-assumption}, for every \(\xi\in X\),
\begin{align*}
 \int_{\mathscr C}\xi(c)\cdot A_f(c)\xi(c)\,\lambda(\d c)
 \leq D^2\Psi_f^*(0)[\xi,\xi]<\infty.
\end{align*}
For the jump term, notice that
\begin{align*}
 (z\cdot\xi(c))^2
 e^{|z\cdot\eta_f(c)|+|z\cdot\xi(c)|}\leq
 \sum_{\alpha,\beta\in\{-1,1\}}
 (z\cdot\xi(c))^2
 e^{z\cdot(\alpha\eta_f(c)+\beta\xi(c))}.
\end{align*}
After integration with respect to \(\nu_f(\d c,\d z)\), we bound each term
on the right by
\(D^2\Psi_f^*(\alpha\eta_f+\beta\xi)[\xi,\xi]\), which is finite by
\eqref{integrated-Hessian-assumption}.  We obtain
\eqref{Levy-exponential-integrability}, and hence \((A_f,\nu_f)\) are
admissible L\'evy channel data.
\end{proof}

The two basic fluctuation mechanisms are both contained in this criterion.
If
\(\psi_f^*(c,r)=\frac12r\cdot M_f(c)r\), then
\(H_f(c,r)=M_f(c)\), the measure in
\eqref{Hessian-Bernstein-Widder-representation} is concentrated at zero,
and \(\mathcal C_f\) is Gaussian.  If
\begin{align*}
 \psi_f^*(c,r)=\kappa_f(c)\bigl(\cosh(b(c)\cdot r)-1\bigr),
\end{align*}
then
\begin{align*}
 H_f(c,r)=\kappa_f(c)b(c)\otimes b(c)
 \cosh(b(c)\cdot r),
\end{align*}
and the symmetric jump measure is
\begin{align*}
 \nu_{f,c}
 =\frac{\kappa_f(c)}2
 \bigl(\delta_{b(c)}+\delta_{-b(c)}\bigr).
\end{align*}
In this case, the centered remainder is the cumulant of two oppositely
directed, force-biased Poisson currents.

Before constructing the stochastic current, we record that the L\'evy
characteristics in the preceding representation are intrinsic.

\begin{proposition}
\label{prop:local-Levy-uniqueness}
Assume that there is a \(\pi\)-system
\(\mathcal K\subset\mathcal B\) such that
\begin{align}\label{localized-channel-test-assumption}
 \sigma(\mathcal K)=\mathcal B,\qquad
 \lambda(K)<\infty\quad\text{for every }K\in\mathcal K,
 \qquad
 K_n\uparrow\mathscr C
 \quad\text{for some }K_n\in\mathcal K,
\end{align}
and that
\begin{align}\label{localized-channel-fields-in-X}
 \boldsymbol1_K\theta\in X
 \qquad\text{for every }K\in\mathcal K
 \text{ and every }\theta\in\R^m.
\end{align}
Let \((A_i,\nu_i)\), \(i=1,2\), be two admissible L\'evy channel data
sets in the sense of
Definition~\ref{def:symmetric-Levy-data}.  Suppose that, for every
\(\xi\in X\),
\begin{align}\label{equality-of-local-Levy-exponents}
 &\frac12\int_{\mathscr C}\xi(c)\cdot A_1(c)\xi(c)\,\lambda(\d c)
 +\int_{\mathscr C\times(\R^m\setminus\{0\})}
 \bigl(e^{z\cdot\xi(c)}-1-z\cdot\xi(c)\bigr)
 \,\nu_1(\d c,\d z)\notag\\
 &\qquad={}
 \frac12\int_{\mathscr C}\xi(c)\cdot A_2(c)\xi(c)\,\lambda(\d c)
 +\int_{\mathscr C\times(\R^m\setminus\{0\})}
 \bigl(e^{z\cdot\xi(c)}-1-z\cdot\xi(c)\bigr)
 \,\nu_2(\d c,\d z).
\end{align}
Then
\begin{align*}
 A_1(c)=A_2(c)\quad\text{for \(\lambda\)-almost every \(c\)},
 \qquad \nu_1=\nu_2.
\end{align*}
Consequently, whenever \(\mathcal C_f\) has the local
L\'evy--Khintchine representation
\eqref{abstract-current-Levy-Khintchine}, its Gaussian covariance
\(A_f\) and jump measure \(\nu_f^{\rm tilt}\) are unique.
\end{proposition}

\begin{proof}
Fix \(K\in\mathcal K\) and take \(\xi=\boldsymbol1_K\theta\) in
\eqref{equality-of-local-Levy-exponents}.  Set
\begin{align*}
 A_i^K:=\int_KA_i(c)\,\lambda(\d c),
 \qquad
 \nu_i^K(\Gamma):=\nu_i(K\times\Gamma),
 \qquad
 \Gamma\in\mathcal B(\R^m\setminus\{0\}),\quad i=1,2.
\end{align*}
Then \eqref{equality-of-local-Levy-exponents} states that the two functions
\begin{align*}
 \theta\longmapsto
 \frac12\theta\cdot A_i^K\theta
 +\int_{\R^m\setminus\{0\}}
 \bigl(e^{z\cdot\theta}-1-z\cdot\theta\bigr)\,\nu_i^K(\d z)
\end{align*}
coincide.  They are the logarithmic moment-generating functions of two
infinitely divisible laws on \(\R^m\).  The exponential integrability in
\eqref{Levy-exponential-integrability} makes them finite in a neighborhood
of the origin, so their equality implies that the two laws coincide.

Applying the uniqueness of the finite-dimensional L\'evy characteristics
\cite[Corollary~2.4.21]{ApplebaumLevy}, we obtain
\begin{align*}
 A_1^K=A_2^K,
 \qquad
 \nu_1^K(\Gamma)=\nu_2^K(\Gamma),
\end{align*}
for every Borel set
\(\Gamma\subset\R^m\setminus\{0\}\) bounded away from zero.  We vary
\(K\) over \(\mathcal K\) and use the exhaustion in
\eqref{localized-channel-test-assumption}, the uniqueness for
\(\sigma\)-finite measures on a generating \(\pi\)-system then identifies
the matrix-valued measures, and we obtain \(A_1=A_2\)
\(\lambda\)-almost everywhere.  For each \(\Gamma\) bounded away from zero,
the same argument identifies
\(K\mapsto\nu_i(K\times\Gamma)\) on \(\mathscr C\).  Such sets \(\Gamma\),
together with an exhaustion of \(\R^m\setminus\{0\}\) by annuli, form a
determining class for the jump variable.  The local L\'evy condition gives
\(\sigma\)-finiteness, so a second monotone-class argument yields
\(\nu_1=\nu_2\) on
\(\mathscr C\times(\R^m\setminus\{0\})\).
\end{proof}

\begin{remark}
The localization assumptions of
Proposition~\ref{prop:local-Levy-uniqueness} hold for the VFP, Landau,
Boltzmann, and kinetic \(\alpha\)-stable channel spaces fixed in
\eqref{VFP-channel-space}, \eqref{Landau-channel-space},
\eqref{Boltzmann-channel-space}, and \eqref{stable-channel-space}.  For the
first three models, one may choose \(\mathcal K\) to be a generating
\(\pi\)-system of bounded rectangles, together with a generating class of
Borel subsets of \(\mathbb S^{d-1}\) in the Boltzmann case.  These sets
have finite \(\lambda\)-measure and their indicators belong to the
corresponding ambient \(L^\infty\) space.  For the kinetic
\(\alpha\)-stable model, take bounded rectangles in \((x,v)\) and restrict
the jump variable to annuli
\begin{align*}
 \{z:r\leq |z|\leq R\},\qquad 0<r<R<\infty.
\end{align*}
Such sets have finite \(\lambda_\alpha\)-measure, their indicators belong
to \(X_\alpha\), and a countable family of them generates the Borel
\(\sigma\)-algebra and exhausts \(\mathscr C_\alpha\).  Consequently,
whenever the centered cumulant of any of these four models has a local
L\'evy--Khintchine representation, its Gaussian and jump characteristics
are uniquely determined.
For a fixed state \(f\), the centered cumulant \(\mathcal C_f\)
determines not only the cylindrical law of the centered current, but also
its local Gaussian and jump characteristics whenever the test space \(X\)
is rich enough to satisfy the localization assumptions above.  This
uniqueness concerns the L\'evy characteristics and does not assert
uniqueness of a particular stochastic-integral realization.
\end{remark}

We next realize the associated centered current using independent Gaussian
and Poisson random measures.

\begin{theorem}\label{thm:stochastic-integral-realization}
Fix a strictly positive kinetic density
\begin{align}\label{frozen-kinetic-state-class}
 f\in L_+^1(\mathsf Z)
 :=\left\{g\in L^1(\mathsf Z):g\geq0\ \text{a.e.}\right\},
 \qquad f>0\quad\text{a.e.}
\end{align}
such that \(\mathcal E(f)\in\R\), the first variation
\[
 \frac{\delta\mathcal E}{\delta f}(f)
\]
is represented by a measurable function on \(\mathsf Z\), and the channel
field
\[
 -\frac12\mathscr D\frac{\delta\mathcal E}{\delta f}(f)
\]
is well defined, belongs to \(X\), and is a
G\^ateaux differentiability point of \(\Psi_f^*\).  Suppose that \(A_f\) and
\(\nu_f^{\rm tilt}\) satisfy the channel absolute-continuity and
integrability conditions \eqref{local-Levy-integrability}--
\eqref{Levy-exponential-integrability} and represent
\(\mathcal C_f\) through
\eqref{abstract-current-Levy-Khintchine}.  Let
\(\sigma_f(c):=A_f(c)^{1/2}\) be the measurable symmetric nonnegative square
root of \(A_f(c)\).  On a sufficiently rich filtered probability space,
let
\[
 W(\d s,\d c)
\]
be an \(\R^m\)-valued Gaussian white-noise random measure on
\([0,\infty)\times\mathscr C\) with control measure
\(\d s\,\lambda(\d c)\).  Independently, for each \(\varepsilon>0\), let
\[
 \mathcal N_f^\varepsilon(\d s,\d c,\d z)
\]
be a Poisson random measure on
\([0,\infty)\times\mathscr C\times(\R^m\setminus\{0\})\) with intensity
\[
 \varepsilon^{-1}\d s\,
 \nu_f^{\rm tilt}(\d c,\d z),
\]
and write
\[
 \widetilde{\mathcal N}_f^\varepsilon
 :=\mathcal N_f^\varepsilon
   -\varepsilon^{-1}\d s\,\nu_f^{\rm tilt}.
\]
Then the stochastic integrals
\begin{align}\label{stochastic-integral-Levy-realization}
 \left\langle\mathbb L_t^{\varepsilon,f},\xi\right\rangle_{\mathscr C}
 :=\sqrt{\varepsilon}
 \int_0^t\int_{\mathscr C}
 \sigma_f(c)^{\mathsf T}\xi(c)\cdot W(\d s,\d c)+\varepsilon
 \int_0^t\int_{\mathscr C\times(\R^m\setminus\{0\})}
 z\cdot\xi(c)\,
 \widetilde{\mathcal N}_f^\varepsilon(\d s,\d c,\d z),
 \qquad \xi\in X,
\end{align}
define a centered cylindrical L\'evy current and satisfy, for every
\(t,h\geq0\) and every \(\xi\in X\),
\begin{align}\label{stochastic-integral-cumulant-identity}
 \varepsilon\log\E\exp\left\{
 \varepsilon^{-1}
 \left\langle
 \mathbb L_{t+h}^{\varepsilon,f}
 -\mathbb L_t^{\varepsilon,f},\xi
 \right\rangle_{\mathscr C}
 \right\}
 =h\mathcal C_f(\xi).
\end{align}
\end{theorem}

\begin{proof}
For every admissible test field \(\xi\), the assumptions give
\begin{align*}
 \int_{\mathscr C}
 \left|\sigma_f(c)^{\mathsf T}\xi(c)\right|^2\lambda(\d c)
 &=\int_{\mathscr C}
 \xi(c)\cdot A_f(c)\xi(c)\,\lambda(\d c)<\infty,\\
 \int_{\mathscr C\times(\R^m\setminus\{0\})}
 |z\cdot\xi(c)|^2\,\nu_f^{\rm tilt}(\d c,\d z)&<\infty.
\end{align*}
It follows that the first integral in
\eqref{stochastic-integral-Levy-realization} exists by the Gaussian
isometry, and the second exists as a square-integrable compensated Poisson
integral.  Both have mean zero.  Because the driving random measures are
independently scattered in time and mutually independent, every finite
collection of the projections in
\eqref{stochastic-integral-Levy-realization} has stationary independent
increments.

In the following, we identify their laws on the interval \((t,t+h]\).  The
Wiener integral below is centered Gaussian, with variance given by the
Gaussian isometry; see \cite[Lemma~4.3.11]{ApplebaumLevy}.  Hence
\begin{align*}
 \log\E\exp\left\{
 \varepsilon^{-1/2}
 \int_t^{t+h}\int_{\mathscr C}
 \sigma_f(c)^{\mathsf T}\xi(c)\cdot W(\d s,\d c)
 \right\}=
 \frac{h}{2\varepsilon}\int_{\mathscr C}
 \xi(c)\cdot A_f(c)\xi(c)\,\lambda(\d c).
\end{align*}
The Poisson exponential formula and its compensated form follow from
\cite[Theorem~2.3.8(1) and Eq.~(2.9)]{ApplebaumLevy}.  Using the
exponential integrability in \eqref{Levy-exponential-integrability} to pass
from characteristic functions to moment-generating functions, we obtain
\begin{align*}
 \log\E\exp\left\{
 \int_t^{t+h}\int_{\mathscr C\times(\R^m\setminus\{0\})}
 z\cdot\xi(c)\,
 \widetilde{\mathcal N}_f^\varepsilon(\d s,\d c,\d z)
 \right\}=
 \frac{h}{\varepsilon}
 \int_{\mathscr C\times(\R^m\setminus\{0\})}
 \bigl(e^{z\cdot\xi(c)}-1-z\cdot\xi(c)\bigr)
 \,\nu_f^{\rm tilt}(\d c,\d z).
\end{align*}
By independence, we may add the two logarithmic moment-generating
functions.  Multiplying by \(\varepsilon\) and using
\eqref{abstract-current-Levy-Khintchine}, we obtain
\eqref{stochastic-integral-cumulant-identity}.
\end{proof}

\begin{remark}[An explicit Gaussian--Poisson form of the master equation]
\label{rem:explicit-master-SPDE}
Theorem~\ref{thm:stochastic-integral-realization} justifies the frozen-state
Gaussian--Poisson integrals used in the formal master equation
\eqref{intro-explicit-master-SPDE}.  We write the distributional equation
in weak form.  For every
\(p\in\operatorname{Dom}(\mathscr D)\),
\begin{align}\label{explicit-Gaussian-Poisson-master-weak}
 \d\langle F_t^\varepsilon,p\rangle
 +\langle\mathcal T(F_t^\varepsilon),p\rangle\,\d t=&
 D\Psi_{F_t^\varepsilon}^*
 \left(-\frac12\mathscr D
 \frac{\delta\mathcal E}{\delta f}(F_t^\varepsilon)\right)
 [\mathscr Dp]\,\d t\notag\\
 &+\sqrt{\varepsilon}\int_{\mathscr C}
 \sigma_{F_{t-}^\varepsilon}(c)^{\mathsf T}
 \mathscr Dp(c)\cdot W(\d t,\d c)\notag\\
 &+\varepsilon
 \int_{\mathscr C\times(\R^m\setminus\{0\})}
 z\cdot\mathscr Dp(c)\,
 \widetilde{\mathcal N}^\varepsilon(\d t,\d c,\d z).
\end{align}
For the Boltzmann equation, the Poisson term in
\eqref{explicit-Gaussian-Poisson-master-weak} agrees with the collision
noise of Section~2 when \(m=1\), \(A_f=0\), and the tilted jump measure is
the two-point measure
\begin{align*}
 \nu_{f,\mathrm B}^{\rm tilt}(\d c,\d z)
 =\frac14B(c)\left[
 a_f(c)\delta_{+1}(\d z)+b_f(c)\delta_{-1}(\d z)
 \right]\lambda(\d c),
\end{align*}
where
\[
 a_f(c):=f(x,v)f(x,v_*),
 \qquad
 b_f(c):=f(x,v')f(x,v_*').
\]
This is the measure identified in
\eqref{Boltzmann-Levy-measure}.  Let
\(\widetilde{\mathcal N}_+\) and
\(\widetilde{\mathcal N}_-\) be independent compensated Poisson random
measures with intensity
\(\d t\,\lambda(\d c)\,\d u\).  The marking and thinning properties of
Poisson random measures \cite[Chapter~2]{ApplebaumLevy} give, conditionally on
\(F_{t-}^\varepsilon=f\),
\begin{align*}
 \varepsilon\int_{\mathscr C\times\{-1,+1\}}
 z\,\mathscr D_{\mathrm B}p(c)\,
 \widetilde{\mathcal N}^\varepsilon(\d t,\d c,\d z)&\stackrel{\mathrm{law}}{=}
 \varepsilon\int_{\mathscr C\times[0,\infty)}
 \mathscr D_{\mathrm B}p(c)
 \boldsymbol1_{\{u\leq(4\varepsilon)^{-1}B(c)a_f(c)\}}
 \widetilde{\mathcal N}_+(\d t,\d c,\d u)\notag\\
 &\qquad\quad{}
 -\varepsilon\int_{\mathscr C\times[0,\infty)}
 \mathscr D_{\mathrm B}p(c)
 \boldsymbol1_{\{u\leq(4\varepsilon)^{-1}B(c)b_f(c)\}}
 \widetilde{\mathcal N}_-(\d t,\d c,\d u).
\end{align*}
If, in addition, the pre--post-collisional involution preserves
\(\lambda\), the kernel is microreversible, and
\(b_f(\mathfrak r c)=a_f(c)\), then the calculation
\eqref{reversed-channel-derivative-calculation}--
\eqref{unified-gives-Boltzmann} converts this two-sided current, after
application of \(\mathscr D_{\mathrm B}^*\), into the forward-only
Boltzmann noise
\begin{align*}
 \varepsilon\int_{\mathscr C\times[0,\infty)}
 \mathcal J_c
 \boldsymbol1_{\{u\leq(2\varepsilon)^{-1}B(c)a_f(c)\}}
 \widetilde{\mathcal N}(\d t,\d c,\d u).
\end{align*}
It follows that the Poisson term in
\eqref{explicit-Gaussian-Poisson-master-weak} and the fluctuating
Boltzmann noise in \eqref{unified-gives-Boltzmann} define the same
state-space noise in law.
\end{remark}

\section{Consistency of the canonical and SPDE actions}
\label{sec:rate-identification}

In this section, we compare two density-path actions associated with the
master fluctuating kinetic equation.  One is the canonical
energy--dissipation action of the generalized gradient structure, and the
other is the Freidlin-Wentzell action obtained formally from the SPDE through the weak-convergence approach.  After introducing them
separately, we prove their equality on a concrete class of regular paths by
using channel currents and Gaussian--Poisson convex duality.

We use the notation fixed in Section~\ref{sec:Psi-structure-examples}.
The entropy and transport are specified in
\eqref{kinetic-state-functional-setting} and
\eqref{kinetic-transport-domain}.  The channel spaces \(X,X^*\), the maps
\(\mathscr D,\mathscr D^*\), and their duality are defined in
\eqref{Banach-channel-space}, \eqref{channel-derivative-domain},
\eqref{channel-adjoint-domain}, and \eqref{channel-adjoint-pairing}.  The dual
potential \(\Psi_f^*\), its Legendre transform \(\Psi_f\), and the entropy
force are given in \eqref{dual-dissipation-potential-domain},
\eqref{primal-dissipation-potential}, and
\eqref{entropy-force-convention}.  When the centered cumulant is
L\'evy realizable, \(A_f\) and \(\nu_f^{\rm tilt}\) denote its Gaussian
covariance and jump measure.

\subsection{The canonical rate function and its physical meaning}
\label{subsec:canonical-action}

Generalized gradient systems carry a canonical action which measures the
failure of a path to follow the deterministic dissipative evolution.  At
each time, a channel current \(J\in X^*\) must produce the prescribed
density velocity through
\begin{align}\label{canonical-current-constraint}
 \mathscr D^*J=\partial_t f_t+\mathcal T(f_t).
\end{align}
For every path for which
the terms below are meaningful, define
\begin{align}\label{canonical-density-action}
 I_{[0,T]}^{\rm can}(f)
 :=\frac12\bigl[\mathcal E(f_T)-\mathcal E(f_0)\bigr]
 +\int_0^T
 \inf_{\substack{J\in X^*:\;\mathscr D^*J
 =\partial_t f_t+\mathcal T(f_t)}}
 \Psi_{f_t}(J)\,\d t+\int_0^T\Psi_{f_t}^*\left(-\frac12\mathscr D
 \frac{\delta\mathcal E}{\delta f}(f_t)\right)\,\d t.
\end{align}
Here the infimum is \(+\infty\) when the constraint set is empty.  The
three terms in \eqref{canonical-density-action} have a direct
thermodynamic interpretation.  The first term represents the change
of the dissipated free energy \(\mathcal E\), the second term is the least primal dissipation cost among all channel currents
that produce the observed density velocity, and the last time integral
evaluates the dual dissipation at the entropy force
\begin{align*}
 \eta_{f_t}=-\frac12\mathscr D
 \frac{\delta\mathcal E}{\delta f}(f_t).
\end{align*}

Suppose that the transport preserves \(\mathcal E\), and let \(J_t\) be any
current satisfying \eqref{canonical-current-constraint}.  Then the chain
rule and the adjoint identity give
\begin{align*}
 \frac12\bigl[\mathcal E(f_T)-\mathcal E(f_0)\bigr]
 =-\int_0^T J_t[\eta_{f_t}]\,\d t.
\end{align*}
It follows that the value associated with this current is
\begin{align*}
 \int_0^T\left(
 \Psi_{f_t}(J_t)+\Psi_{f_t}^*(\eta_{f_t})
 -J_t[\eta_{f_t}]
 \right)\,\d t\geq0,
\end{align*}
where the inequality is precisely the Fenchel--Young inequality
\begin{align*}
 \Psi_f(J)+\Psi_f^*(\eta)\geq J[\eta];
\end{align*}
see \cite[Proposition~13.13]{BauschkeCombettes}.  Although the cited result
is stated for a Hilbert-space pairing, its proof uses only the dual pairing
and applies equally to \(X^*\) and \(X\).  Equality holds if and only if
\(J_t\in\partial\Psi_{f_t}^*(\eta_{f_t})\), equivalently
\(\eta_{f_t}\in\partial\Psi_{f_t}(J_t)\).  At differentiability points,
this means \(J_t=D\Psi_{f_t}^*(\eta_{f_t})\), and the current constraint
becomes the deterministic generalized gradient equation.  Taking the
infimum over all admissible currents, we deduce that
\(I_{[0,T]}^{\rm can}\geq0\), with zero attained along the deterministic
evolution.

This variational characterization also explains the term \emph{canonical}.
The free energy, the dual pair
\((\Psi_f,\Psi_f^*)\), and the current constraint determine the action
before any stochastic realization is chosen.  Equivalently, the deterministic
generalized gradient flow is characterized by
\(I_{[0,T]}^{\rm can}(f)=0\).  This is the standard \(\Psi+\Psi^*\)
energy--dissipation principle; see
\cite[Lemma~2.1 and Eqs.~(2.2)--(2.3)]{MielkePeletierRenger} and
\cite[Eq.~(1.9)]{PattersonRengerSharma}.

The same action has a probabilistic interpretation when the generalized
gradient system arises from an underlying thermodynamic particle dynamics.
It is then the dynamical large-deviation cost of observing a density path.
The primal potential measures the cost of the current, the dual potential
normalizes the local current fluctuations, and the free-energy difference
accounts for the thermodynamic driving
\cite{BertiniMFT,PattersonRengerSharma,RengerMFTGENERIC}.  For the
Boltzmann cosh potential, the current is the empirical collision flux and
the corresponding action is the collision-path cost of the hard-sphere
dynamics studied in
\cite{BodineauGallagherSaintRaymondSimonella}.  Contracting over collision
currents gives the density-path functional; on the regular paths considered
here, its identification with the canonical action agrees with
\cite[Eq.~(1.24) and Appendix~B]{WuPoissonBoltzmann}.

\subsection{The SPDE rate function from the weak-convergence approach}
\label{subsec:formal-FW}

The stochastic-integral realization in
Theorem~\ref{thm:stochastic-integral-realization} puts the master equation
in the form used by the weak-convergence approach.  Write
\begin{align}\label{Gaussian-square-root-definition}
 \sigma_f(c):=A_f(c)^{1/2},
 \qquad A_f(c)=\sigma_f(c)\sigma_f(c)^{\mathsf T}.
\end{align}
Let \(W\) be an \(\R^m\)-valued Gaussian white-noise random measure with
control measure \(\d t\,\lambda(\d c)\).  Conditionally on the left-limit
state, let \(\mathcal N^\varepsilon\) have predictable compensator
\(\varepsilon^{-1}\d t\,
\nu_{F_{t-}^\varepsilon}^{\rm tilt}(\d c,\d z)\), and set
\begin{align}\label{state-dependent-compensated-Poisson-measure}
 \widetilde{\mathcal N}^\varepsilon(\d t,\d c,\d z)
 :=\mathcal N^\varepsilon(\d t,\d c,\d z)
 -\varepsilon^{-1}\d t\,
 \nu_{F_{t-}^\varepsilon}^{\rm tilt}(\d c,\d z).
\end{align}
The formal state-dependent stochastic equation is
\begin{align}\label{explicit-Gaussian-Poisson-master-SPDE}
 \d F_t^\varepsilon+\mathcal T(F_t^\varepsilon)\,\d t
 ={}&\mathscr D^*D\Psi_{F_t^\varepsilon}^*\left(
 -\frac12\mathscr D\frac{\delta\mathcal E}{\delta f}
 (F_t^\varepsilon)\right)\,\d t\notag\\
 &+\sqrt\varepsilon\,\mathscr D^*
 \bigl(\sigma_{F_{t-}^\varepsilon}(c)W(\d t,\d c)\bigr)+\varepsilon\mathscr D^*
 \int_{\R^m\setminus\{0\}}z\,
 \widetilde{\mathcal N}^\varepsilon(\d t,\d c,\d z).
\end{align}

Equation~\eqref{explicit-Gaussian-Poisson-master-SPDE} remains formal at
this stage, but its stochastic-integral form suggests a natural path-space
rate function through the weak-convergence approach.  Its rigorous
justification requires a solution theory of this SPDE and a
Freidlin--Wentzell large-deviation principle.  These questions depend on
the concrete kinetic model and belong to the next stages of the program.

In order to invoke the framework established in \cite{BudhirajaDupuisMaroulas}, we assume that, for every
sufficiently regular state \(f\), the jump
measure has the factorization
\begin{align}\label{common-reference-jump-measure}
 \nu_f^{\rm tilt}(\d c,\d z)
 =k_f(c,z)\,\bar\nu(\d c,\d z)
\end{align}
for a nonnegative measurable density \(k_f\) and a \(\sigma\)-finite
measure \(\bar\nu\) on
\(\mathscr C\times(\R^m\setminus\{0\})\) that is independent of \(f\). Let
\(\overline{\mathcal N}^{\varepsilon}\) be a Poisson random measure on
\([0,T]\times\mathscr C\times(\R^m\setminus\{0\})\times[0,\infty)\) with
intensity
\[
 \varepsilon^{-1}\d t\,\bar\nu(\d c,\d z)\,\d u.
\]
Then the state-dependent Poisson random measure in
\eqref{state-dependent-compensated-Poisson-measure} can be realized by
thinning:
\begin{align}\label{common-reference-thinning-realization}
 \mathcal N^\varepsilon(\d t,\d c,\d z)
 :={}&\int_0^\infty
 \boldsymbol 1_{\{u\leq k_{F_{t-}^\varepsilon}(c,z)\}}
 \,\overline{\mathcal N}^{\varepsilon}
 (\d t,\d c,\d z,\d u), 
\end{align}
and its predictable compensator is
\(\varepsilon^{-1}\d t\,
\nu_{F_{t-}^\varepsilon}^{\rm tilt}(\d c,\d z)\).  This is the enlarged
Poisson-space construction of
\cite[Section~2.1, Eq.~(2.1), and Theorem~2.1]{BudhirajaDupuisMaroulas}.
This construction puts the Poisson part of
\eqref{explicit-Gaussian-Poisson-master-SPDE} on the fixed noise space of
the weak-convergence method, then we can use the framework of
Budhiraja--Chen--Dupuis and Budhiraja--Dupuis--Maroulas formally
\cite{BudhirajaChenDupuis,BudhirajaDupuisMaroulas} to identify the expected
rate function.

The Boltzmann realization has this form with
\begin{align*}
 \bar\nu(\d c,\d z)
 &=\lambda(\d c)\bigl(\delta_{+1}+\delta_{-1}\bigr)(\d z),\\
 k_f(c,+1)&=\frac14B(c)a_f(c),
 &k_f(c,-1)&=\frac14B(c)b_f(c),
\end{align*}
which gives the tilted measure in \eqref{Boltzmann-Levy-measure}.

A predictable Poisson control \(q_t(c,z)\geq0\) replaces the indicator in
\eqref{common-reference-thinning-realization} by
\begin{align*}
 \mathcal N^{\varepsilon,q}(\d t,\d c,\d z)
 :={}&\int_0^\infty
 \boldsymbol 1_{\{u\leq
 q_t(c,z)k_{F_{t-}^\varepsilon}(c,z)\}}
 \,\overline{\mathcal N}^{\varepsilon}
 (\d t,\d c,\d z,\d u).
\end{align*}
By \eqref{common-reference-jump-measure}, its predictable compensator is
\begin{align*}
 \varepsilon^{-1}\d t\,q_t(c,z)
 k_{F_{t-}^\varepsilon}(c,z)\,\bar\nu(\d c,\d z)=
 \varepsilon^{-1}\d t\,q_t(c,z)
 \nu_{F_{t-}^\varepsilon}^{\rm tilt}(\d c,\d z).
\end{align*}
Relative to the original compensator, the controlled jump integral
decomposes formally as
\begin{align}\label{controlled-Poisson-drift-decomposition}
 &\varepsilon\int_{\mathscr C\times(\R^m\setminus\{0\})} z\left(
 \mathcal N^{\varepsilon,q}(\d t,\d c,\d z)
 -\varepsilon^{-1}\d t\,
 \nu_{F_{t-}^\varepsilon}^{\rm tilt}(\d c,\d z)\right)\notag\\
 &\quad={}
 \varepsilon\int_{\mathscr C\times(\R^m\setminus\{0\})} z\left(
 \mathcal N^{\varepsilon,q}(\d t,\d c,\d z)
 -\varepsilon^{-1}\d t\,q_t(c,z)
 \nu_{F_{t-}^\varepsilon}^{\rm tilt}(\d c,\d z)\right)\notag\\
 &\qquad{}+\d t
 \int_{\mathscr C\times(\R^m\setminus\{0\})}
 z\bigl(q_t(c,z)-1\bigr)
 k_{F_{t-}^\varepsilon}(c,z)\,\bar\nu(\d c,\d z)\notag\\
 &\quad={}
 \varepsilon\int_{\mathscr C\times(\R^m\setminus\{0\})} z\left(
 \mathcal N^{\varepsilon,q}(\d t,\d c,\d z)
 -\varepsilon^{-1}\d t\,q_t(c,z)
 \nu_{F_{t-}^\varepsilon}^{\rm tilt}(\d c,\d z)\right)\notag\\
 &\qquad{}+\d t
 \int_{\mathscr C\times(\R^m\setminus\{0\})}
 z\bigl(q_t(c,z)-1\bigr)
 \nu_{F_{t-}^\varepsilon}^{\rm tilt}(\d c,\d z).
\end{align}
The first term on the right is the controlled centered fluctuation, which
vanishes in the formal small-noise limit, and the second term becomes the
Poisson part of the skeleton current.  In the last equality,
\(k_{F_{t-}^\varepsilon}\bar\nu\) recombines into
\(\nu_{F_{t-}^\varepsilon}^{\rm tilt}\) by
\eqref{common-reference-jump-measure}.  We therefore see that the limiting current depends
only on \(\nu_f^{\rm tilt}\) and is independent of the factorization in
\eqref{common-reference-jump-measure}.

Let
\begin{align*}
 h&\in L^2([0,T]\times\mathscr C,\d t\,\lambda;\R^m),\\
 q&:[0,T]\times\mathscr C\times(\R^m\setminus\{0\})
 \longrightarrow[0,\infty)
\end{align*}
be measurable Gaussian and Poisson controls.  In the variational
representation, the Gaussian noise is shifted by
\(\varepsilon^{-1/2}h_t\,\d t\,\lambda(\d c)\), while the Poisson
intensity is multiplied by \(q_t\).  Thus the Gaussian shift contributes
\(\sigma_fh\) to the limiting current, while
\eqref{controlled-Poisson-drift-decomposition} contributes
\(z(q-1)\nu_f^{\rm tilt}\).  Thus, at a frozen state \(f\), the skeleton
current is
\begin{align}\label{formal-controlled-current}
 J_f^{h,q}[\xi]
 :={}&D\Psi_f^*\left(-\frac12\mathscr D
 \frac{\delta\mathcal E}{\delta f}(f)\right)[\xi]
 +\int_{\mathscr C}
 \xi(c)\cdot\sigma_f(c)h(c)\,\lambda(\d c)\notag\\
 &+\int_{\mathscr C\times(\R^m\setminus\{0\})}
 z\cdot\xi(c)\bigl(q(c,z)-1\bigr)
 \nu_f^{\rm tilt}(\d c,\d z).
\end{align}
The limiting controlled equation is the skeleton
\begin{align}\label{formal-controlled-skeleton}
 \partial_t f+\mathcal T(f)=\mathscr D^*J_{f_t}^{h_t,q_t},
 \qquad f(0)=f_0.
\end{align}
It is understood weakly: for every
\(\phi\in\operatorname{Dom}(\mathscr D)\),
\begin{align}\label{formal-controlled-skeleton-weak}
 \left\langle\partial_t f_t+\mathcal T(f_t),\phi\right\rangle
 =J_{f_t}^{h_t,q_t}[\mathscr D\phi].
\end{align}

The instantaneous control cost reads 
\begin{align}\label{formal-control-cost}
 \mathfrak K_f(h,q)
 :={}&\frac12\int_{\mathscr C}|h(c)|^2\,\lambda(\d c)\notag\\
 &+\int_{\mathscr C\times(\R^m\setminus\{0\})}
 \ell\bigl(q(c,z)\bigr)\,
 \nu_f^{\rm tilt}(\d c,\d z),
 \qquad \ell(r):=r\log r-r+1,\qquad 0\log0:=0.
\end{align}
The first term is the Cameron--Martin energy, and the second is the relative
entropy cost of changing the Poisson intensity from
\(\nu_f^{\rm tilt}\) to \(q\nu_f^{\rm tilt}\).

For any path \(f\) for which the terms below are defined, such a pair
\((h,q)\) is called admissible for \(f\) if
\begin{align}\label{rigorous-control-integrability}
 \int_0^T\int_{\mathscr C\times(\R^m\setminus\{0\})}
 \ell(q_t(c,z))\,\nu_{f_t}^{\rm tilt}(\d c,\d z)\,\d t<\infty,
\end{align}
if \(t\mapsto J_{f_t}^{h_t,q_t}[\mathscr D\phi]\) is integrable for every
\(\phi\in\operatorname{Dom}(\mathscr D)\), and if
\eqref{formal-controlled-skeleton-weak} holds for almost every \(t\) and
every such \(\phi\).  We denote this class by \(\mathcal U(f)\).

The Poisson variational representation is given in
\cite[Theorem~2.1]{BudhirajaDupuisMaroulas}; the corresponding abstract
weak-convergence principles are
\cite[Theorem~4.2]{BudhirajaDupuisMaroulas} for joint Brownian--Poisson
driving noise and \cite[Theorem~2.3 and Eq.~(2.4)]{BudhirajaChenDupuis} for
Poisson-driven systems.  Applying these results formally to
\eqref{explicit-Gaussian-Poisson-master-SPDE}, we obtain
\begin{align}\label{formal-FW-path-rate}
 I_{[0,T]}^{\rm SPDE}(f)
 :=\inf_{(h,q)\in\mathcal U(f)}
 \int_0^T\mathfrak K_{f_t}(h_t,q_t)\,\d t.
\end{align}
The value is \(+\infty\) if \(\mathcal U(f)=\varnothing\).  The
uncontrolled pair \(h=0\), \(q=1\) has zero cost and gives the
deterministic master equation.

The factorization \eqref{common-reference-jump-measure} serves only to
place all state-dependent jump intensities on the fixed Poisson space
required by the variational representation.  The skeleton equation 
\eqref{formal-controlled-skeleton} and the cost
\eqref{formal-control-cost} depend on the factorization only through the
product \(k_f\bar\nu=\nu_f^{\rm tilt}\), hence the candidate action
\eqref{formal-FW-path-rate} is intrinsic to \(\nu_f^{\rm tilt}\).

Formula \eqref{formal-FW-path-rate} is therefore the action predicted by
the weak-convergence method.  A full Laplace principle will require a
measurable solution map, well-posedness of the controlled equations,
compactness and continuity of the skeleton map on bounded-cost sets, and
convergence of the controlled stochastic equations to the skeleton. Instead of discussing these problems, here we only focus on the comparison of \(I_{[0,T]}^{\rm SPDE}\) defined by \eqref{formal-FW-path-rate} and \( I_{[0,T]}^{\rm can}\) defined by \eqref{canonical-density-action}. 

\subsection{Equality of the two actions on regular tilted paths}
\label{subsec:rigorous-action-comparison}

We next introduce a regular class of paths on which the two actions can be
compared rigorously.  The space \(X\) is fixed along each path, and the
classes below depend on this choice through the entropy force, the tilted
force, and the current constraint.  The chain rule required in the proof is
included in the definition.

\begin{definition}[Regular paths and regular tilted paths]
\label{def:regular-canonical-path}
Let \(\mathcal A_{\rm reg}\) be the class of strictly positive paths
\begin{align}\label{regular-canonical-path}
 f\in C^1([0,T];L^1(\mathsf Z))\cap
 C([0,T];C^\infty(\mathsf Z)),
 \qquad f(0)=f_0,
\end{align}
such that, for every \(t\in[0,T]\),
\begin{align*}
 f_t&\in\operatorname{Dom}(\mathcal E)\cap
       \operatorname{Dom}(\mathcal T),
 &\mathcal T(f)&\in C([0,T];L^1(\mathsf Z)),\\
 \frac{\delta\mathcal E}{\delta f}(f_t)&\in
       \operatorname{Dom}(\mathscr D),
 &-\frac12\mathscr D\frac{\delta\mathcal E}{\delta f}(f_t)&\in X.
\end{align*}
We require the two maps
\begin{align*}
 t&\longmapsto
 \inf_{\substack{J\in X^*:\;\mathscr D^*J
 =\partial_t f_t+\mathcal T(f_t)}}\Psi_{f_t}(J),\\
 t&\longmapsto
 \Psi_{f_t}^*\left(-\frac12\mathscr D
 \frac{\delta\mathcal E}{\delta f}(f_t)\right)
\end{align*}
to be Lebesgue measurable and to satisfy
\begin{align}\label{regular-path-action-integrability}
 \int_0^T\Bigg[
 \inf_{\substack{J\in X^*:\;\mathscr D^*J
 =\partial_t f_t+\mathcal T(f_t)}}\Psi_{f_t}(J)
 +\Psi_{f_t}^*\left(-\frac12\mathscr D
 \frac{\delta\mathcal E}{\delta f}(f_t)\right)
 \Bigg]\,\d t<\infty.
\end{align}
In particular, the current constraint is solvable for almost every time.
We also require the duality pairings
\begin{align*}
 t&\longmapsto
 \left\langle\frac{\delta\mathcal E}{\delta f}(f_t),
 \partial_t f_t\right\rangle,
 &
 t&\longmapsto
 \left\langle\frac{\delta\mathcal E}{\delta f}(f_t),
 \mathcal T(f_t)\right\rangle
\end{align*}
to be well defined and to belong to \(L^1(0,T)\).  Finally,
\(t\mapsto\mathcal E(f_t)\) is required to be absolutely continuous with
\begin{align}\label{regular-path-entropy-identities}
 \frac{\d}{\d t}\mathcal E(f_t)
 &=\left\langle\frac{\delta\mathcal E}{\delta f}(f_t),
   \partial_t f_t\right\rangle,
 &
 \left\langle\frac{\delta\mathcal E}{\delta f}(f_t),
   \mathcal T(f_t)\right\rangle&=0
\end{align}
for almost every \(t\).

A path \(f\in\mathcal A_{\rm reg}\) belongs to
\(\mathcal A_{\rm tilt}\) if there is a field
\begin{align}\label{regular-tilt-field}
 p\in C^1([0,T];C_c^\infty(\mathsf Z)),
 \qquad \mathscr Dp_t\in X,
\end{align}
with a common compact support, such that, for almost every \(t\),
\begin{align}\label{regular-tilt-domain-condition}
 -\frac12\mathscr D\frac{\delta\mathcal E}{\delta f}(f_t)
 +\mathscr Dp_t
 \in\operatorname{Dom}(D\Psi_{f_t}^*)
\end{align}
and
\begin{align}\label{smooth-Hamiltonian-tilted-equation}
 \left\langle\partial_t f_t+\mathcal T(f_t),\phi\right\rangle
 =D\Psi_{f_t}^*\left(
 -\frac12\mathscr D\frac{\delta\mathcal E}{\delta f}(f_t)
 +\mathscr Dp_t\right)[\mathscr D\phi]
\end{align}
for almost every \(t\) and every
\(\phi\in\operatorname{Dom}(\mathscr D)\).  We further require that
\begin{align*}
 t\longmapsto
 D\Psi_{f_t}^*\left(
 -\frac12\mathscr D\frac{\delta\mathcal E}{\delta f}(f_t)
 +\mathscr Dp_t\right)[\mathscr D\phi]
\end{align*}
belongs to \(L^1(0,T)\) for every such \(\phi\), and that
\begin{align}\label{regular-tilt-control-cost}
 \int_0^T\Bigg[&\frac12\int_{\mathscr C}
 \left|\sigma_{f_t}(c)^{\mathsf T}\mathscr Dp_t(c)\right|^2
 \,\lambda(\d c)\notag\\
 &+\int_{\mathscr C\times(\R^m\setminus\{0\})}
 \ell\left(e^{z\cdot\mathscr Dp_t(c)}\right)
 \,\nu_{f_t}^{\rm tilt}(\d c,\d z)\Bigg]\,\d t<\infty.
\end{align}
\end{definition}

\begin{remark}
The abstract framework developed here does not describe the structure of
\(\mathcal A_{\rm tilt}\), nor does it guarantee that this class is
nonempty.  These questions depend on the solution theory for the tilted
equation after \(\mathcal E\), \(\mathscr D\), and \(\Psi_f^*\) have been
specified.  In the Vlasov--Fokker--Planck, Landau, and Boltzmann examples,
the tilted equation is, respectively, a Fokker--Planck equation with an
additional smooth drift, a controlled Landau equation, and an exponentially
biased Boltzmann equation.  Under the corresponding classical regularity,
positivity, decay, and local well-posedness assumptions, these equations
give nontrivial regular paths in \(\mathcal A_{\rm tilt}\) for a large class
of fields satisfying \(\mathscr Dp\not\equiv0\).  
\end{remark}

The following current-level statement contains the convex argument used in
the comparison.

\begin{lemma}
\label{lem:Gaussian-Poisson-current-contraction}
Fix a strictly positive state \(f\in L_+^1(\mathsf Z)\) such that
\begin{align*}
 -\frac12\mathscr D\frac{\delta\mathcal E}{\delta f}(f)
 \in\operatorname{Dom}(D\Psi_f^*)\subset X
\end{align*}
is well defined and the centered representation
\eqref{Hessian-criterion-centered-LK} holds.  If \((h,q)\) has finite cost
and \(J_f^{h,q}\) is given by
\eqref{formal-controlled-current}, then
\begin{align}\label{current-control-cost-inequality}
 \Psi_f(J_f^{h,q})
 +\Psi_f^*\left(-\frac12\mathscr D
 \frac{\delta\mathcal E}{\delta f}(f)\right)
 -J_f^{h,q}\left[-\frac12\mathscr D
 \frac{\delta\mathcal E}{\delta f}(f)\right]
 \leq\mathfrak K_f(h,q).
\end{align}
Furthermore, let \(\xi\in X\) satisfy
\begin{align*}
 -\frac12\mathscr D\frac{\delta\mathcal E}{\delta f}(f)+\xi
 \in\operatorname{Dom}(D\Psi_f^*),
\end{align*}
and suppose that the controls below have finite cost.  Set
\begin{align}\label{Fenchel-saturating-controls}
 h^\xi(c):=\sigma_f(c)^{\mathsf T}\xi(c),
 \qquad
 q^\xi(c,z):=e^{z\cdot\xi(c)}.
\end{align}
Then
\begin{align}\label{tilted-current-derivative}
 J_f^{h^\xi,q^\xi}
 =D\Psi_f^*\left(-\frac12\mathscr D
 \frac{\delta\mathcal E}{\delta f}(f)+\xi\right),
\end{align}
and equality holds in \eqref{current-control-cost-inequality}.
\end{lemma}

\begin{proof}
Fix \(\xi\in X\).  From the definition
\eqref{formal-controlled-current}, we have
\begin{align*}
 J_f^{h,q}[\xi]
 ={}&D\Psi_f^*\left(-\frac12\mathscr D
 \frac{\delta\mathcal E}{\delta f}(f)\right)[\xi]\notag\\
 &+\int_{\mathscr C}
 h(c)\cdot\sigma_f(c)^{\mathsf T}\xi(c)\,\lambda(\d c)+\int_{\mathscr C\times(\R^m\setminus\{0\})}
 \bigl(q(c,z)-1\bigr)z\cdot\xi(c)\,
 \nu_f^{\rm tilt}(\d c,\d z).
\end{align*}
The centered representation
\eqref{Hessian-criterion-centered-LK} gives
\begin{align*}
 \Psi_f^*\left(-\frac12\mathscr D
 \frac{\delta\mathcal E}{\delta f}(f)+\xi\right)
 -\Psi_f^*\left(-\frac12\mathscr D
 \frac{\delta\mathcal E}{\delta f}(f)\right)&={}
 D\Psi_f^*\left(-\frac12\mathscr D
 \frac{\delta\mathcal E}{\delta f}(f)\right)[\xi]
 +\frac12\int_{\mathscr C}
 \xi(c)\cdot A_f(c)\xi(c)\,\lambda(\d c)\\
 &\quad{}+
 \int_{\mathscr C\times(\R^m\setminus\{0\})}
 \bigl(e^{z\cdot\xi(c)}-1-z\cdot\xi(c)\bigr)
 \,\nu_f^{\rm tilt}(\d c,\d z).
\end{align*}
Since \(A_f=\sigma_f\sigma_f^{\mathsf T}\),
\begin{align*}
 \xi(c)\cdot A_f(c)\xi(c)
 =|\sigma_f(c)^{\mathsf T}\xi(c)|^2.
\end{align*}
Moreover, the jump integrands satisfy
\begin{align*}
 &\bigl(q(c,z)-1\bigr)z\cdot\xi(c)
 -\bigl(e^{z\cdot\xi(c)}-1-z\cdot\xi(c)\bigr)\\
 &\qquad={}
 q(c,z)z\cdot\xi(c)-\bigl(e^{z\cdot\xi(c)}-1\bigr).
\end{align*}
Subtracting these identities, we obtain
\begin{align*}
 &J_f^{h,q}[\xi]
 -\left\{
 \Psi_f^*\left(-\frac12\mathscr D
 \frac{\delta\mathcal E}{\delta f}(f)+\xi\right)
 -\Psi_f^*\left(-\frac12\mathscr D
 \frac{\delta\mathcal E}{\delta f}(f)\right)
 \right\}\\
 &={}
 \int_{\mathscr C}\left\{
 h(c)\cdot\sigma_f(c)^{\mathsf T}\xi(c)
 -\frac12|\sigma_f(c)^{\mathsf T}\xi(c)|^2
 \right\}\,\lambda(\d c)\\
 &\quad+
 \int_{\mathscr C\times(\R^m\setminus\{0\})}
 \left\{q(c,z)z\cdot\xi(c)
 -\bigl(e^{z\cdot\xi(c)}-1\bigr)\right\}
 \nu_f^{\rm tilt}(\d c,\d z).
\end{align*}
We now use the pointwise inequalities
\begin{align}\label{Gaussian-Poisson-Fenchel-inequalities}
 h\cdot r-\frac12|r|^2&\leq\frac12|h|^2,
 &qs-(e^s-1)&\leq\ell(q)
\end{align}
which show that this expression is bounded by \(\mathfrak K_f(h,q)\).
Therefore, for every \(\xi\in X\),
\begin{align*}
 &J_f^{h,q}[\xi]
 -\Psi_f^*\left(-\frac12\mathscr D
 \frac{\delta\mathcal E}{\delta f}(f)+\xi\right)
 +\Psi_f^*\left(-\frac12\mathscr D
 \frac{\delta\mathcal E}{\delta f}(f)\right)
 \leq\mathfrak K_f(h,q).
\end{align*}
Set
\begin{align*}
 \zeta:=-\frac12\mathscr D
 \frac{\delta\mathcal E}{\delta f}(f)+\xi.
\end{align*}
Since the entropy force belongs to \(X\), the variable \(\zeta\) ranges
over all of \(X\) as \(\xi\) ranges over \(X\).  Moreover, linearity of
\(J_f^{h,q}\in X^*\) gives
\begin{align*}
 J_f^{h,q}[\xi]
 =J_f^{h,q}[\zeta]
 -J_f^{h,q}\left[-\frac12\mathscr D
 \frac{\delta\mathcal E}{\delta f}(f)\right].
\end{align*}
Taking the supremum over \(\xi\in X\), or equivalently over
\(\zeta\in X\), therefore yields
\begin{align*}
 &\sup_{\xi\in X}\Bigg\{
 J_f^{h,q}[\xi]
 -\Psi_f^*\left(-\frac12\mathscr D
 \frac{\delta\mathcal E}{\delta f}(f)+\xi\right)
 +\Psi_f^*\left(-\frac12\mathscr D
 \frac{\delta\mathcal E}{\delta f}(f)\right)
 \Bigg\}\\
 &={}
 \Psi_f^*\left(-\frac12\mathscr D
 \frac{\delta\mathcal E}{\delta f}(f)\right)
 -J_f^{h,q}\left[-\frac12\mathscr D
 \frac{\delta\mathcal E}{\delta f}(f)\right]+
 \sup_{\zeta\in X}
 \left\{J_f^{h,q}[\zeta]-\Psi_f^*(\zeta)\right\}\\
 &={}
 \Psi_f(J_f^{h,q})
 +\Psi_f^*\left(-\frac12\mathscr D
 \frac{\delta\mathcal E}{\delta f}(f)\right)
 -J_f^{h,q}\left[-\frac12\mathscr D
 \frac{\delta\mathcal E}{\delta f}(f)\right],
\end{align*}
where the last equality follows from
\eqref{primal-dissipation-potential}.  This completes the proof of 
\eqref{current-control-cost-inequality}.

We next identify the current generated by the controls in
\eqref{Fenchel-saturating-controls}.  For an arbitrary direction
\(\rho\in X\), differentiate
\eqref{Hessian-criterion-centered-LK} at \(\xi\) in the direction \(\rho\).
The differentiability and integrability assumptions give
\begin{align*}
 &D\Psi_f^*\left(-\frac12\mathscr D
 \frac{\delta\mathcal E}{\delta f}(f)+\xi\right)[\rho]
 -D\Psi_f^*\left(-\frac12\mathscr D
 \frac{\delta\mathcal E}{\delta f}(f)\right)[\rho]\\
 &={}
 \int_{\mathscr C}\rho(c)\cdot A_f(c)\xi(c)\,\lambda(\d c)+
 \int_{\mathscr C\times(\R^m\setminus\{0\})}
 z\cdot\rho(c)\bigl(e^{z\cdot\xi(c)}-1\bigr)
 \,\nu_f^{\rm tilt}(\d c,\d z).
\end{align*}
Since \(A_f=\sigma_f\sigma_f^{\mathsf T}\) and
\begin{align*}
 h^\xi(c)=\sigma_f(c)^{\mathsf T}\xi(c),
 \qquad q^\xi(c,z)=e^{z\cdot\xi(c)},
\end{align*}
the right-hand side equals
\begin{align*}
 \int_{\mathscr C}
 \rho(c)\cdot\sigma_f(c)h^\xi(c)\,\lambda(\d c)+
 \int_{\mathscr C\times(\R^m\setminus\{0\})}
 z\cdot\rho(c)\bigl(q^\xi(c,z)-1\bigr)
 \,\nu_f^{\rm tilt}(\d c,\d z).
\end{align*}
We add
\(D\Psi_f^*(-\frac12\mathscr D
\frac{\delta\mathcal E}{\delta f}(f))[\rho]\) to both sides and use
\eqref{formal-controlled-current} to obtain
\begin{align*}
 D\Psi_f^*\left(-\frac12\mathscr D
 \frac{\delta\mathcal E}{\delta f}(f)+\xi\right)[\rho]
 =J_f^{h^\xi,q^\xi}[\rho].
\end{align*}
Since this identity holds for every \(\rho\in X\), we deduce
\eqref{tilted-current-derivative} as an identity in \(X^*\).

For the equality case, the Gaussian inequality in
\eqref{Gaussian-Poisson-Fenchel-inequalities} holds when
\(h=\sigma_f^{\mathsf T}\xi=h^\xi\), while equality in its Poisson
counterpart is an equality when \(q=e^{z\cdot\xi}=q^\xi\).  Thus the expression
before taking the supremum attains \(\mathfrak K_f(h^\xi,q^\xi)\) at the
test field \(\xi\).  The preceding supremum calculation then gives equality in
\eqref{current-control-cost-inequality} for
\((h^\xi,q^\xi)\).
\end{proof}

We also use the following form of the Fenchel--Young inequality and its
equality case.
\begin{lemma}[Fenchel--Young inequality and equality]
\label{lem:Fenchel-Young}
Fix a state \(f\), and let \(\Psi_f\) be the Legendre dual of
\(\Psi_f^*\) defined in \eqref{primal-dissipation-potential}.  Then, for
every \(J\in X^*\) and \(\xi\in X\),
\begin{align}\label{Fenchel-Young-current-pair}
 \Psi_f(J)+\Psi_f^*(\xi)\geq J[\xi].
\end{align}
If \(\Psi_f^*\) is G\^ateaux differentiable at \(\xi\), then equality
holds in \eqref{Fenchel-Young-current-pair} if and only if
\begin{align*}
 J=D\Psi_f^*(\xi).
\end{align*}
\end{lemma}

\begin{proof}
Inequality~\eqref{Fenchel-Young-current-pair} follows from
the definition
\begin{align*}
 \Psi_f(J)=\sup_{\zeta\in X}
 \{J[\zeta]-\Psi_f^*(\zeta)\}
\end{align*}
by choosing \(\zeta=\xi\).  The equality characterization is the
Fenchel--Young equality; see
\cite[Proposition~13.13]{BauschkeCombettes}.  The same convex-conjugacy
argument gives the stated result for the dual pairing between \(X^*\) and
\(X\).
\end{proof}

We can now prove the consistency between the canonical rate function and
the rate function predicted by the controlled SPDE.

\begin{theorem}
\label{thm:compatible-tilt-rate-equality}
Every path \(f\in\mathcal A_{\rm reg}\) satisfies
\begin{align}\label{rigorous-canonical-control-inequality}
 I_{[0,T]}^{\rm can}(f)\leq I_{[0,T]}^{\rm SPDE}(f).
\end{align}
If \(f\in\mathcal A_{\rm tilt}\) and \(p\) is a field from
Definition~\ref{def:regular-canonical-path}, then the two actions agree:
\begin{align}\label{compatible-tilt-rate-equality}
 I_{[0,T]}^{\rm can}(f)
 =I_{[0,T]}^{\rm SPDE}(f).
\end{align}
\end{theorem}

\begin{proof}
Let \((h,q)\in\mathcal U(f)\) and put
\(J_t=J_{f_t}^{h_t,q_t}\).  The skeleton equation gives
\(\mathscr D^*J_t=\partial_t f_t+\mathcal T(f_t)\).  By
\eqref{regular-path-entropy-identities} and the adjoint identity
\eqref{channel-adjoint-pairing},
\begin{align}\label{entropy-endpoint-current-identity}
 \frac12\bigl[\mathcal E(f_T)-\mathcal E(f_0)\bigr]
 =-\int_0^T J_t\left[-\frac12\mathscr D
 \frac{\delta\mathcal E}{\delta f}(f_t)\right]\,\d t.
\end{align}
Using \(J_t\) as a competitor in \eqref{canonical-density-action} and
applying Lemma~\ref{lem:Gaussian-Poisson-current-contraction}, we obtain
\begin{align*}
 I_{[0,T]}^{\rm can}(f)
 \leq\int_0^T\left\{
 \Psi_{f_t}(J_t)
 +\Psi_{f_t}^*\left(-\frac12\mathscr D
 \frac{\delta\mathcal E}{\delta f}(f_t)\right)
 -J_t\left[-\frac12\mathscr D
 \frac{\delta\mathcal E}{\delta f}(f_t)\right]
 \right\}\,\d t\leq\int_0^T\mathfrak K_{f_t}(h_t,q_t)\,\d t.
\end{align*}
Taking the infimum over \((h,q)\in\mathcal U(f)\) proves
\eqref{rigorous-canonical-control-inequality}.

Let now \(f\in\mathcal A_{\rm tilt}\).  For the field \(p\) in
Definition~\ref{def:regular-canonical-path}, define the Gaussian and
Poisson controls
\begin{align}\label{compatible-Hamiltonian-controls}
 h_t^p(c):=\sigma_{f_t}(c)^{\mathsf T}\mathscr Dp_t(c),\qquad q_t^p(c,z):=\exp\{z\cdot\mathscr Dp_t(c)\}.
\end{align}
Condition~\eqref{regular-tilt-control-cost} gives
\begin{align*}
 \int_0^T\mathfrak K_{f_t}(h_t^p,q_t^p)\,\d t<\infty.
\end{align*}
By Lemma~\ref{lem:Gaussian-Poisson-current-contraction}, applied with
\(\xi=\mathscr Dp_t\),
\begin{align*}
 J_{f_t}^{h_t^p,q_t^p}
 =D\Psi_{f_t}^*\left(
 -\frac12\mathscr D\frac{\delta\mathcal E}{\delta f}(f_t)
 +\mathscr Dp_t\right)
\end{align*}
for almost every \(t\).  It follows from
\eqref{smooth-Hamiltonian-tilted-equation} and the integrability condition
in Definition~\ref{def:regular-canonical-path} that
\begin{align*}
 \left\langle\partial_t f_t+\mathcal T(f_t),\phi\right\rangle
 =J_{f_t}^{h_t^p,q_t^p}[\mathscr D\phi]
\end{align*}
for almost every \(t\) and every
\(\phi\in\operatorname{Dom}(\mathscr D)\), and hence
\((h^p,q^p)\in\mathcal U(f)\).  We claim that the current
\begin{align*}
 D\Psi_{f_t}^*\left(
 -\frac12\mathscr D\frac{\delta\mathcal E}{\delta f}(f_t)
 +\mathscr Dp_t\right)
\end{align*}
minimizes \(\Psi_{f_t}\) under the constraint
\(\mathscr D^*J=\partial_t f_t+\mathcal T(f_t)\).  If \(J\) is
any other feasible current, then
\begin{align*}
 &\left\{J-D\Psi_{f_t}^*\left(
 -\frac12\mathscr D\frac{\delta\mathcal E}{\delta f}(f_t)
 +\mathscr Dp_t\right)\right\}
 \left[-\frac12\mathscr D
 \frac{\delta\mathcal E}{\delta f}(f_t)+\mathscr Dp_t\right]=0
\end{align*}
by \eqref{channel-adjoint-pairing}.  Applying
Lemma~\ref{lem:Fenchel-Young} to \(J\) and to the displayed current gives
\begin{align*}
 \Psi_{f_t}(J)
 &\geq J\left[-\frac12\mathscr D
 \frac{\delta\mathcal E}{\delta f}(f_t)+\mathscr Dp_t\right]
 -\Psi_{f_t}^*\left(-\frac12\mathscr D
 \frac{\delta\mathcal E}{\delta f}(f_t)+\mathscr Dp_t\right)\\
 &=D\Psi_{f_t}^*\left(
 -\frac12\mathscr D\frac{\delta\mathcal E}{\delta f}(f_t)
 +\mathscr Dp_t\right)
 \left[-\frac12\mathscr D
 \frac{\delta\mathcal E}{\delta f}(f_t)+\mathscr Dp_t\right]-\Psi_{f_t}^*\left(-\frac12\mathscr D
 \frac{\delta\mathcal E}{\delta f}(f_t)+\mathscr Dp_t\right)\\
 &=\Psi_{f_t}\left(D\Psi_{f_t}^*\left(
 -\frac12\mathscr D\frac{\delta\mathcal E}{\delta f}(f_t)
 +\mathscr Dp_t\right)\right).
\end{align*}
The first equality follows from the preceding annihilation identity, and
the last one is the equality case in
Lemma~\ref{lem:Fenchel-Young}, since the displayed current is the
derivative of \(\Psi_{f_t}^*\) at the shifted force.
We deduce that this current attains the pointwise infimum in
\eqref{canonical-density-action}.  Lemma~
\ref{lem:Gaussian-Poisson-current-contraction}, with
\(\xi=\mathscr Dp_t\), also shows that
\begin{align*}
 \Psi_{f_t}\left(D\Psi_{f_t}^*\left(
 -\frac12\mathscr D\frac{\delta\mathcal E}{\delta f}(f_t)
 +\mathscr Dp_t\right)\right)
 &+\Psi_{f_t}^*\left(-\frac12\mathscr D
 \frac{\delta\mathcal E}{\delta f}(f_t)\right)
 \\
 &-D\Psi_{f_t}^*\left(
 -\frac12\mathscr D\frac{\delta\mathcal E}{\delta f}(f_t)
 +\mathscr Dp_t\right)\left[-\frac12\mathscr D
 \frac{\delta\mathcal E}{\delta f}(f_t)\right]
 =\mathfrak K_{f_t}(h_t^p,q_t^p).
\end{align*}
Integrating in time and using
\eqref{entropy-endpoint-current-identity}, we obtain
\begin{align*}
 I_{[0,T]}^{\rm can}(f)
 =\int_0^T\mathfrak K_{f_t}(h_t^p,q_t^p)\,\d t.
\end{align*}
By the definition of \(I^{\rm SPDE}\), it is no larger than the right-hand
side.  Together with the opposite inequality proved above, this gives the
claimed equality and also shows that
\((h^p,q^p)\) minimizes the SPDE control cost.
\end{proof}

For the Boltzmann cosh potential there is no Gaussian part, and the
controls in \eqref{compatible-Hamiltonian-controls} reduce to
\begin{align}\label{Boltzmann-compatible-two-sided-controls}
 q_t^p(c,+1)=e^{\mathscr D_{\mathrm B}p_t(c)},
 \qquad
 q_t^p(c,-1)=e^{-\mathscr D_{\mathrm B}p_t(c)}.
\end{align}
The pre--post-collisional involution turns the two-sided control into the
one-sided collision-rate multiplier
\[
 q_t(c)=e^{\mathscr D_{\mathrm B}p_t(c)}.
\]
This is the control used in
\cite[Appendix~B]{WuPoissonBoltzmann}.  Theorem~
\ref{thm:compatible-tilt-rate-equality} therefore recovers, on the same
smooth exponentially tilted class (except the regularity conditions on $p$ are slightly different), the equality between the canonical
density action and the Poisson control action associated with
\cite[Eq.~(1.24)]{WuPoissonBoltzmann}.

\section{Hamiltonian consistency from fluctuating Boltzmann to LLNS}
\label{sec:Boltzmann-LLNS}

In this section, we study the stochastic hydrodynamic limit from the
fluctuating Boltzmann equation to the Landau--Lifshitz--Navier--Stokes
equations through the convergence of their Hamiltonians.  This point of
view is inspired by the Hamiltonian approach of Feng, Mikami, and Zimmer to
hydrodynamic fluctuations \cite{FengMikamiZimmer2021}.  We begin with the
two mesoscopic models and the roles of the kinetic fluctuation parameter
\(\varepsilon\), the hydrodynamic parameter \(\delta\), and the limiting
noise amplitude \(\mathfrak a\).  

\subsection{The two mesoscopic models}

We work on \(\T^3_x\times\R^3_v\) around the normalized Maxwellian
\begin{align*}
 M(v):=(2\pi)^{-3/2}e^{-|v|^2/2}.
\end{align*}
For velocity functions and space--velocity functions, respectively, set
\begin{align}\label{Maxwellian-pairings}
 \langle h,k\rangle_M
 &:=\int_{\R^3}h(v)k(v)M(v)\,\d v,\notag\\
 ((p,q))_M
 &:=\int_{\T^3}\langle p(x,\cdot),q(x,\cdot)\rangle_M\,\d x.
\end{align}
The collision-channel space used throughout this section is
\begin{align}\label{section-five-collision-space}
 \mathscr C_{\rm B}:=
 \T^3_x\times\R^3_v\times\R^3_{v_*}\times\mathbb S^2,
 \qquad \d c:=\d x\,\d v\,\d v_*\,\d\sigma.
\end{align}
We use the following normalization of the Boltzmann operator:
\begin{align}\label{Boltzmann-weak-normalization}
 \int_{\R^3}Q_{\rm B}(f,f)(v)p(v)\,\d v
 =\frac12\int_{\R^3\times\R^3\times\mathbb S^2}
 Bff_*\,\mathscr D_{\rm B}p\,\d v\,\d v_*\,\d\sigma.
\end{align}
Here and below, microreversibility means the following two integral
identities.  For every nonnegative measurable function
\(G:\mathscr C_{\rm B}\to[0,\infty]\),
\begin{align}\label{collision-measure-microreversibility}
 \int_{\mathscr C_{\rm B}}
 G(x,v,v_*,\sigma)B(x,v,v_*,\sigma)\,
 \d x\,\d v\,\d v_*\,\d\sigma&=\int_{\mathscr C_{\rm B}}
 G(x,v_*,v,-\sigma)B(x,v,v_*,\sigma)\,
 \d x\,\d v\,\d v_*\,\d\sigma\notag\\
 &=\int_{\mathscr C_{\rm B}}
 G\left(x,v',v_*',\frac{v-v_*}{|v-v_*|}\right)
 B(x,v,v_*,\sigma)\,
 \d x\,\d v\,\d v_*\,\d\sigma,
\end{align}
where
\[
 v'=\frac{v+v_*}{2}+\frac{|v-v_*|}{2}\sigma,
 \qquad
 v_*'=\frac{v+v_*}{2}-\frac{|v-v_*|}{2}\sigma.
\]
An arbitrary value may be assigned to
\((v-v_*)/|v-v_*|\) on the null set \(v=v_*\).  

We begin with the deterministic scaling.  For each \(\delta\in(0,1)\), the
diffusive incompressible rescaling of the Boltzmann equation is
\begin{align*}
 \partial_tF^\delta+\frac1\delta v\cdot\nabla_xF^\delta
 =\frac1{\delta^2}Q_{\rm B}(F^\delta,F^\delta). 
\end{align*}
This is the classical scaling
under which the Boltzmann equation converges to the incompressible
Navier--Stokes equations; see, for example,
\cite{GolseSaintRaymondNSLimit}.  On the macroscopic time scale, free
transport is therefore accelerated by \(\delta^{-1}\), while the collision
frequency is accelerated by \(\delta^{-2}\).

For the fluctuating equation, the collision clock is accelerated in the
same way.  Before hydrodynamic rescaling, one effective
collision has jump size \(\varepsilon\mathcal J_c\) and accepted rate
\((2\varepsilon)^{-1}B(c)F(c)F_*(c)\,\d c\).  The diffusive scaling
multiplies this rate by \(\delta^{-2}\), while the mass \(\varepsilon\)
transferred by one event remains unchanged.  Thus the rescaling accelerates
the Poisson intensity and leaves the jump size unchanged.

Here \(\delta\) sets the hydrodynamic scale, while \(\varepsilon\) records
the effective particle mass in the normalized kinetic field and controls
the size of its fluctuations.  Since they describe different
approximations, we keep them independent until the joint limit is imposed.

The stochastic normalization follows by accelerating every collision clock
by
\(\delta^{-2}\).  A Poisson random measure with base intensity
\(\delta^{-2}\d t\,\lambda(\d c)\,\d u\), thinned at
\[
 u\leq\frac1{2\varepsilon}B(c)F(c)F_*(c),
\]
has the same accepted intensity as a
Poisson random measure with base intensity
\(\d t\,\lambda(\d c)\,\d u\), thinned at
\begin{align*}
 u\leq\frac1{2\varepsilon\delta^2}B(c)F(c)F_*(c).
\end{align*}
Both accepted intensities equal
\[
 \frac1{2\varepsilon\delta^2}
 B(c)F(c)F_*(c)\,\d t\,\lambda(\d c).
\]
We thus multiply the number
of collision events per unit macroscopic time by \(\delta^{-2}\), while
each accepted event continues to change the kinetic field by
\(\varepsilon\mathcal J_c\). Then we derive the rescaled fluctuating Boltzmann equation
\begin{align}\label{rescaled-fluctuating-Boltzmann-LLNS}
 \d F_t^{\delta,\varepsilon}
 +\frac1\delta v\cdot\nabla_xF_t^{\delta,\varepsilon}\,\d t
 ={}&\frac1{\delta^2}Q_{\rm B}
 (F_t^{\delta,\varepsilon},F_t^{\delta,\varepsilon})\,\d t\notag\\
 &+\varepsilon\int_{\mathscr C_{\rm B}\times[0,\infty)}
 \mathcal J_c
 \boldsymbol1_{\{u\leq(2\varepsilon\delta^2)^{-1}
 BF_{t-}^{\delta,\varepsilon}F_{t-,*}^{\delta,\varepsilon}\}}
 \widetilde{\mathcal N}(\d t,\d c,\d u).
\end{align}
Here \(\widetilde{\mathcal N}\) is the compensated Poisson random measure
associated with a Poisson random measure of intensity
\(\d t\,\d c\,\d u\) on
\([0,\infty)\times\mathscr C_{\rm B}\times[0,\infty)\).

For a fluctuation amplitude \(\mathfrak a\geq0\), the target velocity equations are the incompressible
Landau--Lifshitz--Navier--Stokes equations
\begin{align}\label{LLNS-model}
 \d u+\Pi\nabla_x\cdot(u\otimes u)\,\d t
 =\mu\Delta u\,\d t
 -\mathfrak a\sqrt{2\mu}\,
 \Pi\nabla_x\cdot\d W^{\rm F},
 \qquad \nabla_x\cdot u=0,
\end{align}
Here \(\Pi\) is the Helmholtz projection and \(W^{\rm F}\) is symmetric
trace-free space--time white noise.  The two tensor divergences are defined
componentwise by
\[
 \bigl(\nabla_x\cdot(u\otimes u)\bigr)_i
 :=\sum_{j=1}^3\partial_j(u_i u_j),
 \qquad
 \bigl(\nabla_x\cdot W^{\rm F}\bigr)_i
 :=\sum_{j=1}^3\partial_jW^{\rm F}_{ij},
 \qquad 1\leq i\leq3.
\]
The covariance of \(W^{\rm F}\) is
\begin{align}\label{LLNS-stress-covariance}
 \E[\dot W^{\rm F}_{ij}(t,x)\dot W^{\rm F}_{kl}(s,y)]
 =\mathsf C_{ij,kl}\delta(t-s)\delta(x-y),
 \qquad 1\leq i,j,k,l\leq3,
\end{align}
with
\begin{align}\label{trace-free-isotropy-tensor}
 \mathsf C_{ij,kl}
 :=\delta_{ik}\delta_{jl}+\delta_{il}\delta_{jk}
 -\frac23\delta_{ij}\delta_{kl},
 \qquad 1\leq i,j,k,l\leq3.
\end{align}

We study \eqref{rescaled-fluctuating-Boltzmann-LLNS} through its  exponential-generator Hamiltonian, both of which can be
defined from the formal It\^o's formula, and these Hamiltonians
record the deterministic drifts together with all collision cumulants.  In
the following, we restrict it to the local-equilibrium manifold, add the
standard stress correction, and prove convergence to the Hamiltonian of
\eqref{LLNS-model}.  

\subsection{The corrected Boltzmann and LLNS Hamiltonians}

Here the Hamiltonian is the local cumulant-generating function of the
increments of a Markov process.  We give its precise meaning before using
the corresponding algebraic object for the fluctuating Boltzmann equation.
Let \(E\) be a Polish state space and let
\((X_t)_{t\geq0}\) be an \(E\)-valued time-homogeneous Markov process with
transition semigroup \((P_t)_{t\geq0}\).  Suppose that this semigroup is
strongly continuous on a Banach space \(\mathcal B\) of bounded measurable
observables.  Thus
\begin{align*}
 P_hG(x):=\E_x[G(X_h)],
 \qquad G\in\mathcal B,\quad x\in E,
\end{align*}
where \(\E_x\) denotes expectation conditional on \(X_0=x\).  Its linear
generator
\(\mathcal L:\operatorname{Dom}(\mathcal L)\subset\mathcal B
\to\mathcal B\) is defined by
\begin{align*}
 \mathcal LG
 :=\lim_{h\downarrow0}\frac{P_hG-G}h,
 \qquad G\in\operatorname{Dom}(\mathcal L),
\end{align*}
where the limit is taken in \(\mathcal B\).  See
\cite[Chapter~1]{EthierKurtzMarkov} for the standard generator and
martingale-problem framework.

Let \(\Phi:E\to\R\) be a real-valued test observable such that
\(e^\Phi\in\operatorname{Dom}(\mathcal L)\) and the exponential moment
below is finite.  The same definition applies through the extended
generator when \(\Phi\) is unbounded and the required local integrability
holds.  The exponential, or nonlinear, generator at \(x\in E\) is
\begin{align}\label{exponential-generator-general-definition}
 \mathcal H(x,\Phi)
 :=e^{-\Phi(x)}\mathcal L e^\Phi(x)
 =\lim_{h\downarrow0}\frac1h
 \log\E_x\exp\{\Phi(X_h)-\Phi(x)\}.
\end{align}
To verify the second identity, we use the expansion
\begin{align*}
 \E_x\exp\{\Phi(X_h)-\Phi(x)\}
 &=e^{-\Phi(x)}P_he^\Phi(x)\\
 &=1+h e^{-\Phi(x)}\mathcal Le^\Phi(x)+o(h).
\end{align*}
Applying \(\log(1+r)=r+o(r)\) gives the second identity in
\eqref{exponential-generator-general-definition}. 

This nonlinear transform of a Markov generator is standard in the
Hamilton--Jacobi approach to large deviations, see
\cite[Chapters~1 and~5]{FengKurtzLargeDeviations}.  The generator
determines \(\mathcal H\), and conversely, on a common test algebra on which
\(\Phi\in\operatorname{Dom}(\mathcal L)\),
\(e^{\alpha\Phi}\in\operatorname{Dom}(\mathcal L)\) for \(\alpha\) near
zero, and differentiation through \(\mathcal L\) is valid,
\begin{align}\label{recover-generator-from-Hamiltonian}
 \mathcal L\Phi(x)
 =\left.\frac{\d}{\d\alpha}
 \mathcal H(x,\alpha\Phi)\right|_{\alpha=0}.
\end{align}
Using the definition of the nonlinear generator and the linearity of
\(\mathcal L\), we have
\begin{align*}
 \frac{\d}{\d\alpha}\mathcal H(x,\alpha\Phi)
 &={\frac{\d}{\d\alpha}}
 \left[e^{-\alpha\Phi(x)}
 \mathcal L\bigl(e^{\alpha\Phi}\bigr)(x)\right]\\
 &=-\Phi(x)e^{-\alpha\Phi(x)}
 \mathcal L\bigl(e^{\alpha\Phi}\bigr)(x)
 +e^{-\alpha\Phi(x)}
 \mathcal L\bigl(\Phi e^{\alpha\Phi}\bigr)(x).
\end{align*}
At \(\alpha=0\), this becomes
\[
 \left.\frac{\d}{\d\alpha}
 \mathcal H(x,\alpha\Phi)\right|_{\alpha=0}
 =-\Phi(x)\mathcal L1(x)+\mathcal L\Phi(x).
\]
Since \((P_t)_{t\geq0}\) is a conservative Markov semigroup,
\(P_t1=1\), and therefore \(\mathcal L1=0\).  This proves
\eqref{recover-generator-from-Hamiltonian}.
It follows that the linear Markov generator \(\mathcal L\) can be recovered from the
exponential-generator Hamiltonian \(\mathcal H\), provided that
\(\mathcal H(x,\alpha\Phi)\) is known for a sufficiently rich class of
observables and for \(\alpha\) near zero.

We now specify the test class used for the kinetic state.  Let
\(\mathcal M(\mathsf Z_{\rm B})\) denote the vector space of finite signed
Radon measures on \(\mathsf Z_{\rm B}=\T^3\times\R^3\).  We identify each
\(F\in L_+^1(\mathsf Z_{\rm B})\) with the measure
\(F(x,v)\,\d x\,\d v\).  For
\(c=(x,v,v_*,\sigma)\), the collision increment is the finite signed
measure
\[
 \mathcal J_c
 :=\delta_{(x,v')}+\delta_{(x,v_*')}
   -\delta_{(x,v)}-\delta_{(x,v_*)}
 \in\mathcal M(\mathsf Z_{\rm B}).
\]
In particular, \(\mathcal J_c\) does not have an \(L^1\)-density.  A
cylindrical observable is a map
\[
 \Gamma:\mathcal M(\mathsf Z_{\rm B})\longrightarrow\R
\]
of the form
\[
 \Gamma(F)
 =\gamma\bigl(\langle F,\phi_1\rangle,\ldots,
                  \langle F,\phi_N\rangle\bigr),
\]
where \(N\in\mathbb N\),
\(\gamma\in C_b^1(\R^N)\), and
\(\phi_1,\ldots,\phi_N\in C_c^\infty(\mathsf Z_{\rm B})\).  For
\(F,H\in\mathcal M(\mathsf Z_{\rm B})\), its directional differential is
\[
 D\Gamma(F)[H]
 =\sum_{r=1}^N
 \partial_r\gamma\bigl(\langle F,\phi_1\rangle,\ldots,
                         \langle F,\phi_N\rangle\bigr)
 \langle H,\phi_r\rangle.
\]
Equivalently,
\[
 D\Gamma(F)=\sum_{r=1}^N
 \partial_r\gamma\bigl(\langle F,\phi_1\rangle,\ldots,
                         \langle F,\phi_N\rangle\bigr)\phi_r.
\]
The generator below is evaluated at sufficiently regular nonnegative
densities \(F\in L_+^1(\mathsf Z_{\rm B})\) for which its collision
integrals are finite.  Since
\[
 \langle\mathcal J_c,\phi_r\rangle
 =\mathscr D_{\rm B}\phi_r(c)
 =\phi_r(x,v')+\phi_r(x,v_*')
  -\phi_r(x,v)-\phi_r(x,v_*),
\]
a collision changes this observable by
\begin{align*}
 \Gamma(F+\varepsilon\mathcal J_c)-\Gamma(F)=\gamma\Bigl(
 \langle F,\phi_1\rangle+\varepsilon\mathscr D_{\rm B}\phi_1(c),
 \ldots,
 \langle F,\phi_N\rangle+\varepsilon\mathscr D_{\rm B}\phi_N(c)
 \Bigr)
 -\gamma\bigl(\langle F,\phi_1\rangle,\ldots,
                \langle F,\phi_N\rangle\bigr).
\end{align*}

The resulting algebraic generator is
\begin{align}\label{rescaled-Boltzmann-generator}
 (\mathcal L_{\delta,\varepsilon}^{\rm B}\Gamma)(F)
 :={}&-\frac1\delta
 \langle v\cdot\nabla_xF,D\Gamma(F)\rangle\notag\\
 &+\frac1{2\varepsilon\delta^2}
 \int_{\mathscr C_{\rm B}}BFF_*
 \bigl[\Gamma(F+\varepsilon\mathcal J_c)-\Gamma(F)\bigr]\,\d c.
\end{align}
This is the standard generator of a process with deterministic transport
and state-dependent jumps, we derive it from the compensated formulation
in \eqref{rescaled-fluctuating-Boltzmann-LLNS}.  In the It\^o formula for
Banach-space-valued stochastic integrals driven by a compensated Poisson
random measure, see \cite[formula~(3)]{RudigerZiglioIto}, take the jump
integrand to be
\[
 G_{\delta,\varepsilon}(F;c,u)
 :=\varepsilon\mathcal J_c
 \boldsymbol1_{\{u\leq(2\varepsilon\delta^2)^{-1}B(c)F(c)F_*(c)\}}.
\]
Applying this formula formally to
\eqref{rescaled-fluctuating-Boltzmann-LLNS}, we obtain
\begin{align*}
 \d\Gamma(F_t^{\delta,\varepsilon})
 ={}&D\Gamma(F_{t-}^{\delta,\varepsilon})\left[
 -\frac1\delta v\cdot\nabla_xF_{t-}^{\delta,\varepsilon}
 +\frac1{\delta^2}Q_{\rm B}
 (F_{t-}^{\delta,\varepsilon},F_{t-}^{\delta,\varepsilon})\right]\,\d t\\
 &+\int_{\mathscr C_{\rm B}\times[0,\infty)}
 \left[\Gamma\bigl(F_{t-}^{\delta,\varepsilon}
 +G_{\delta,\varepsilon}(F_{t-}^{\delta,\varepsilon};c,u)\bigr)
 -\Gamma(F_{t-}^{\delta,\varepsilon})\right]
 \widetilde{\mathcal N}(\d t,\d c,\d u)\\
 &+\int_{\mathscr C_{\rm B}\times[0,\infty)}
 \Bigl[\Gamma\bigl(F_{t-}^{\delta,\varepsilon}
 +G_{\delta,\varepsilon}(F_{t-}^{\delta,\varepsilon};c,u)\bigr)
 -\Gamma(F_{t-}^{\delta,\varepsilon})\notag\\
 &\hspace{43mm}
 -D\Gamma(F_{t-}^{\delta,\varepsilon})
 \bigl[G_{\delta,\varepsilon}
 (F_{t-}^{\delta,\varepsilon};c,u)\bigr]\Bigr]
 \,\d c\,\d u\,\d t.
\end{align*}
The last integrand vanishes above the thinning threshold and, below the
threshold, becomes
\[
 \Gamma(F_{t-}^{\delta,\varepsilon}+\varepsilon\mathcal J_c)
 -\Gamma(F_{t-}^{\delta,\varepsilon})
 -\varepsilon D\Gamma(F_{t-}^{\delta,\varepsilon})[\mathcal J_c].
\]
Integrating in \(u\), we have
\begin{align*}
 &\int_0^\infty
 \Bigl[\Gamma\bigl(F_{t-}^{\delta,\varepsilon}
 +G_{\delta,\varepsilon}(F_{t-}^{\delta,\varepsilon};c,u)\bigr)
 -\Gamma(F_{t-}^{\delta,\varepsilon})
 -D\Gamma(F_{t-}^{\delta,\varepsilon})
 \bigl[G_{\delta,\varepsilon}
 (F_{t-}^{\delta,\varepsilon};c,u)\bigr]\Bigr]\,\d u\\
 &\qquad={}
 \frac{B(c)F_{t-}^{\delta,\varepsilon}(c)
 F_{t-,*}^{\delta,\varepsilon}(c)}{2\varepsilon\delta^2}
 \left[\Gamma(F_{t-}^{\delta,\varepsilon}+\varepsilon\mathcal J_c)
 -\Gamma(F_{t-}^{\delta,\varepsilon})
 -\varepsilon D\Gamma(F_{t-}^{\delta,\varepsilon})[\mathcal J_c]\right].
\end{align*}
We thus obtain the following formal drift of
\(\Gamma(F_t^{\delta,\varepsilon})\):
\begin{align*}
 &D\Gamma(F_{t-}^{\delta,\varepsilon})\left[
 -\frac1\delta v\cdot\nabla_xF_{t-}^{\delta,\varepsilon}
 +\frac1{\delta^2}Q_{\rm B}
 (F_{t-}^{\delta,\varepsilon},F_{t-}^{\delta,\varepsilon})\right]\\
 &\quad+\frac1{2\varepsilon\delta^2}
 \int_{\mathscr C_{\rm B}}B
 F_{t-}^{\delta,\varepsilon}F_{t-,*}^{\delta,\varepsilon}
 \left[
 \Gamma(F_{t-}^{\delta,\varepsilon}+\varepsilon\mathcal J_c)
 -\Gamma(F_{t-}^{\delta,\varepsilon})
 -\varepsilon D\Gamma(F_{t-}^{\delta,\varepsilon})[\mathcal J_c]
 \right]\,\d c.
\end{align*}
The second line is the usual compensation term in the generator of a jump
semimartingale; see, for example,
\cite[Chapters~4 and~6]{ApplebaumLevy}.  Moreover,
\eqref{Boltzmann-weak-normalization} gives the distributional identity
\[
 Q_{\rm B}(F_{t-}^{\delta,\varepsilon},F_{t-}^{\delta,\varepsilon})
 =\frac12\int_{\mathscr C_{\rm B}}B
 F_{t-}^{\delta,\varepsilon}F_{t-,*}^{\delta,\varepsilon}
 \mathcal J_c\,\d c.
\]
It follows that
\begin{align*}
 \frac1{\delta^2}D\Gamma(F_{t-}^{\delta,\varepsilon})
 \left[Q_{\rm B}
 (F_{t-}^{\delta,\varepsilon},F_{t-}^{\delta,\varepsilon})\right]=
 \frac1{2\varepsilon\delta^2}
 \int_{\mathscr C_{\rm B}}B
 F_{t-}^{\delta,\varepsilon}F_{t-,*}^{\delta,\varepsilon}
 \varepsilon D\Gamma(F_{t-}^{\delta,\varepsilon})[\mathcal J_c]\,\d c.
\end{align*}
This term cancels the linear compensation term in the previous display and
leaves the uncompensated jump term in
\eqref{rescaled-Boltzmann-generator}.

Since the formal update \(F+\varepsilon\mathcal J_c\) may fail to preserve
nonnegativity for a continuum density, we therefore use
\eqref{rescaled-Boltzmann-generator} as an algebraic jump operator.  Its
exponential symbol below is well defined under the integrability hypotheses of the theorem.

For \(q\in C_c^\infty(\mathsf Z_{\rm B})\), consider the cylindrical
observable \(F\mapsto e^{\langle F,q\rangle}\).  Its differential and its
increment under one collision are
\begin{align*}
 D\bigl(e^{\langle\,\cdot\,,q\rangle}\bigr)(F)[H]
 &=e^{\langle F,q\rangle}\langle H,q\rangle,\\
 e^{\langle F+\varepsilon\mathcal J_c,q\rangle}
 -e^{\langle F,q\rangle}
 &=e^{\langle F,q\rangle}
 \left(e^{\varepsilon\langle\mathcal J_c,q\rangle}-1\right)=e^{\langle F,q\rangle}
 \left(e^{\varepsilon\mathscr D_{\rm B}q(c)}-1\right).
\end{align*}
Substituting these identities into
\eqref{rescaled-Boltzmann-generator}, we obtain
\begin{align*}
 \left(\mathcal L_{\delta,\varepsilon}^{\rm B}
 e^{\langle\,\cdot\,,q\rangle}\right)(F)
 ={}&-\frac1\delta e^{\langle F,q\rangle}
 \langle v\cdot\nabla_xF,q\rangle\\
 &+\frac{e^{\langle F,q\rangle}}{2\varepsilon\delta^2}
 \int_{\mathscr C_{\rm B}}BFF_*
 \left(e^{\varepsilon\mathscr D_{\rm B}q(c)}-1\right)\,\d c.
\end{align*}
Multiplication by \(e^{-\langle F,q\rangle}\) then gives the
Boltzmann exponential symbol
\begin{align}\label{rescaled-Boltzmann-Hamiltonian}
 \mathcal H_{\delta,\varepsilon}^{\rm B}(F,q)
 :={}&e^{-\langle F,q\rangle}
 \left(\mathcal L_{\delta,\varepsilon}^{\rm B}
 e^{\langle\,\cdot\,,q\rangle}\right)(F)\notag\\
 ={}&-\frac1\delta\langle v\cdot\nabla_xF,q\rangle+\frac1{2\varepsilon\delta^2}
 \int_{\mathscr C_{\rm B}}BFF_*
 \bigl(e^{\varepsilon\mathscr D_{\rm B}q(c)}-1\bigr)\,\d c.
\end{align}
Using
\eqref{Boltzmann-weak-normalization}, we separate it into the deterministic
drift and the centered collision cumulant:
\begin{align}\label{Boltzmann-Hamiltonian-drift-cumulant-splitting}
 \mathcal H_{\delta,\varepsilon}^{\rm B}(F,q)
 ={}&-\frac1\delta\langle v\cdot\nabla_xF,q\rangle
 +\frac1{\delta^2}\langle Q_{\rm B}(F,F),q\rangle\notag\\
 &+\frac1{2\varepsilon\delta^2}
 \int_{\mathscr C_{\rm B}}BFF_*
 \left(e^{\varepsilon\mathscr D_{\rm B}q(c)}
 -1-\varepsilon\mathscr D_{\rm B}q(c)\right)\d c.
\end{align}
For density \(\varrho>0\), bulk velocity \(U\in\mathbb R^3\), and temperature \(T>0\),
write
\begin{align*}
 \mathcal M_{\varrho,U,T}(v)
 :=\varrho(2\pi T)^{-3/2}
 \exp\left(-\frac{|v-U|^2}{2T}\right).
\end{align*}
Define the normalized global Maxwellian by
\[
 M(v):=\mathcal M_{1,0,1}(v)
 =(2\pi)^{-3/2}e^{-|v|^2/2}.
\]
Given
\[
 u\in C^\infty(\T^3;\R^3),
 \qquad \nabla_x\cdot u=0, \qquad \theta\in C^\infty(\T^3;\R), 
\]
and \(\delta\in(0,1)\) satisfying
\(\delta\|\theta\|_{L^\infty(\T^3)}<1\), define the local Maxwellian
\begin{align}\label{local-Maxwellian-family}
 F_\delta[u,\theta](x,v)
 :=\mathcal M_{1-\delta\theta(x),\,\delta u(x),\,
                    1+\delta\theta(x)}(v).
\end{align}
Direct substitution gives
\begin{align*}
 \frac{F_\delta[u,\theta](x,v)}{M(v)}
 ={}&(1-\delta\theta(x))(1+\delta\theta(x))^{-3/2}\exp\left\{
 \frac{\delta\theta(x)|v|^2+2\delta u(x)\cdot v
       -\delta^2|u(x)|^2}
 {2(1+\delta\theta(x))}
 \right\}.
\end{align*}
As \(\delta\downarrow0\),
\begin{align*}
 &(1-\delta\theta)(1+\delta\theta)^{-3/2}
 =1-\frac52\delta\theta+O(\delta^2),\\
 &\exp\left\{
 \frac{\delta\theta|v|^2+2\delta u\cdot v-\delta^2|u|^2}
 {2(1+\delta\theta)}
 \right\}
 =1+\delta\left(u\cdot v+\frac{\theta|v|^2}{2}\right)
 +O(\delta^2),
\end{align*}
locally uniformly in \((x,v)\).  Multiplying these two expansions, we obtain
\begin{align}\label{hydrodynamic-profile}
 \frac{F_\delta[u,\theta]-M}{\delta M}
 ={}&u\cdot v+\theta\frac{|v|^2-5}{2}+O(\delta)
 \longrightarrow
 g_0[u,\theta]
 :=u\cdot v+\theta\frac{|v|^2-5}{2}.
\end{align}

For velocity tests \(h,k\in C_c^\infty(\R^3)\), introduce the nonnegative
symmetric Dirichlet form
\begin{align}\label{linearized-Boltzmann-form}
 \mathscr E_{\rm B}(h,k)
 :=\frac14\int_{\R^3\times\R^3\times\mathbb S^2}
 BMM_*\,\mathscr D_{\rm B}h\,\mathscr D_{\rm B}k
 \,\d v\,\d v_*\,\d\sigma.
\end{align}
\begin{assumption}[Bounded isotropic cutoff Maxwell kernel]
\label{ass:Boltzmann-kernel-conditions}
There is a measurable function \(b:[-1,1]\to(0,\infty)\) and constants
\(0<b_-\leq b_+<\infty\) such that, for \(v\neq v_*\),
\begin{align}\label{linearized-Boltzmann-kernel-conditions}
 B(v-v_*,\sigma)
 =b\left(\frac{v-v_*}{|v-v_*|}\cdot\sigma\right),
 \qquad
 b_-\leq b(s)\leq b_+
 \quad\text{for almost every }s\in[-1,1].
\end{align}
On the set \(v=v_*\), the kernel \(B\) may be assigned any value.   
\end{assumption}

From this assumption, we can see that \(B\) is a bounded,
strictly positive, rotationally invariant cutoff Maxwell kernel, and it has
the exchange and pre--postcollisional symmetries in
\eqref{collision-measure-microreversibility}. This is the bounded-kernel
hypothesis used in \cite[assumption (H1) and Proposition~1.4]
{GolseSaintRaymondNSLimit}, written in the postcollisional-direction
parametrization used here.  

We next introduce the linearized Boltzmann collision operator.  For a 
smooth velocity function \(h\), set
\begin{align}\label{linearized-Boltzmann-strong-operator}
 Lh:=-M^{-1}\bigl[Q_{\rm B}(Mh,M)+Q_{\rm B}(M,Mh)\bigr].
\end{align}
Also define the stress and heat-flux modes by
\begin{align}\label{Boltzmann-stress-mode}
 \mathsf A_{ij}(v):=v_iv_j-\frac{|v|^2}{3}\delta_{ij},
 \qquad 1\leq i,j\leq3,
 \qquad
 \mathsf K(v):=v\frac{|v|^2-5}{2}.
\end{align}

\begin{lemma}
\label{lem:linearized-Boltzmann-kernel}
Suppose that \(B\) satisfies
Assumption~\ref{ass:Boltzmann-kernel-conditions}.  Then the following
properties hold.
\begin{enumerate}
\item[(i)] The operator \(L\) extends to a bounded nonnegative self-adjoint
Fredholm operator on \(L^2(M\,\d v)\).  For every
\(h,k\in L^2(M\,\d v)\),
\begin{align}\label{linearized-Boltzmann-operator}
 \langle h,Lk\rangle_M
 =\mathscr E_{\rm B}(h,k)
 =\frac14\int_{\R^3\times\R^3\times\mathbb S^2}
 BMM_*\,\mathscr D_{\rm B}h\,\mathscr D_{\rm B}k
 \,\d v\,\d v_*\,\d\sigma,
\end{align}
where the last integral is understood by density.

\item[(ii)] The nullspace of \(L\) is
\begin{align}\label{linearized-kernel-assumption}
 \ker L=\operatorname{span}\{1,v_1,v_2,v_3,|v|^2\}.
\end{align}

\item[(iii)]
There is a number \(\lambda_{\mathsf A}>0\) such that
\begin{align}\label{stress-mode-eigenvalue}
 L\mathsf A_{ij}=\lambda_{\mathsf A}\mathsf A_{ij},
 \qquad 1\leq i,j\leq3.
\end{align}
Consequently, the stress cell problems
\begin{align}\label{Boltzmann-stress-cell-problem}
 L\widehat{\mathsf A}_{ij}=\mathsf A_{ij},
 \qquad
 \widehat{\mathsf A}_{ij}\perp\ker L,
 \qquad 1\leq i,j\leq3,
\end{align}
have the unique solutions
\begin{align}\label{explicit-stress-corrector}
 \widehat{\mathsf A}_{ij}
 =\lambda_{\mathsf A}^{-1}\mathsf A_{ij}, 
\end{align}
and they are even in \(v\).

\item[(iv)]
There is a number \(\mu>0\) such that
\begin{align}\label{stress-isotropy-identity}
 \langle\widehat{\mathsf A}_{ij},\mathsf A_{kl}\rangle_M
 =\mu\mathsf C_{ij,kl},
 \qquad 1\leq i,j,k,l\leq3,
\end{align}
where
\begin{align}\label{stress-isotropy-tensor}
 \mathsf C_{ij,kl}
 :=\delta_{ik}\delta_{jl}+\delta_{il}\delta_{jk}
 -\frac23\delta_{ij}\delta_{kl}.
\end{align}
Equivalently,
\begin{align}\label{stress-viscosity-explicit-sum}
 \mu=\frac1{10}\sum_{i,j=1}^3
 \langle\widehat{\mathsf A}_{ij},\mathsf A_{ij}\rangle_M
 =\lambda_{\mathsf A}^{-1}.
\end{align}

\item[(v)]
The correctors satisfy all integrability conditions used below.  For every
\(\beta\in(0,1)\), with
\(M_\beta(v):=(\beta/2\pi)^{3/2}e^{-\beta|v|^2/2}\),
\begin{align}\label{stress-corrector-transport-moment}
 \int_{\R^3}M_\beta(v)(1+|v|^3)
 \left(\sum_{i,j=1}^3
 |\widehat{\mathsf A}_{ij}(v)|^2\right)^{1/2}\d v<\infty.
\end{align}
Moreover,
\begin{align}\label{finite-collision-activity}
 \int_{\mathscr C_{\rm B}}B(c)M_\beta(v)M_\beta(v_*)\,\d c<\infty.
\end{align}
For every \(\varphi\in C^\infty(\T^3;\R^3)\), there is an \(\eta>0\)
such that, with
\(\nabla_{\rm sym}\varphi
:=(\nabla\varphi+(\nabla\varphi)^{\mathsf T})/2\),
\begin{align}\label{stress-corrector-exponential-moment}
 &\int_{\mathscr C_{\rm B}}BM_\beta M_{\beta,*}
 \left|\mathscr D_{\rm B}\left(
 \sum_{i,j=1}^3(\nabla_{\rm sym}\varphi)_{ij}
 \widehat{\mathsf A}_{ij}\right)\right|^2
 \exp\left\{\eta\left|\mathscr D_{\rm B}\left(
 \sum_{i,j=1}^3(\nabla_{\rm sym}\varphi)_{ij}
 \widehat{\mathsf A}_{ij}\right)\right|\right\}\,\d c<\infty.
\end{align}
\end{enumerate}
\end{lemma}

\begin{proof}
By the definition of \(Q_{\rm B}\), 
\begin{align*}
 \bigl[Q_{\rm B}(Mk,M)+Q_{\rm B}(M,Mk)\bigr](v)=
 \int_{\R^3\times\mathbb S^2}B
 \left[M'M_*'(k'+k_*')-MM_*(k+k_*)\right]
 \,\d v_*\,\d\sigma.
\end{align*}
Multiplying by \(h(v)\) and integrating in \(v\), we obtain
\begin{align*}
 &\int_{\R^3}h(v)
 \bigl[Q_{\rm B}(Mk,M)+Q_{\rm B}(M,Mk)\bigr](v)\,\d v\\
 &\quad={}
 \int_{\R^3\times\R^3\times\mathbb S^2}
 B\left[M'M_*'(k'+k_*')-MM_*(k+k_*)\right]h
 \,\d v\,\d v_*\,\d\sigma.
\end{align*}
For the gain term, we use the pre--postcollisional change of variables and
the microreversibility of the collision measure to obtain
\begin{align*}
 \int_{\R^3\times\R^3\times\mathbb S^2}
 BM'M_*'(k'+k_*')h\,\d v\,\d v_*\,\d\sigma=
 \int_{\R^3\times\R^3\times\mathbb S^2}
 BMM_*(k+k_*)h'\,\d v\,\d v_*\,\d\sigma.
\end{align*}
We therefore have
\begin{align*}
 \int_{\R^3}h(v)
 \bigl[Q_{\rm B}(Mk,M)+Q_{\rm B}(M,Mk)\bigr](v)\,\d v=\int_{\R^3\times\R^3\times\mathbb S^2}
 BMM_*\,(k+k_*)\,(h'-h)
 \,\d v\,\d v_*\,\d\sigma=:I.
\end{align*}
Exchanging \(v\) and \(v_*\), together with the corresponding change
\(\sigma\mapsto-\sigma\), replaces \(h'-h\) by \(h_*'-h_*\).
Averaging the two expressions for \(I\) therefore yields
\begin{align*}
 I=\frac12\int_{\R^3\times\R^3\times\mathbb S^2}
 BMM_*\,(k+k_*)\,
 \bigl(h'+h_*'-h-h_*\bigr)
 \,\d v\,\d v_*\,\d\sigma.
\end{align*}
Under the pre--postcollisional change of variables,
\(MM_*=M'M_*'\), the collision difference of \(h\) changes sign, and
\(k+k_*\) becomes \(k'+k_*'\).  Averaging once more gives
\begin{align*}
 -I
 =\frac14\int_{\R^3\times\R^3\times\mathbb S^2}
 BMM_*\,
 \bigl(h'+h_*'-h-h_*\bigr)
 \bigl(k'+k_*'-k-k_*\bigr)
 \,\d v\,\d v_*\,\d\sigma.
\end{align*}
Combining this identity with
\eqref{linearized-Boltzmann-strong-operator}, we obtain
\eqref{linearized-Boltzmann-operator} first for smooth functions and then
by density.  The bounded self-adjoint Fredholm extension and its nullspace
are the classical linearized Boltzmann result recalled in
\cite[Proposition~1.4]{GolseSaintRaymondNSLimit}.

The nullspace \eqref{linearized-kernel-assumption} also follows directly
from the Dirichlet form, as the functions
\(1,v_1,v_2,v_3,|v|^2\) are collision invariants.  Conversely, if
\(h\in\ker L\), then
\begin{align*}
 0=\langle h,Lh\rangle_M
 =\frac14\int_{\R^3\times\R^3\times\mathbb S^2}
 BMM_*\,|\mathscr D_{\rm B}h|^2
 \,\d v\,\d v_*\,\d\sigma, 
\end{align*}
and strict positivity of \(B\) gives \(\mathscr D_{\rm B}h=0\) almost
everywhere.  By the classical characterization of measurable collision
invariants \cite{CercignaniBoltzmann,CercignaniIllnerPulvirenti}, we have 
\(h(v)=a+b\cdot v+c|v|^2\).

We next compute the stress eigenvalue.  For each
\(1\leq i,j\leq3\), the detailed-balance identity \(M'M_*'=MM_*\) and
\eqref{linearized-Boltzmann-strong-operator} give
\begin{align*}
 L\mathsf A_{ij}(v)
 =\int_{\R^3\times\mathbb S^2}BM_*
 \bigl(\mathsf A_{ij}+\mathsf A_{ij,*}
 -\mathsf A_{ij}'-\mathsf A_{ij,*}'\bigr)
 \,\d v_*\,\d\sigma.
\end{align*}
By the collision rule,
\begin{align*}
 v_iv_j+v_{*,i}v_{*,j}
 &={}
 \frac12(v_i+v_{*,i})(v_j+v_{*,j})
 +\frac12(v_i-v_{*,i})(v_j-v_{*,j}),\\
 v_i'v_j'+v_{*,i}'v_{*,j}'
 &={}
 \frac12(v_i+v_{*,i})(v_j+v_{*,j})
 +\frac12|v-v_*|^2\sigma_i\sigma_j.
\end{align*}
Conservation of kinetic energy cancels the trace terms in
\(\mathsf A_{ij}\), and we have
\begin{align*}
 &\mathsf A_{ij}(v)+\mathsf A_{ij}(v_*)
 -\mathsf A_{ij}(v')-\mathsf A_{ij}(v_*')\\
 &\qquad={}
 v_iv_j+v_{*,i}v_{*,j}-v_i'v_j'-v_{*,i}'v_{*,j}'=
 \frac12\left((v_i-v_{*,i})(v_j-v_{*,j})
 -|v-v_*|^2\sigma_i\sigma_j\right).
\end{align*}
It follows that
\begin{align*}
 L\mathsf A_{ij}(v)
 =\frac12\int_{\R^3}M_*
 \int_{\mathbb S^2}b\left(
 \frac{v-v_*}{|v-v_*|}\cdot\sigma\right)
 \left((v_i-v_{*,i})(v_j-v_{*,j})
 -|v-v_*|^2\sigma_i\sigma_j\right)
 \,\d\sigma\,\d v_*.
\end{align*}
To compute the angular integral, we choose an orthogonal matrix \(R\) such
that \(R(1,0,0)=\frac{v-v_*}{|v-v_*|}\), and set
\(\sigma=R\bar\sigma\).  Rotation invariance of the surface measure
gives
\begin{align*}
 \int_{\mathbb S^2}b\left(
 \frac{v-v_*}{|v-v_*|}\cdot\sigma\right)
 \sigma\otimes\sigma\,\d\sigma=
 R\left(\int_{\mathbb S^2}b(\bar\sigma_1)
 \bar\sigma\otimes\bar\sigma\,\d\bar\sigma\right)R^{\mathsf T},
\end{align*}
where \(\bar\sigma_1\) is the first coordinate of
\(\bar\sigma=(\bar\sigma_1,\bar\sigma_2,\bar\sigma_3)\in\mathbb S^2\).  Reflection
symmetry makes all off-diagonal entries of the matrix integral vanish.
Moreover, rotational symmetry in the
\((\bar\sigma_2,\bar\sigma_3)\)-plane and
\(\bar\sigma_2^2+\bar\sigma_3^2=1-\bar\sigma_1^2\) give
\begin{align*}
 \int_{\mathbb S^2}b(\bar\sigma_1)\bar\sigma_2^2\,\d\bar\sigma
 =\int_{\mathbb S^2}b(\bar\sigma_1)\bar\sigma_3^2\,\d\bar\sigma
 =\frac12\int_{\mathbb S^2}b(\bar\sigma_1)
 (1-\bar\sigma_1^2)\,\d\bar\sigma.
\end{align*}
We thus obtain
\begin{align*}
 \int_{\mathbb S^2}b(\bar\sigma_1)
 \bar\sigma\otimes\bar\sigma\,\d\bar\sigma&=
 \frac12\int_{\mathbb S^2}b(\bar\sigma_1)
 (1-\bar\sigma_1^2)\,\d\bar\sigma\,I
 +\frac12\int_{\mathbb S^2}b(\bar\sigma_1)
 (3\bar\sigma_1^2-1)\,\d\bar\sigma\,
 (1,0,0)\otimes(1,0,0).
\end{align*}
After conjugation by \(R\), this becomes, componentwise,
\begin{align*}
 \int_{\mathbb S^2}b\left(
 \frac{v-v_*}{|v-v_*|}\cdot\sigma\right)
 \sigma_i\sigma_j\,\d\sigma=
 \frac12\int_{\mathbb S^2}b(\bar\sigma_1)
 (1-\bar\sigma_1^2)\,\d\bar\sigma\,
 \delta_{ij}
 +\frac12\int_{\mathbb S^2}b(\bar\sigma_1)
 (3\bar\sigma_1^2-1)\,\d\bar\sigma\,
 \frac{(v_i-v_{*,i})(v_j-v_{*,j})}{|v-v_*|^2}.
\end{align*}
We substitute this identity into the preceding angular integral to obtain
\begin{align*}
 \frac12\int_{\mathbb S^2}b\left(
 \frac{v-v_*}{|v-v_*|}\cdot\sigma\right)
 \left((v_i-v_{*,i})(v_j-v_{*,j})
 -|v-v_*|^2\sigma_i\sigma_j\right)\,\d\sigma={}
 \lambda_{\mathsf A}\left(
 (v_i-v_{*,i})(v_j-v_{*,j})
 -\frac{|v-v_*|^2}{3}\delta_{ij}\right),
\end{align*}
where the constant is explicitly
\begin{align*}
 \lambda_{\mathsf A}
 :=\frac34\int_{\mathbb S^2}b(\bar\sigma_1)
 \bigl(1-\bar\sigma_1^2\bigr)\,\d\bar\sigma>0.
\end{align*}
Since \(M\) has zero mean and covariance matrix \(I\),
\begin{align*}
 \int_{\R^3}M(v_*)
 \left((v_i-v_{*,i})(v_j-v_{*,j})
 -\frac{|v-v_*|^2}{3}\delta_{ij}\right)\,\d v_*=
 v_iv_j+\delta_{ij}
 -\frac{|v|^2+3}{3}\delta_{ij}
 =\mathsf A_{ij}(v).
\end{align*}
We conclude that \(L\mathsf A_{ij}=\lambda_{\mathsf A}\mathsf A_{ij}\),
which is \eqref{stress-mode-eigenvalue}.  The same calculation gives
existence and uniqueness in
\eqref{Boltzmann-stress-cell-problem}, in agreement with
\cite[Proposition~1.4 and (1.42)]{GolseSaintRaymondNSLimit}, and gives
\eqref{explicit-stress-corrector}.

By the definition of \(L\), we can see that the map \(h(v)\mapsto h(-v)\) commutes with \(L\).  Since
\(\mathsf A_{ij}\) is even, uniqueness makes
\(\widehat{\mathsf A}_{ij}\) even as well.  For the standard Maxwellian \(M\),
the Gaussian moment identities give
\begin{align*}
 \int_{\R^3}v_iv_jv_kv_lM(v)\,\d v
 &=\delta_{ij}\delta_{kl}+\delta_{ik}\delta_{jl}
 +\delta_{il}\delta_{jk},\\
 \int_{\R^3}v_iv_j|v|^2M(v)\,\d v&=5\delta_{ij},\qquad \int_{\R^3}|v|^4M(v)\,\d v=15.
\end{align*}
Using these identities, we have
\begin{align*}
 \langle\mathsf A_{ij},\mathsf A_{kl}\rangle_M
 &={}\int_{\R^3}
 \left(v_iv_j-\frac{|v|^2}{3}\delta_{ij}\right)
 \left(v_kv_l-\frac{|v|^2}{3}\delta_{kl}\right)M(v)\,\d v\\
 &={}\delta_{ij}\delta_{kl}+\delta_{ik}\delta_{jl}
 +\delta_{il}\delta_{jk}
 -\frac53\delta_{ij}\delta_{kl}
 -\frac53\delta_{ij}\delta_{kl}
 +\frac{15}{9}\delta_{ij}\delta_{kl}\\
 &={}\delta_{ik}\delta_{jl}+\delta_{il}\delta_{jk}
 -\frac23\delta_{ij}\delta_{kl}.
\end{align*}
Together with \eqref{explicit-stress-corrector}, we obtain
\eqref{stress-isotropy-identity}.  Since
\(\sum_{i,j=1}^3\mathsf C_{ij,ij}=10\), summing over \(i,j\) gives
\eqref{stress-viscosity-explicit-sum}.

To prove the remaining integrability statements, we first note from
\eqref{explicit-stress-corrector} that every component of
\(\widehat{\mathsf A}\) is a quadratic polynomial.  Gaussian decay proves
\eqref{stress-corrector-transport-moment}, while boundedness of \(B\)
proves \eqref{finite-collision-activity}.  For fixed smooth \(\varphi\),
conservation of kinetic energy gives
\begin{align*}
 \left|\mathscr D_{\rm B}\left(
 \sum_{i,j=1}^3(\nabla_{\rm sym}\varphi)_{ij}
 \widehat{\mathsf A}_{ij}\right)\right|
 \leq C_\varphi(1+|v|^2+|v_*|^2).
\end{align*}
Choosing \(\eta>0\) so that \(\eta C_\varphi<\beta/2\) proves
\eqref{stress-corrector-exponential-moment} from the Gaussian decay of
\(M_\beta(v)M_\beta(v_*)\).
\end{proof}

Fix a smooth divergence-free test field
\(\varphi:\T^3\to\R^3\) and put
\begin{align}\label{perturbed-momentum-test}
 \nabla_{\rm sym}w&:=\frac12\bigl(\nabla w+(\nabla w)^{\mathsf T}\bigr),
 \notag\\
 q_{\delta,\varphi}(x,v)
 &:=\frac1\delta v\cdot\varphi(x)
 +\sum_{i,j=1}^3
 (\nabla_{\rm sym}\varphi)_{ij}(x)\widehat{\mathsf A}_{ij}(v).
\end{align}
Because \(\nabla\cdot\varphi=0\), the matrix
\(\nabla_{\rm sym}\varphi\) is trace-free.  The
first term of \(q_{\delta,\varphi}\) tests the rescaled momentum, while the
second is the kinetic stress corrector.  We define the corrected kinetic
Hamiltonian by
\begin{align}\label{corrected-Boltzmann-Hamiltonian}
\mathbb H_{\delta,\varepsilon}^{\rm B}
 (u,\theta;\varphi)
 :=\mathcal H_{\delta,\varepsilon}^{\rm B}
 \bigl(F_\delta[u,\theta],q_{\delta,\varphi}\bigr).
\end{align}
By conservation of momentum in each collision,
\begin{align*}
 \mathscr D_{\rm B}\bigl(v\cdot\varphi(x)\bigr)(c)
 =\bigl(v'+v_*'-v-v_*\bigr)\cdot\varphi(x)=0.
\end{align*}
Substituting \eqref{local-Maxwellian-family} and
\eqref{perturbed-momentum-test} into
\eqref{rescaled-Boltzmann-Hamiltonian} therefore gives the explicit formula
\begin{align}\label{corrected-Boltzmann-Hamiltonian-explicit}
 \mathbb H_{\delta,\varepsilon}^{\rm B}
 (u,\theta;\varphi)&={}-\frac1{\delta^2}
 \int_{\T^3\times\R^3}
 \bigl(v\cdot\nabla_xF_\delta[u,\theta]\bigr)
 \bigl(v\cdot\varphi(x)\bigr)\,\d v\,\d x\notag\\
 &\quad-\frac1\delta
 \int_{\T^3\times\R^3}
 \bigl(v\cdot\nabla_xF_\delta[u,\theta]\bigr)
 \sum_{i,j=1}^3(\nabla_{\rm sym}\varphi)_{ij}(x)
 \widehat{\mathsf A}_{ij}(v)\,\d v\,\d x\notag\\
 &\quad+\frac1{2\varepsilon\delta^2}
 \int_{\mathscr C_{\rm B}}
 B(c)F_\delta[u,\theta](x,v)F_\delta[u,\theta](x,v_*)\notag\\
 &\qquad\qquad\times\left[\exp\left\{
 \varepsilon\sum_{i,j=1}^3
 (\nabla_{\rm sym}\varphi)_{ij}(x)
 \bigl(
 \widehat{\mathsf A}_{ij}(v')
 +\widehat{\mathsf A}_{ij}(v_*')
 -\widehat{\mathsf A}_{ij}(v)
 -\widehat{\mathsf A}_{ij}(v_*)
 \bigr)\right\}-1\right]\,\d c.
\end{align}

In the following, we derive the Hamiltonian of LLNS.  Testing
LLNS equations against the smooth divergence-free field \(\varphi\), using
\(\Pi\varphi=\varphi\), and integrating by parts gives the formal scalar
semimartingale
\begin{align*}
 \d\langle u_t,\varphi\rangle
 ={}&\int_{\T^3}\left[
 \sum_{i,j=1}^3(u_t)_i(u_t)_j\partial_j\varphi_i
 -2\mu\sum_{i,j=1}^3
 (\nabla_{\rm sym}u_t)_{ij}
 (\nabla_{\rm sym}\varphi)_{ij}\right]\d x\,\d t\\
 &+\mathfrak a\sqrt{2\mu}
 \sum_{i,j=1}^3\int_{\T^3}
 (\nabla_{\rm sym}\varphi)_{ij}\,\d W^{\rm F}_{ij}(t,x).
\end{align*}
The covariance \eqref{LLNS-stress-covariance} shows that the
quadratic-variation density of the martingale part is
\begin{align*}
 2\mathfrak a^2\mu
 \int_{\T^3}\sum_{i,j,k,l=1}^3
 (\nabla_{\rm sym}\varphi)_{ij}\mathsf C_{ij,kl}
 (\nabla_{\rm sym}\varphi)_{kl}\,\d x.
\end{align*}
For the exponential test observable
\(u\mapsto\exp\{\langle u,\varphi\rangle\}\), its first two derivatives are
\begin{align*}
 D\bigl(e^{\langle\,\cdot\,,\varphi\rangle}\bigr)(u)[h]
 =e^{\langle u,\varphi\rangle}\langle h,\varphi\rangle,\qquad D^2\bigl(e^{\langle\,\cdot\,,\varphi\rangle}\bigr)(u)[h,k]
 =e^{\langle u,\varphi\rangle}
 \langle h,\varphi\rangle\langle k,\varphi\rangle.
\end{align*}
The scalar It\^o formula therefore gives
\begin{align*}
 \d e^{\langle u_t,\varphi\rangle}
 ={}&e^{\langle u_t,\varphi\rangle}
 \d\langle u_t,\varphi\rangle
 +\frac12e^{\langle u_t,\varphi\rangle}
 \d[\langle u,\varphi\rangle]_t\\
 ={}&e^{\langle u_t,\varphi\rangle}
 \Bigg\{\int_{\T^3}\left[
 \sum_{i,j=1}^3(u_t)_i(u_t)_j\partial_j\varphi_i
 -2\mu\sum_{i,j=1}^3
 (\nabla_{\rm sym}u_t)_{ij}
 (\nabla_{\rm sym}\varphi)_{ij}\right]\d x\\
 &\hspace{25mm}
 +\mathfrak a^2\mu\int_{\T^3}
 \sum_{i,j,k,l=1}^3
 (\nabla_{\rm sym}\varphi)_{ij}\mathsf C_{ij,kl}
 (\nabla_{\rm sym}\varphi)_{kl}\,\d x\Bigg\}\d t\\
 &+e^{\langle u_t,\varphi\rangle}
 \mathfrak a\sqrt{2\mu}\sum_{i,j=1}^3\int_{\T^3}
 (\nabla_{\rm sym}\varphi)_{ij}\,\d W^{\rm F}_{ij}(t,x).
\end{align*}
It follows that the formal generator
\(\mathcal L_{\mathfrak a}^{\rm LLNS}\) of the LLNS dynamics acts on this
exponential observable as
\begin{align}\label{LLNS-generator-on-exponential}
 \left(\mathcal L_{\mathfrak a}^{\rm LLNS}
 e^{\langle\,\cdot\,,\varphi\rangle}\right)(u)&={}
 e^{\langle u,\varphi\rangle}\Bigg\{
 \int_{\T^3}\left[
 \sum_{i,j=1}^3u_i u_j\partial_j\varphi_i
 -2\mu\sum_{i,j=1}^3(\nabla_{\rm sym}u)_{ij}
 (\nabla_{\rm sym}\varphi)_{ij}\right]\d x\notag\\
 &\hspace{39mm}
 +\mathfrak a^2\mu\int_{\T^3}
 \sum_{i,j,k,l=1}^3
 (\nabla_{\rm sym}\varphi)_{ij}\mathsf C_{ij,kl}
 (\nabla_{\rm sym}\varphi)_{kl}\,\d x
 \Bigg\}.
\end{align}
The last integral in \eqref{LLNS-generator-on-exponential} is the It\^o
correction produced by the second derivative of the exponential
observable.  Using the exponential-generator definition
\eqref{exponential-generator-general-definition}, we obtain
\begin{align}\label{LLNS-Hamiltonian}
 \mathbb H_{\mathfrak a}^{\rm LLNS}(u;\varphi)
 :={}&e^{-\langle u,\varphi\rangle}
 \left(\mathcal L_{\mathfrak a}^{\rm LLNS}
 e^{\langle\,\cdot\,,\varphi\rangle}\right)(u)\notag\\
 ={}&\int_{\T^3}
 \left[\sum_{i,j=1}^3u_i u_j\partial_j\varphi_i
 -2\mu\sum_{i,j=1}^3(\nabla_{\rm sym}u)_{ij}
 (\nabla_{\rm sym}\varphi)_{ij}\right]\d x\notag\\
 &+\mathfrak a^2\mu\int_{\T^3}
 \sum_{i,j,k,l=1}^3
 (\nabla_{\rm sym}\varphi)_{ij}\mathsf C_{ij,kl}
 (\nabla_{\rm sym}\varphi)_{kl}\,\d x.
\end{align}

\subsection{The convergence of Hamiltonians}
We now prove the convergence of the corrected Boltzmann Hamiltonian to the
LLNS Hamiltonian.
\begin{theorem}
\label{thm:Boltzmann-LLNS-Hamiltonian}
Let \(B\) be a measurable collision kernel satisfying
Assumption~\ref{ass:Boltzmann-kernel-conditions}, and let
\(\widehat{\mathsf A}\) and \(\mu\) be given by
Lemma~\ref{lem:linearized-Boltzmann-kernel}.
Let \(u,\varphi\in C^\infty(\T^3;\R^3)\) and
\(\theta\in C^\infty(\T^3;\R)\), with
\(\nabla\cdot u=\nabla\cdot\varphi=0\).
If, for some \(\mathfrak a\geq0\),
\begin{align}\label{Hamiltonian-distinguished-scaling}
 \delta\downarrow0,\qquad \varepsilon\downarrow0,\qquad
 \frac{\varepsilon}{\delta^2}\longrightarrow\mathfrak a^2,
\end{align}
then
\begin{align}\label{Boltzmann-to-LLNS-Hamiltonian-limit}
 \mathbb H_{\delta,\varepsilon}^{\rm B}(u,\theta;\varphi)
 \longrightarrow
 \mathbb H_{\mathfrak a}^{\rm LLNS}(u;\varphi).
\end{align}
The two parts of the kinetic Hamiltonian satisfy
\begin{align}\label{Hamiltonian-transport-limit}
 -\frac1\delta
 \langle v\cdot\nabla_xF_\delta,q_{\delta,\varphi}\rangle
 \longrightarrow
 \int_{\T^3}\left[
 \sum_{i,j=1}^3u_i u_j\partial_j\varphi_i
 -2\mu\sum_{i,j=1}^3
 (\nabla_{\rm sym}u)_{ij}
 (\nabla_{\rm sym}\varphi)_{ij}\right]\d x
\end{align}
and
\begin{align}\label{Hamiltonian-collision-limit}
 \frac1{2\varepsilon\delta^2}
 \int_{\mathscr C_{\rm B}}BF_\delta F_{\delta,*}
 \bigl(e^{\varepsilon\mathscr D_{\rm B}
 q_{\delta,\varphi}}-1\bigr)\,\d c\longrightarrow
 \mathfrak a^2\mu\int_{\T^3}
 \sum_{i,j,k,l=1}^3
 (\nabla_{\rm sym}\varphi)_{ij}\mathsf C_{ij,kl}
 (\nabla_{\rm sym}\varphi)_{kl}\,\d x.
\end{align}
\end{theorem}

\begin{proof}
Fix any \(\beta\in(0,1)\), and choose \(\eta>0\) as in
Lemma~\ref{lem:linearized-Boltzmann-kernel} for the given test field
\(\varphi\).  We identify the transport and collision parts separately.
For the leading momentum test, integration by parts in \(x\) gives
\begin{align}\label{momentum-transport-computation}
 -\frac1{\delta^2}\int_{\T^3\times\R^3}
 (v\cdot\nabla_xF_\delta)(v\cdot\varphi)\,\d v\,\d x
 =\frac1{\delta^2}\int_{\T^3}
 \sum_{i,j=1}^3
 \left(\int_{\R^3}F_\delta v_i v_j\,\d v\right)
 \partial_j\varphi_i\,\d x.
\end{align}
For the Maxwellian
\begin{align*}
 \mathcal M_{\varrho,U,T}(v)
 :=\frac{\varrho}{(2\pi T)^{3/2}}
 \exp\left(-\frac{|v-U|^2}{2T}\right),
\end{align*}
by the change of variables \(w=v-U\), so that \(v=w+U\) and
\(\d v=\d w\), and we have 
\begin{align*}
 \int_{\R^3}\mathcal M_{\varrho,U,T}(v)v\otimes v\,\d v&={}
 \frac{\varrho}{(2\pi T)^{3/2}}
 \int_{\R^3}e^{-|w|^2/(2T)}
 (w+U)\otimes(w+U)\,\d w\\
 &={}
 \frac{\varrho}{(2\pi T)^{3/2}}
 \int_{\R^3}e^{-|w|^2/(2T)}w\otimes w\,\d w
 +\varrho U\otimes U\\
 &={}\varrho T I+\varrho U\otimes U.
\end{align*}
Here the terms containing only one factor of \(w\) vanish by oddness, and
the centered Gaussian identities
\begin{align*}
 \frac{1}{(2\pi T)^{3/2}}
 \int_{\R^3}e^{-|w|^2/(2T)}\,\d w=1,
 \qquad
 \frac{1}{(2\pi T)^{3/2}}
 \int_{\R^3}e^{-|w|^2/(2T)}w_iw_j\,\d w
 =T\delta_{ij}
\end{align*}
were used in the last equality.  For \eqref{local-Maxwellian-family},
\begin{align*}
 \varrho=1-\delta\theta,
 \qquad U=\delta u,
 \qquad T=1+\delta\theta.
\end{align*}
We therefore obtain
\begin{align*}
 \int_{\R^3}F_\delta v\otimes v\,\d v
 &=(1-\delta\theta)(1+\delta\theta)I
 +(1-\delta\theta)(\delta u)\otimes(\delta u)\\
 &=(1-\delta^2\theta^2)I
 +\delta^2(1-\delta\theta)u\otimes u.
\end{align*}
Since \(\varphi\) is divergence-free, the isotropic part of the second
moment gives
\begin{align*}
 \frac1{\delta^2}\int_{\T^3}
 (1-\delta^2\theta^2)
 \sum_{i=1}^3\partial_i\varphi_i\,\d x
 =\frac1{\delta^2}\int_{\T^3}
 (1-\delta^2\theta^2)\nabla_x\cdot\varphi\,\d x=0.
\end{align*}
The remaining part of the first transport term is
\begin{align*}
 \int_{\T^3}(1-\delta\theta)
 \sum_{i,j=1}^3u_i u_j\partial_j\varphi_i\,\d x,
\end{align*}
and it converges to
\(\int_{\T^3}\sum_{i,j=1}^3
u_i u_j\partial_j\varphi_i\,\d x\).

For the corrector part, direct differentiation of the explicit Maxwellian
\eqref{local-Maxwellian-family} gives, pointwise in \((x,v)\),
\begin{align*}
 \delta^{-1}v\cdot\nabla_xF_\delta
 \longrightarrow Mv\cdot\nabla_xg_0.
\end{align*}
More explicitly, logarithmic differentiation gives, for \(k=1,2,3\),
\begin{align*}
 \delta^{-1}\partial_{x_k}\log F_\delta
 =-\frac{\partial_{x_k}\theta}{1-\delta\theta}
 -\frac{3\partial_{x_k}\theta}{2(1+\delta\theta)}
 +\frac{(v-\delta u)\cdot\partial_{x_k}u}{1+\delta\theta}+\frac{|v-\delta u|^2\partial_{x_k}\theta}
 {2(1+\delta\theta)^2}.
\end{align*}
Choose \(\delta_0\in(0,1)\) sufficiently small that \(1\pm\delta\theta\) stay
uniformly away from zero for \(0<\delta<\delta_0\).  Multiplying the last
identity by \(F_\delta v_k\), summing over \(k\), and using the boundedness
of \(u,\theta\) and their first derivatives yield
\begin{align*}
 \left|\delta^{-1}v\cdot\nabla_xF_\delta(x,v)\right|
 \leq C F_\delta(x,v)|v|(1+|v-\delta u(x)|^2).
\end{align*}
Every Maxwellian in this family with \(0<\delta<\delta_0\) is bounded by a
constant multiple of the broader Maxwellian \(M_\beta\).  Together with
the boundedness of \(u\), the preceding estimate gives
\begin{align}\label{local-Maxwellian-transport-domination}
 \left|\delta^{-1}v\cdot\nabla_xF_\delta(x,v)\right|
 \leq C M_\beta(v)(1+|v|^3),
 \qquad 0<\delta<\delta_0.
\end{align}
We can now use \eqref{stress-corrector-transport-moment} and dominated
convergence to obtain
\begin{align*}
 -\frac1\delta\left\langle v\cdot\nabla_xF_\delta,
 \sum_{i,j=1}^3(\nabla_{\rm sym}\varphi)_{ij}
 \widehat{\mathsf A}_{ij}\right\rangle
 \longrightarrow
 -\int_{\T^3}
 \left\langle
 \sum_{i,j=1}^3(\nabla_{\rm sym}\varphi)_{ij}
 \widehat{\mathsf A}_{ij},v\cdot\nabla_xg_0
 \right\rangle_M\,\d x.
\end{align*}
Recall from \eqref{Boltzmann-stress-mode} that
\(\mathsf A_{ij}(v)=v_iv_j-|v|^2\delta_{ij}/3\) and
\(\mathsf K(v)=v(|v|^2-5)/2\).  Direct differentiation of \(g_0\) gives
\begin{align*}
 v\cdot\nabla_xg_0
 ={}&\sum_{i,j=1}^3v_iv_j\partial_i u_j
 +\mathsf K(v)\cdot\nabla_x\theta.
\end{align*}
For the first term, the definition of \(\mathsf A\) and
\(\nabla_x\cdot u=0\) yield
\begin{align*}
 \sum_{i,j=1}^3v_iv_j\partial_i u_j=\sum_{i,j=1}^3\mathsf A_{ij}(v)\partial_i u_j
 +\frac{|v|^2}{3}\sum_{i=1}^3\partial_i u_i=\frac12\sum_{i,j=1}^3\mathsf A_{ij}(v)
 \bigl(\partial_i u_j+\partial_j u_i\bigr)=\sum_{i,j=1}^3
 \mathsf A_{ij}(v)(\nabla_{\rm sym}u)_{ij}.
\end{align*}
Thus
\begin{align*}
 v\cdot\nabla_xg_0
 =\sum_{i,j=1}^3
 \mathsf A_{ij}(v)(\nabla_{\rm sym}u)_{ij}
 +\mathsf K(v)\cdot\nabla_x\theta.
\end{align*}
Since \(\widehat{\mathsf A}\) is even in \(v\) and \(\mathsf K\) is odd,
\begin{align*}
 \left\langle
 \sum_{i,j=1}^3(\nabla_{\rm sym}\varphi)_{ij}
 \widehat{\mathsf A}_{ij},
 \mathsf K\cdot\nabla_x\theta
 \right\rangle_M=0.
\end{align*}
For the stress part, \eqref{stress-isotropy-identity} gives
\begin{align*}
 \left\langle
 \sum_{i,j=1}^3(\nabla_{\rm sym}\varphi)_{ij}
 \widehat{\mathsf A}_{ij},
 \sum_{k,l=1}^3\mathsf A_{kl}
 (\nabla_{\rm sym}u)_{kl}
 \right\rangle_M
 &=
 \mu\sum_{i,j,k,l=1}^3
 (\nabla_{\rm sym}\varphi)_{ij}\mathsf C_{ij,kl}
 (\nabla_{\rm sym}u)_{kl}\\
 &=\mu\sum_{i,j,k,l=1}^3
 (\nabla_{\rm sym}\varphi)_{ij}
 \left(\delta_{ik}\delta_{jl}+\delta_{il}\delta_{jk}
 -\frac23\delta_{ij}\delta_{kl}\right)
 (\nabla_{\rm sym}u)_{kl}\\
 &=\mu\sum_{i,j=1}^3
 (\nabla_{\rm sym}\varphi)_{ij}(\nabla_{\rm sym}u)_{ij}
 +\mu\sum_{i,j=1}^3
 (\nabla_{\rm sym}\varphi)_{ij}(\nabla_{\rm sym}u)_{ji}\\
 &\quad-\frac{2\mu}{3}
 \left(\sum_{i=1}^3(\nabla_{\rm sym}\varphi)_{ii}\right)
 \left(\sum_{k=1}^3(\nabla_{\rm sym}u)_{kk}\right)\\
 &=2\mu\sum_{i,j=1}^3
 (\nabla_{\rm sym}\varphi)_{ij}
 (\nabla_{\rm sym}u)_{ij}.
\end{align*}
We have proved that
\begin{align*}
 \int_{\T^3}\left\langle
 \sum_{i,j=1}^3(\nabla_{\rm sym}\varphi)_{ij}
 \widehat{\mathsf A}_{ij},v\cdot\nabla_xg_0
 \right\rangle_M\,\d x
 =2\mu\int_{\T^3}\sum_{i,j=1}^3
 (\nabla_{\rm sym}\varphi)_{ij}
 (\nabla_{\rm sym}u)_{ij}\,\d x.
\end{align*}
This completes the proof of \eqref{Hamiltonian-transport-limit}.

In the following, we turn to the collision part.  The velocity function
\(v\cdot\varphi(x)\) is a collision invariant, and hence
\begin{align*}
 \mathscr D_{\rm B}q_{\delta,\varphi}
 =\mathscr D_{\rm B}\left(
 \sum_{i,j=1}^3(\nabla_{\rm sym}\varphi)_{ij}
 \widehat{\mathsf A}_{ij}\right).
\end{align*}
Moreover, \(F_\delta(x,\cdot)\) is a Maxwellian at every \(x\), so
the conservation laws
\begin{align*}
 v+v_*=v'+v_*',
 \qquad
 |v|^2+|v_*|^2=|v'|^2+|v_*'|^2
\end{align*}
imply
\begin{align*}
 |v-\delta u(x)|^2+|v_*-\delta u(x)|^2=|v'-\delta u(x)|^2
 +|v_*'-\delta u(x)|^2.
\end{align*}
The explicit formula \eqref{local-Maxwellian-family} then yields the
detailed-balance identity
\begin{align}\label{local-Maxwellian-detailed-balance}
 F_\delta(x,v)F_\delta(x,v_*)
 =F_\delta(x,v')F_\delta(x,v_*').
\end{align}
The pre--post-collisional change of variables preserves the collision
measure.  Under this change, the collision kernel becomes
\begin{align*}
 B\left(v'-v_*',\frac{v-v_*}{|v-v_*|}\right)
 =b\left(
 \frac{v'-v_*'}{|v'-v_*'|}\cdot
 \frac{v-v_*}{|v-v_*|}\right)=b\left(\sigma\cdot\frac{v-v_*}{|v-v_*|}\right)
 =B(v-v_*,\sigma),
\end{align*}
where we used \(v'-v_*'=|v-v_*|\sigma\) and
\eqref{linearized-Boltzmann-kernel-conditions}.  Under this change of
variables, \eqref{local-Maxwellian-detailed-balance} gives
\begin{align*}
 &\int_{\R^3\times\R^3\times\mathbb S^2}
 B(v-v_*,\sigma)F_\delta F_{\delta,*}
 \sum_{i,j=1}^3(\nabla_{\rm sym}\varphi)_{ij}(x)
 \left(\widehat{\mathsf A}_{ij}(v')
 +\widehat{\mathsf A}_{ij}(v_*')\right)
 \,\d v\,\d v_*\,\d\sigma\\
 &\qquad={}
 \int_{\R^3\times\R^3\times\mathbb S^2}
 B(v-v_*,\sigma)F_\delta'F_{\delta,*}'
 \sum_{i,j=1}^3(\nabla_{\rm sym}\varphi)_{ij}(x)
 \left(\widehat{\mathsf A}_{ij}(v)
 +\widehat{\mathsf A}_{ij}(v_*)\right)
 \,\d v\,\d v_*\,\d\sigma\\
 &\qquad={}
 \int_{\R^3\times\R^3\times\mathbb S^2}
 B(v-v_*,\sigma)F_\delta F_{\delta,*}
 \sum_{i,j=1}^3(\nabla_{\rm sym}\varphi)_{ij}(x)
 \left(\widehat{\mathsf A}_{ij}(v)
 +\widehat{\mathsf A}_{ij}(v_*)\right)
 \,\d v\,\d v_*\,\d\sigma.
\end{align*}
After integration in \(x\), the gain and loss terms cancel, and hence
\begin{align*}
 \int_{\mathscr C_{\rm B}}BF_\delta F_{\delta,*}
 \mathscr D_{\rm B}\left(
 \sum_{i,j=1}^3(\nabla_{\rm sym}\varphi)_{ij}
 \widehat{\mathsf A}_{ij}\right)\,\d c=0.
\end{align*}
We may therefore center the collision term without changing its value:
\begin{align}\label{centered-collision-Hamiltonian}
 &\frac1{2\varepsilon\delta^2}
 \int_{\mathscr C_{\rm B}} BF_\delta F_{\delta,*}
 \left[\exp\left\{\varepsilon\mathscr D_{\rm B}\left(
 \sum_{i,j=1}^3(\nabla_{\rm sym}\varphi)_{ij}
 \widehat{\mathsf A}_{ij}\right)\right\}-1\right]\d c\notag\\
 &={}
 \frac1{2\varepsilon\delta^2}
 \int_{\mathscr C_{\rm B}}BF_\delta F_{\delta,*}
 \left[\exp\left\{\varepsilon\mathscr D_{\rm B}\left(
 \sum_{i,j=1}^3(\nabla_{\rm sym}\varphi)_{ij}
 \widehat{\mathsf A}_{ij}\right)\right\}
 -1-\varepsilon\mathscr D_{\rm B}\left(
 \sum_{i,j=1}^3(\nabla_{\rm sym}\varphi)_{ij}
 \widehat{\mathsf A}_{ij}\right)\right]\d c.
\end{align}
For every real \(X\),
\begin{align}\label{exponential-Taylor-domination}
 0\leq\frac{e^{\varepsilon X}-1-\varepsilon X}{\varepsilon^2}
 \leq\frac12|X|^2e^{\varepsilon|X|}.
\end{align}
Moreover, boundedness of \(u\) and \(\theta\) and the explicit formula
\eqref{local-Maxwellian-family} imply, after decreasing \(\delta_0\) if
necessary,
\begin{align}\label{local-Maxwellian-collision-domination}
 F_\delta(x,v)F_\delta(x,v_*)
 \leq C M_\beta(v)M_\beta(v_*),
 \qquad 0<\delta<\delta_0.
\end{align}
Along a sequence satisfying \eqref{Hamiltonian-distinguished-scaling}, the
ratio \(\varepsilon/\delta^2\) is bounded and eventually
\(\varepsilon\leq\eta\).  Hence
\eqref{stress-corrector-exponential-moment} and
\eqref{exponential-Taylor-domination}--
\eqref{local-Maxwellian-collision-domination} provide an integrable
dominating function.  Since \(F_\delta\to M\) pointwise, dominated
convergence gives
\begin{align}\label{quadratic-collision-Hamiltonian-limit}
 \eqref{centered-collision-Hamiltonian}
 \longrightarrow
 \frac{\mathfrak a^2}{4}
 \int_{\mathscr C_{\rm B}}BMM_*
 \left|\mathscr D_{\rm B}\left(
 \sum_{i,j=1}^3(\nabla_{\rm sym}\varphi)_{ij}
 \widehat{\mathsf A}_{ij}\right)\right|^2\,\d c.
\end{align}
By the Dirichlet identity \eqref{linearized-Boltzmann-operator},
\begin{align}\label{Boltzmann-fluctuation-dissipation-identity}
 \frac12\int_{\mathscr C_{\rm B}}BMM_*
 \mathscr D_{\rm B}p\,\mathscr D_{\rm B}q\,\d c
 =2((p,Lq))_M.
\end{align}
The cell problem gives
\begin{align*}
 L\left(\sum_{i,j=1}^3(\nabla_{\rm sym}\varphi)_{ij}
 \widehat{\mathsf A}_{ij}\right)
 =\sum_{i,j=1}^3(\nabla_{\rm sym}\varphi)_{ij}\mathsf A_{ij},
\end{align*}
We use \eqref{Boltzmann-fluctuation-dissipation-identity} and
\eqref{stress-isotropy-identity} to rewrite the right-hand side of
\eqref{quadratic-collision-Hamiltonian-limit} as
\begin{align*}
 \mathfrak a^2\left(\left(
 \sum_{i,j=1}^3(\nabla_{\rm sym}\varphi)_{ij}
 \widehat{\mathsf A}_{ij},
 L\left(\sum_{k,l=1}^3(\nabla_{\rm sym}\varphi)_{kl}
 \widehat{\mathsf A}_{kl}\right)
 \right)\right)_M
 =\mathfrak a^2\mu\int_{\T^3}
 \sum_{i,j,k,l=1}^3
 (\nabla_{\rm sym}\varphi)_{ij}\mathsf C_{ij,kl}
 (\nabla_{\rm sym}\varphi)_{kl}\,\d x.
\end{align*}
This proves \eqref{Hamiltonian-collision-limit}.  Adding the transport and
collision limits gives \eqref{Boltzmann-to-LLNS-Hamiltonian-limit}.
\end{proof}

\begin{remark}[Why the limit contains no temperature field]
\label{rem:LLNS-without-temperature}
Theorem~\ref{thm:Boltzmann-LLNS-Hamiltonian} identifies the Hamiltonian
governing the divergence-free velocity field.  The local Maxwellians
\eqref{local-Maxwellian-family} have temperature
\(T_\delta=1+\delta\theta\) and density
\(\varrho_\delta=1-\delta\theta\), and hence approach the global
equilibrium \(M=\mathcal M_{1,0,1}\).  The random-stress covariance in
fluctuating hydrodynamics generally depends on the local absolute
temperature.  Under the present scaling, this dependence is evaluated at
the reference temperature \(T_0=1\).  The quadratic part of the
collision Hamiltonian is proportional to
\begin{align*}
 \frac{\varepsilon}{\delta^2}
 \int_{\mathscr C_{\rm B}}BF_\delta F_{\delta,*}
 \left|\mathscr D_{\rm B}\left(
 \sum_{i,j=1}^3(\nabla_{\rm sym}\varphi)_{ij}
 \widehat{\mathsf A}_{ij}\right)\right|^2\,\d c.
\end{align*}
By \eqref{hydrodynamic-profile},
\begin{align*}
 F_\delta(x,v)
 &=M(v)\bigl(1+\delta g_0(x,v)+O(\delta^2)\bigr),\qquad
 F_\delta(x,v_*)
 =M(v_*)\bigl(1+\delta g_0(x,v_*)+O(\delta^2)\bigr).
\end{align*}
Multiplying the two expansions gives
\begin{align*}
 F_\delta F_{\delta,*}
 =MM_*\bigl(1+\delta g_0+\delta g_{0,*}+O(\delta^2)\bigr)
 =MM_*+O(\delta),
\end{align*}
locally uniformly in \((x,v,v_*)\).  Since
\(\varepsilon/\delta^2\to\mathfrak a^2\), the contribution of
\(\theta\) vanishes in the limiting covariance.  The covariance is thus
evaluated at the equilibrium Maxwellian \(M\), and the limit contains the
additive LLNS random stress with equilibrium viscosity \(\mu\).  Moreover,
the theorem tests only the divergence-free velocity field, so the pressure
and temperature modes do not appear.

To obtain a fluctuating Navier--Stokes--Fourier limit, one would also have
to test the energy variable.  This would produce the Fourier temperature
equation together with a random heat flux.  Such an extension is beyond the
momentum Hamiltonian studied here.
\end{remark}

\begin{remark}[Why the Poisson noise becomes Gaussian]
\label{rem:cumulant-scaling}
The \(k\)-th collision cumulant in
\eqref{rescaled-Boltzmann-Hamiltonian} carries the factor
\begin{align}\label{hydrodynamic-cumulant-scaling}
 \frac{\varepsilon^k}{\varepsilon\delta^2}
 =\frac{\varepsilon^{k-1}}{\delta^2}.
\end{align}
Under \eqref{Hamiltonian-distinguished-scaling}, the second cumulant
converges to a nonzero quadratic form, while every cumulant of order
\(k\geq3\) vanishes.  This is why the limiting Hamiltonian produces the Gaussian limit.
\end{remark}

\section{Scope of the program}

As mentioned in the introduction, the FKT program introduced above consists of four steps.  In this paper, we
carry out Step~1 by proposing the master fluctuating kinetic equation and
developing its mathematical setting.  Steps~2 and~3 concern the solution
theory and the asymptotic behavior of the stochastic equation associated
with each choice of \(\Psi_f^*\), \(\mathscr D\), and \(\mathcal E\).
Recent work on the fluctuating VFP
\cite{MullerVonRenesseZimmerVFP,HaoWuZimmerVFPDK}, Landau
\cite{DuongHeWuLandau}, and Boltzmann equations
\cite{WuPoissonBoltzmann} gives first results in this direction.  The
solution theories obtained in Steps~2 and~3 are also needed for the
process-level limits in Step~4.  We summarize the available results and the
open problems below.

\subsection{The fluctuating Vlasov--Fokker--Planck equation}
The fluctuating Vlasov--Fokker--Planck equation takes the form
\begin{align*}
 \partial_tf={}&(\Delta_v-v\cdot\nabla_x)f+\nabla_v\cdot(vf)
 -\nabla_v\cdot(fV\star_x\langle f\rangle)
 -\varepsilon^{1/2}\nabla_v\cdot(\sqrt{f}\xi),
\end{align*}
where \(\langle f\rangle\) denotes the spatial marginal of \(f\), and
\(\xi\) is space--time--velocity white noise.  M\"uller, von Renesse, and
Zimmer \cite{MullerVonRenesseZimmerVFP} formulated this equation by a
martingale problem.  The empirical measure of the underlying mean-field
particle system solves this martingale problem, and its law is unique.  This
gives the first mathematical justification of the fluctuating VFP model.

For function-valued solutions, the noise is regularized and the equation
takes the form
\begin{align*}
 \partial_tf={}&(\Delta_v-v\cdot\nabla_x)f+\nabla_v\cdot(vf)
 -\nabla_v\cdot(fV\star_x\langle f\rangle)
 -\varepsilon^{1/2}\nabla_v\cdot(\sigma(f)\circ\xi_F),
\end{align*}
where \(\xi_F\) is white in time and correlated in space and velocity,
and \(\sigma\) is a regularized approximation of the square root.  Hao,
Wu, and Zimmer \cite{HaoWuZimmerVFPDK} proved existence and uniqueness for
sufficiently regular \(\sigma\), under suitable integrability assumptions
on the interaction kernel \(V\).  For the square-root coefficient and
\(V=0\), probabilistically weak renormalized kinetic solutions were
constructed.  Their uniqueness remains open.

Unlike the classical Dean--Kawasaki equation \cite{FehrmanGess2024}, the
diffusion acts only in velocity.  The compactness argument requires to use
the kinetic semigroup to transfer regularity from velocity to position.  On
the full phase space, the tails and the entropy also require moment bounds.
These bounds are available when \(\sigma\) is regular.  For the square-root
coefficient, the difficulty already becomes apparent in the
Stratonovich-to-It\^o correction: its logarithmic singularity breaks the
moment argument and obstructs the nonlocal interaction limit.  The
square-root result in
\cite{HaoWuZimmerVFPDK} consequently assumes \(V=0\) and replaces entropy
estimates by \(L^2\) bounds.  Without \(L^1\)-conservation, the kinetic
measure near zero is difficult to control, which is also a major obstacle
to uniqueness.  A related whole-space difficulty for the classical
Dean--Kawasaki equation is treated in \cite{FehrmanGessWholeSpace} through
relative entropy with respect to a positive constant state.

Thus, uniqueness for integrable square-root solutions, an entropy theory on
the full phase space, and nonlocal or singular interactions remain open.

Regarding to Step~3, the solution theory in
\cite{HaoWuZimmerVFPDK} suggests to study large deviations and central-limit
fluctuations in the joint regime
\begin{align*}
 \bigl(\varepsilon,\delta(\varepsilon),n(\varepsilon)\bigr)
 \longrightarrow(0,0,\infty),
 \qquad
 \xi_F=\xi_{\delta(\varepsilon)},
 \qquad
 \sigma_{n(\varepsilon)}\longrightarrow\sqrt{\,\cdot\,}.
\end{align*}
Here \(\xi_\delta\) is white in time and mollified in both position and
velocity, with correlation length \(\delta\), while \(\sigma_n\) is a
regular approximation of the square root.  Motivated from this solution
theory, one may combine the joint small-noise and singular-limit methods of
\cite{FehrmanGessNonEquilibrium,DirrFehrmanGess2024,GessWuZhang2026}
with the kinetic semigroup estimates developed in
\cite{HaoWuZimmerVFPDK}.  The present theory still requires the
regularization \(\sigma_n\).  A uniqueness theory for the square-root
equation would allow us to remove this auxiliary scale.

\subsection{The fluctuating Landau equation}

Duong, He, and Wu \cite{DuongHeWuLandau} introduced the homogeneous
fluctuating Landau equation in the Stratonovich form.  We write
\(\widetilde\nabla:=\mathscr D_{\mathrm L}\) for the pairwise Landau
gradient in \eqref{Landau-channel-derivative} and
\(\widetilde\nabla\cdot\) for its divergence.  The regularized equation is
\begin{align*}
 \partial_tf={}&\frac12\widetilde\nabla\cdot
 \bigl(Aff_*\widetilde\nabla\log f\bigr)
 -\frac{\sqrt\varepsilon}{2}\widetilde\nabla\cdot
 \left(A^{1/2}\sigma(f)\sigma(f_*)\circ\xi_K\right).
\end{align*}
Here \(f=f(t,v)\), \(f_*=f(t,v_*)\), the symbol \(\circ\) denotes
Stratonovich integration, and \(\sigma\) is a smooth approximation of the
square root.  The noise is finite dimensional and correlated in velocity:
\begin{align*}
 \xi_K(t,v,v_*)
 :=\sum_{k=1}^K\frac{\d}{\d t}B^k(t)g_k(v,v_*),
 \qquad
 g_k(v,v_*)=-g_k(v_*,v),
\end{align*}
where the \(B^k\) are independent Brownian motions.  The antisymmetry of
\(g_k\) preserves the pairwise conservative structure of Landau collisions.
For
moderately soft kernels \(A(v-v_*)=|v-v_*|^{2+\gamma}\),
\(\gamma\in(-2,0)\), probabilistically weak solutions are constructed in
\cite{DuongHeWuLandau}, together with mass conservation, an energy
inequality, and entropy dissipation.

For the square-root coefficient, the obstruction comes from the nonlocal
interaction in the noise.  The Stratonovich-to-It\^o correction contains
velocity integrals with vacuum singularities in both the visible and
integrated variables.  The usual renormalized kinetic formulation controls
only the former.  Consequently, no closed solution concept is currently
available for the square-root equation.  Uniqueness, other collision
kernels, the inhomogeneous equation, and the analysis of fluctuations and
large deviations remain open.

\subsection{The fluctuating Boltzmann equation}

The fluctuating Boltzmann equation introduced in
\cite{WuPoissonBoltzmann} is driven by Poisson noise and takes the form
\begin{align*}
 \d F_t^\varepsilon+v\cdot\nabla_xF_t^\varepsilon\,\d t
 ={}&Q_{\mathrm B}(F_t^\varepsilon,F_t^\varepsilon)\,\d t\\
 &+\varepsilon\int_{\mathscr C_{\mathrm B}\times[0,\infty)}
 \mathcal J_c
 \boldsymbol1_{\{u\leq(2\varepsilon)^{-1}
 B(v-v_*,\sigma)F_{t-}^\varepsilon(x,v)
 F_{t-}^\varepsilon(x,v_*)\}}
 \widetilde{\mathcal N}(\d t,\d c,\d u).
\end{align*}
The Poisson measure records individual collisions with rate proportional to
\(B(v-v_*,\sigma)F^\varepsilon F_*^\varepsilon\), and \(\mathcal J_c\)
transfers mass from the incoming to the outgoing velocities.  Every
increment conserves mass, momentum, and energy, but may take a continuum
density outside the nonnegative cone.  As observed in
\cite{WuPoissonBoltzmann}, the equation therefore has no direct nonnegative solution theory.

To avoid this problem, \cite{WuPoissonBoltzmann} introduces a coarse-grained
model whose cellwise jumps preserve nonnegativity.  For this model,
well-posedness, a large deviation principle, and a restricted
\(\Gamma\)-convergence result on bounded velocity regions are proved.  This
is a contribution to Steps~2 and~3.  It remains open to construct a solution
theory for the continuum SPDE which preserves both nonnegativity and the
collision invariants.

\subsection{Approximate models for the fluctuating Boltzmann equation}

We next consider two simpler models.  They retain different parts of the
fluctuating Boltzmann collision structure.

\medskip
\noindent\textbf{The Gaussian fluctuating Boltzmann equation.}
Replace the compensated Poisson random measure in the fluctuating
Boltzmann equation by a Gaussian white noise
\(\xi(\d t,\d c,\d u)\) on
\([0,\infty)\times\mathscr C_{\mathrm B}\times[0,\infty)\).  This gives
the formal equation
\begin{align}\label{Gaussian-fluctuating-Boltzmann-marked}
 \d F_t^\varepsilon+v\cdot\nabla_xF_t^\varepsilon\,\d t
 ={}&Q_{\mathrm B}(F_t^\varepsilon,F_t^\varepsilon)\,\d t\notag\\
 &+\varepsilon\int_{\mathscr C_{\mathrm B}\times[0,\infty)}
 \mathcal J_c
 \boldsymbol1_{\{u\leq(2\varepsilon)^{-1}
 B(v-v_*,\sigma)F_t^\varepsilon(x,v)
 F_t^\varepsilon(x,v_*)\}}
 \xi(\d t,\d c,\d u).
\end{align}
We notice that the mark variable can be integrated out.  If
\(\bar\xi(\d t,\d c)\) is a Gaussian white noise on
\([0,\infty)\times\mathscr C_{\mathrm B}\), then the same noise can be
written, in law, as
\begin{align}\label{Gaussian-fluctuating-Boltzmann-unmarked}
 \d F_t^\varepsilon+v\cdot\nabla_xF_t^\varepsilon\,\d t
 ={}&Q_{\mathrm B}(F_t^\varepsilon,F_t^\varepsilon)\,\d t\notag\\
 &+\sqrt{\frac{\varepsilon}{2}}
 \int_{\mathscr C_{\mathrm B}}
 \mathcal J_c
 \sqrt{B(v-v_*,\sigma)F_t^\varepsilon(x,v)
 F_t^\varepsilon(x,v_*)}\,
 \bar\xi(\d t,\d c).
\end{align}
To verify the equivalence, we test the two noises against state-space test
functions \(p\) and \(q\).  Their instantaneous cross covariance is, in
both cases,
\begin{align*}
 \frac{\varepsilon}{2}
 \int_{\mathscr C_{\mathrm B}}
 B(v-v_*,\sigma)F_t^\varepsilon(x,v)F_t^\varepsilon(x,v_*)
 \mathscr D_{\mathrm B}p(c)\mathscr D_{\mathrm B}q(c)\,\d c\,\d t, 
\end{align*}
as the integration with respect to \(u\) produces the factor
\((2\varepsilon)^{-1}BF_t^\varepsilon F_{t,*}^\varepsilon\).  Both noises
are centered and Gaussian, and therefore the equality of the covariance
gives the equality of their cylindrical laws.  This second-order
approximation retains the drift and covariance, and hence the central-limit
fluctuations of the underlying hard-sphere system.  It discards the higher
cumulants and the corresponding nonquadratic large-deviation cost.  After
regularizing the coefficient and the noise, the methods of
\cite{DuongHeWuLandau} may provide a method toward the construction of
nonnegative function-valued solutions.

\medskip
\noindent\textbf{The kinetic \(\alpha\)-stable Dean--Kawasaki equation.}

For the second approximation, we replace the nonlinear Boltzmann collision
operator by the fractional velocity generator \(\mathcal L_\alpha\).  The
model derived in Section~\ref{sec:Psi-structure-examples} is
\begin{align*}
 \d F_t^\varepsilon
 +v\cdot\nabla_xF_t^\varepsilon\,\d t
 ={}&\mathcal L_\alpha^*F_t^\varepsilon\,\d t\\
 &+\varepsilon
 \int_{\R_x^d\times\R_v^d\times\R_*^d\times(0,\infty)}
 \bigl(\delta_{(x,v+z)}-\delta_{(x,v)}\bigr)
 \boldsymbol1_{\{u\leq F_{t-}^\varepsilon(x,v)/\varepsilon\}}
 \widetilde{\mathfrak N}(\d t,\d x,\d v,\d z,\d u).
\end{align*}
This equation has the same positivity obstruction as the fluctuating
Boltzmann equation: a jump from a diffuse density subtracts an atom at its
departure point.  Its linear collision structure is directly related to
the empirical measures of kinetic particles whose velocities are driven by
independent stable processes.
The martingale formulations and particle characterizations developed for
the classical Dean--Kawasaki equation
\cite{KonarovskyiLehmannVonRenesse2019,
KonarovskyiLehmannVonRenesse2020,MullerVonRenesseZimmerVFP}
suggest to consider measure-valued martingale solutions.  Their defining
characteristics should retain the full jump current, or equivalently its
higher cumulants, and the solutions may be constructed from empirical
measures of the kinetic \(\alpha\)-stable particles.  Nonnegativity is then
preserved at the particle level.

\subsection{Stochastic hydrodynamic limits}

The hydrodynamic limit studied in Section~\ref{sec:Boltzmann-LLNS}
reconstructs the velocity equation through the momentum mode and its stress
corrector, but does not recover the density and temperature equations.  To
derive the full fluctuating Navier--Stokes--Fourier system studied in
\cite{GessSauerbreyWuNSF}, one also needs kinetic observables for the mass
and energy modes and the heat-flux cell problem mentioned in
Remark~\ref{rem:LLNS-without-temperature}.  The Hamiltonian should recover
the density--temperature constraints, Fourier heat conduction, the random
heat flux, and its joint covariance with the random stress.

These limiting characteristics can be identified at the Hamiltonian level
before a complete SPDE theory is available.  Once the solution theories in
Step~2 are obtained, we can study process-level convergence by tightness and
the limiting martingale problem.  This is the long-term problem in Step~4.

\subsection{Relationships between different models}

Several scaling limits connect deterministic kinetic equations.  One
example is the grazing-collision limit from Boltzmann to Landau
\cite{DesvillettesGrazing1992,DuongHeGENERIC2026}.  Their fluctuating
counterparts require in addition to identify the limiting current.  In the
grazing regime, the Poisson jumps become more frequent and their sizes
decrease.  They should converge to a Gaussian current whose covariance is
the nonlinear, nonlocal Landau covariance.

We expect to study these limits through their Hamiltonians.  In
Section~\ref{sec:Boltzmann-LLNS}, the expansion
\eqref{exponential-Taylor-domination} leads to
\eqref{quadratic-collision-Hamiltonian-limit}.  The second cumulant survives
and all higher cumulants vanish, which produces the Gaussian LLNS noise; see
Remark~\ref{rem:cumulant-scaling}.  The fluctuating Boltzmann-to-Landau limit
should follow the same mechanism.  It has to recover a state-dependent
multiplicative covariance and, at the same time, control the angular
concentration and its relation to the fluctuation scale.  Once the solution
theory is available, the Hamiltonian limits may be strengthened to
process-level convergence.

\subsection{The nonlinear fluctuation--dissipation relation}

We proposed the nonlinear fluctuation--dissipation relation
\eqref{intro-centered-cumulant}.  In
Section~\ref{sec:rate-identification}, we test this relation by the
large-deviation rates.  On regular tilted paths,
Theorem~\ref{thm:compatible-tilt-rate-equality} identifies the
Freidlin--Wentzell-type action of the master equation with the canonical
action of the generalized gradient structure.  Hence the two descriptions
assign the same thermodynamic cost to every admissible path.

A stronger justification may be obtained from a microscopic derivation.  It
requires to consider systems satisfying a suitable detailed-balance
condition.  The question is whether, under kinetic scaling, the Hamiltonian
of the empirical current converges to the force-shifted structure determined
by the same dual potential as the limiting deterministic current.  Such a
derivation would determine whether \eqref{intro-centered-cumulant} is
selected by microscopic systems.

\subsection{Conclusion} 

Fluctuating kinetic theory lies between deterministic kinetic equations and
microscopic particle dynamics.  The kinetic state remains a continuum
field, while the empirical diffusion and collision currents retain their
fluctuations.  In this regime, one may describe non-Gaussian collision
statistics, state-dependent random currents, and their propagation to
hydrodynamic scales.  The mathematical problems combine kinetic transport,
singular or nonlocal coefficients, conservative Gaussian or jump noises,
positivity constraints, and interacting scales.  The master equation gives
a continuum formulation in which these problems can be studied.

\subsection*{Acknowledgements}
The author is grateful to Zihui He and Johannes Zimmer for helpful
discussions and comments.

\subsection*{Declaration of AI-Assisted Technologies}

The author acknowledges the use of OpenAI's ChatGPT 5.6 Pro for literature
collection and for improving the language and presentation of the
manuscript.  The principal ideas, mathematical arguments, and proof
strategies were developed and independently verified by the author.  Some
computations were carried out with AI assistance.  In particular,
AI-assisted exploration suggested the Hamiltonian viewpoint adopted in
Section~\ref{sec:Boltzmann-LLNS} and brought the related work of Feng,
Mikami, and Zimmer \cite{FengMikamiZimmer2021} to the author's attention. Figure~1 was also prepared using AI tools under the author's direction.

\printbibliography

\end{document}